\documentclass[11pt]{amsart} 
\usepackage{fullpage, amsmath, amssymb, young}
\usepackage[all]{xypic}

\title{On the asymptotic $S_{n}$-structure of invariant 
differential operators on symplectic manifolds} 
\author{Qingchun Ren and Travis  Schedler} 
\usepackage{url} \usepackage{graphicx}

\numberwithin{equation}{section}
\numberwithin{equation}{subsection}

\theoremstyle{definition}
\newtheorem{theorem}[equation]{Theorem}
\newtheorem{lemma}[equation]{Lemma}
\newtheorem{proposition}[equation]{Proposition}
\newtheorem{corollary}[equation]{Corollary}

\newtheorem{definition}[equation]{Definition}

\newtheorem{example}[equation]{Example}
\newtheorem{question}[equation]{Question}

\newtheorem{remark}[equation]{Remark}

\newtheorem{claim}[equation]{Claim}

\newcommand\Tline{\rule{0pt}{2.6ex}}
\newcommand\Bline{\rule[-1.2ex]{0pt}{0pt}}

\newcommand{\tr}{\operatorname{tr}}
\newcommand{\Span}{\operatorname{Span}}
\newcommand{\Alt}{\operatorname{Alt}}
\newcommand{\pd}[2]{\frac{\partial #1}{\partial #2}}
\newcommand{\Id}{\operatorname{Id}}
\newcommand{\sgn}{\mathrm{sgn}}
\newcommand{\ad}{\operatorname{ad}}
\newcommand{\h}{\mathfrak{h}}

\newcommand{\Inv}{\operatorname{Inv}}
\newcommand{\Quant}{\operatorname{Quant}}
\newcommand{\Ass}{\operatorname{Ass}}
\newcommand{\SC}{\operatorname{SC}}
\newcommand{\SCLie}{\operatorname{SCLie}}
\newcommand{\Lie}{\operatorname{Lie}}

\newcommand{\HH}{\mathsf{HH}}

\newcommand{\HP}{\mathsf{HP}}

\newcommand{\ZZ}{\mathbb{Z}}

\newcommand{\C}{\mathbb{C}}

\newcommand{\cO}{\mathcal{O}}
\renewcommand{\cD}{\mathcal{D}}

\newcommand{\hh}{\mathfrak{h}}

\newcommand{\Hom}{\text{Hom}}

\newcommand{\Ind}{\operatorname{Ind}}
\newcommand{\Res}{\operatorname{Res}}
\newcommand{\Rep}{\operatorname{Rep}}

\newcommand{\pr}{\operatorname{pr}}

\newcommand{\Sp}{\mathsf{Sp}}
\renewcommand{\sp}{\mathfrak{sp}}
\newcommand{\gr}{\operatorname{\mathsf{gr}}}

\newcommand{\Weyl}{\mathsf{Weyl}}
\newcommand{\onto}{\twoheadrightarrow}
\newcommand{\into}{\hookrightarrow}
\newcommand{\Sym}{\operatorname{\mathsf{Sym}}}

\newcommand{\Har}{\operatorname{Har}}

\newcommand{\CC}{{\mathbb C}}
\begin{document}
\date{May, 2010}
\begin{abstract}We consider the space of polydifferential operators on
  $n$ functions on symplectic manifolds invariant under symplectic
  automorphisms, whose study was initiated by Mathieu in 1995.
  Permutations of inputs yield an action of $S_n$, which extends to an
  action of $S_{n+1}$.  We study this structure viewing $n$ as a
  parameter, in the sense of Deligne's category.  For manifolds of
  dimension $2d$, we show that the isotypic part of this space of
  $\leq 2d+1$-th tensor powers of the reflection representation
  $\hh=\C^n$ of $S_{n+1}$ is spanned by Poisson polynomials.  We also
  prove a partial converse, and compute explicitly the isotypic part
  of $\leq 4$-th tensor powers of the reflection representation. 

  We give generating functions for the isotypic parts corresponding to
  Young diagrams which only differ in the length of the top row, and
  prove that they are rational fractions whose denominators are
  related to hook lengths of the diagrams obtained by removing the top
  row. This also gives such a formula for the same isotypic parts of
  induced representations from $\ZZ/(n+1)$ to $S_{n+1}$ where $n$ is
  viewed as a parameter.

  We show that the space of invariant operators of order $2m$ has
  polynomial dimension in $n$ of degree equal to $2m$, while the part
  not coming from Poisson polynomials has polynomial dimension of
  degree $\leq 2m-3$. We use this to compute asymptotics of the
  dimension of invariant operators. We also give new bounds on the
  order of invariant operators for a fixed $n$.

  We apply this to the Poisson and Hochschild homology associated to the
  singularity $\CC^{2dn}/S_{n+1}$.
   Namely, the canonical surjection from $\HP_0(\cO_{\CC^{2dn}/S_{n+1}}, \cO_{\CC^{2dn}})$ to $\gr \HH_0(\Weyl(\CC^{2dn})^{S_{n+1}}, \Weyl(\CC^{2dn}))$ (the Brylinski
  spectral sequence in degree zero)
   restricts to an isomorphism in the aforementioned isotypic part
   $\hh^{\otimes \leq 2d+1}$, and also in $\hh^{\otimes \leq 4}$. We
   prove a partial converse.  Finally, the kernel of the entire
   surjection has dimension on the order of $\frac{1}{n^3}$ times the
   dimension of the homology group.
\end{abstract}
\maketitle
\tableofcontents

\section{Introduction}
We will study a certain space of polydifferential operators on symplectic
manifolds considered by Mathieu \cite{Matso}. Equivalently, as we will
explain, these are certain Poisson homology groups associated to
quotient singularities corresponding to Weyl groups of type $A_n$.

We first recall the relevant definitions and results of \cite{Matso}.
Given a symplectic manifold $Y$, one may consider polydifferential
operators which are linear in $n$ complex-valued functions $f_1,
\ldots, f_n \in \cO_Y$.  That is, we are interested in linear maps
$\cO_Y^{\otimes n} \rightarrow \cO_Y$ which are linear combinations
of operators of the form
\begin{equation}
f_1 \otimes \cdots \otimes f_n \mapsto (D_1 f_1) \cdots (D_n f_n),
\end{equation}
for differential operators $D_1, \ldots, D_n$.  

In \cite{Matso}, Mathieu considered polydifferential operators of the
above form which are invariant under symplectic automorphisms, which
he called $\Inv_n(Y)$.  He proved that the resulting space is
finite-dimensional, and only depends on the dimension of $Y$, and not
on $Y$ itself.  Thus, we may take $Y = V$ to be a symplectic vector
space of dimension $2d$. 

Note that $\Inv_n(V)$ has a natural action of $S_n$ by permuting the
functions $f_1, \ldots, f_n$. Mathieu observed that this extends to an
action of $S_{n+1}$ if one considers the polydifferential operator to
be a distribution on $n+1$ functions.  Moreover, one has additional
algebraic structure: the compositions $\Inv_n(V) \otimes \Inv_{n'}(V)
\rightarrow \Inv_{n+n'-1}(V)$.  Formally, as Mathieu observed, this
endows the direct sum $\Inv(V) = \bigoplus_n \Inv_n(V)$ with the
structure of a cyclic operad (here, ``cyclic'' incorporates the
$S_{n+1}$-structure; using only the $S_n$-structure, one has an
ordinary operad).
However, we will not essentially need operads, and the reader can
ignore remarks we make about the operadic structure.

We prove that this naturally identifies with a certain space of
Poisson traces, namely functionals $\cO_{V^n} \rightarrow \CC$ which
annihilate all Poisson brackets with $S_{n+1}$-invariant functions on
$V^n$. This is dual to the zeroth Poisson homology group
$\HP_0(\cO_{V^n}^{S_{n+1}}, \cO_{V^n}) = \cO_{V^n} /
\{\cO_{V^n}^{S_{n+1}}, \cO_{V^n}\}$.  In the case $\dim V = 2$, this
is a homology group associated to the singularity $(\hh \oplus
\hh^*)/S_{n+1}$, where $S_{n+1}$ is viewed as the Weyl group of type
$A_n$ and $\hh$ is its reflection representation.  Hence, the study of
Mathieu's symplectic operad $\Inv_n(V)$ yields information about the
Poisson geometry of quotient singularities associated to type $A_n$
Weyl groups. This observation was the original motivation for this
work.

When $\dim V \geq n$, Mathieu showed
that $\Inv_n(V) \cong \CC[S_n]$ as a representation of $S_n$. Moreover,
$\Inv_n(V) \cong \Ind_{\ZZ/(n+1)}^{S_{n+1}} \CC$ as a representation
of $S_{n+1}$.  Hence, the most interesting phenomena take place when
$\dim V$ is small, and in fact we are especially interested in the case when
$\dim V = 2$.

Our goal is to study, for fixed $V$, the asymptotic $\Rep
S_{n+1}$-structure of $\Inv_n(V)$, using the construction of
Deligne \cite{Delst}. As we will see, this becomes quite complicated
and interesting.

As defined by Mathieu, let $\SC_n(V) \subset \Inv_n(V)$ denote the
subspace of operators expressible in terms of Poisson polynomials
(having total degree $n$). Let $P_n$ be the space of \emph{abstract}
Poisson polynomials of degree $n$; one has a canonical surjection $P_n
\onto \SC_n(V)$.  As Mathieu explained \cite[Theorem 7.5]{Matso}, this
surjection is an isomorphism if and only if $n \leq \dim V + 1$. He
also proved that, when $n \leq \dim V$, one has $\SC_n(V) =
\Inv_n(V)$. 

Our first main result, Theorem \ref{scinvisotthm}, is a generalization
of this that replaces degree by $S_{n+1}$-structure.  Specifically, we
prove that the $\hh^{\otimes \leq \dim V + 1}$-isotypic parts of
$\Inv_n(V), \SC_n(V)$, and $P_n$ all coincide. This implies, in
particular, that $\Inv_n(V) = \SC_n(V) = P_n$ if and only if $n \leq
\dim V + 1$, which already strengthens Mathieu's aforementioned
result.

As a consequence, the invariant part $\Inv_n(V)^{S_n}$ is
one-dimensional, spanned by the total multiplication (this was first
observed in \cite{hp0weyl} and motivated this work); in other words,
the $\hh^{\otimes \leq 1}$-isotypic part is one-dimensional (in
particular, there is no $\hh$-isotypic part).  More generally, we
explicitly compute the $S_{n+1}$ structure for the case $\hh^{\otimes
  \leq 3}$ (which is actually subsumed into the study of
abstract Poisson polynomials; see Remark \ref{poissoprem}).

Mathieu explained that invariance under symplectic automorphisms is
equivalent to invariance under Hamiltonian flows. Such operators must
have constant coefficients, i.e., we can assume that the $D_i$ above
are polynomials in the partial derivatives with complex
coefficients. As a result, $\Inv_n(V)$ is graded by the total order of
polydifferential operators. Since these operators are in particular
invariant under $-\Id \in \Sp(V)$, they have even order.  Denote the
subspace of order-$2m$ operators by $\Inv_n(V)_{2m}$.\footnote{Our
  notation differs from that of \cite{Matso}, since we will use $2m$
  to refer to the order of polydifferential operators, and $2d$ to
  refer to the dimension of $V$, whereas Mathieu used $2m$ to refer to
  the dimension of $V$, and $2k$ to refer to the order of
  polydifferential operators, which was also called there the
  Liouville degree.}

Our identification between invariant polydifferential operators and
(abstract) Poisson polynomials on the $\hh^{\otimes \leq \dim V + 1}$-isotypic
part is one of graded $S_{n+1}$-representations,
where the order of a Poisson monomial is its number
of brackets. Using this, we explicitly compute the Hilbert
series of the $\hh^{\otimes \leq 3}$-isotypic parts of $\Inv(V)$
(which is independent of $V$).

We package this information conveniently into generating functions in
two formal parameters: one for order, and one for degree.
Specifically, for each partition $\lambda$, the generating function
$I_{\lambda,V}^+(s,t)$ records, as the coefficient of $s^m t^n$, the
multiplicity in $\Inv_n(V)_{2m}$ of the irreducible representation of
$S_{n+1}$ corresponding to the partition $\lambda[n+1]$ which is
obtained from $\lambda$ by adding a new top row of length
$n+1-|\lambda|$. We similarly define the generating function
$I_{\lambda,V}(s,t)$ using only the $S_n$ structure: the coefficient
of $s^m t^n$ now records the multiplicity of $\lambda[n]$ in
$\Inv_n(V)_{2m}$, as a representation of $S_n$.

Our second main result, Theorem \ref{hlthm}, then proves that
$I_{\lambda,V}^+(s,t)$ and $I_{\lambda,V}(s,t)$ are rational functions
in $s$ and $t$ with a specific denominator, which is $(1-t)$ times the
product over hook lengths $h$ of $\lambda$ of $(1-s^ht^h)$ (in the
case of $I_{\lambda,V}^+(s,t)$, we can omit one of the factors of
$1-st$, corresponding to a hook length of one).  Moreover, we explain
that the numerator essentially has degree less than that of the
denominator when $\dim V \geq |\lambda|-1$, and compute in this case
the value of the numerator at $s=t=1$ (it is $(|\lambda|-1)!$ and
$|\lambda|!$ in the cases of $I_{\lambda,V}^+(s,t)$ and
$I_{\lambda,V}(s,t)$, respectively).  We compute these functions for
partitions $\lambda$ of size $\leq 4$ (using Theorem \ref{ht4thm} in
the case of $|\lambda|=4$), and note that so far it turns out that the
evaluation at $t=1$ of the numerator has nonnegative coefficients
(Question \ref{nonnegques}).  Finally, as an application of this, we
obtain an apparently new formula for the isotypic decomposition of
induced representations from $\ZZ/(n+1)$ to $S_{n+1}$ where $n$ is
viewed as a parameter (Claim \ref{combclaim}).

In the process, we prove a more general result (Theorem \ref{polythm})
which describes, for each $k \geq 0$, a structure of module over the
polynomial ring generated by $k$ variables on the subspace
$\Inv(V)^{(k)} \subseteq \Inv(V)$ spanned by the isotypic part in
$\Inv_n(V)$ of irreducible representations of $S_{n+1}$ whose Young
diagram has at most $k$ cells below the top row.  

Next, we consider, for each fixed $m$ and $V$, the asymptotic $\Rep
S_{n+1}$ structure of $\Inv_n(V)_{2m}$ as $n \rightarrow \infty$.  We
prove that $\dim \Inv_n(V)_{2m}$ and $\dim \SC_n(V)_{2m}$ are
polynomials of degree $2m$.  Our next main result, Theorem
\ref{scinvordthm} (and its sharpened form, Theorem
\ref{scinvordconvthm}), states that the quotient
$\Inv_n(V)_{2m}/\SC_n(V)_{2m}$ is a polynomial of degree only $\leq
2m-3$, and it is exactly $2m-3$ when $n \geq \dim V + 6$. For $n \leq
\dim V + 4$, we prove that $\Inv_n(V)_{2m} = \SC_n(V)_{2m}$, and for
$n \leq \dim V$, we prove moreover that these are isomorphic to
$(P_n)_{2m}$, the subspace of $P_n$ spanned by Poisson monomials with
precisely $m$ pairs of brackets.

We refine these results in terms of the $S_{n+1}$-representation type,
and this is a necessary tool in their proof.  For example, to prove
that $\Inv_n(V)_{2m}$ has polynomial dimension in $n$, we actually
prove it is isomorphic to the direct sum of finitely many induced
representations $\Ind_{S_{k_i} \times S_{n-k_i}}^{S_n} (\rho_{i}
\boxtimes \CC)$, where the $k_i$ and $\rho_i \in \Rep S_{k_i}$ are
independent of $n$ (although, as $S_{n+1}$-representations, the
situation is somewhat more complicated).

We also consider the subspace $\Quant_n(V)$ (``quantum'') defined by
Mathieu, which lies between $\SC_n(V)$ (``semiclassical'') and
$\Inv_n(V)$: roughly, $\Quant_n(V)$ consists of those operators
obtained by star-product quantization of $V$ (see \S \ref{scquantsec}
below for details), or obtainable as the associated graded operator of
one which sends $n$ elements of the Weyl algebra $\Weyl(V)$ to a
certain linear combination of products of them in various orderings
(recall here that $\gr \Weyl(V) = \cO_V$, if we take $\cO_V$ to be the
algebra of polynomial functions on $V$). We prove that $\dim
\Quant_n(V)_{2m}$ is also a polynomial of degree $2m$, valid for all
$n$, and hence our theorem implies that its leading three coefficients
coincide with those of $\Inv_n(V)_{2m}$ and $\SC_n(V)_{2m}$.

We identify $\Quant_n(V)$ with the associated graded vector space
of the Hochschild
trace space 
\begin{equation}
\HH_0(\cD_{U^n}^{S_{n+1}}, \cD_U)^* := (\cD_U /
[\cD_U^{S_{n+1}}, \cD_U])^*, \quad V = U \oplus U^*,
\end{equation}
where $\cD_{U^n}$ is the algebra of differential operators on
$U^n$. Alternatively, $\cD_{U^n} = \Weyl(V)^{\otimes n}$, the $n$-th
tensor power of the Weyl algebra associated to $V$.  As we explain,
the direct sum of these Hochschild trace groups over all $n$ naturally
identifies with the associative operad as a cyclic operad, equipped
with a certain filtration which depends on $V$. In particular,
$\Quant_n(V)$ is the regular representation of $S_n$ (as observed by
Mathieu), and $\Ind_{\ZZ/(n+1)}^{S_{n+1}} \CC$ as a representation of
$S_{n+1}$.

The study of the inclusion $\Quant_n(V) \subseteq \Inv_n(V)$ is
equivalent to the study of the canonical surjection
$\HP_0(\cO_{V^n}^{S_{n+1}}, \cO_{V^n}) \onto \gr
\HH_0(\cD_{U^n}^{S_{n+1}}, \cD_U)$, which is the degree-zero part of
the Brylinski spectral sequence relating $\HP_*$ to $\gr \HH_*$.  This
was one of our motivations for revisiting Mathieu's operads. In
particular, this surjection is an isomorphism if and only if
$\Quant_n(V) = \Inv_n(V)$.  Our aforementioned results then imply that
the $\hh^{\otimes \leq \dim V + 1}$-isotypic part of this sequence
degenerates, as well as the part of order $\leq \dim V + 4$. More
generally, in order $2m \geq \dim V + 6$, our results show that the
kernel of the canonical surjection from Poisson to the associated
graded of Hochschild homology has polynomial dimension of degree $\leq
2m-3$.  We also show that this spectral sequence does \emph{not}
degenerate in the isotypic parts of representations of $S_n$ occurring
in $\hh^{\otimes \geq \dim V + 3}$ but not in $\hh^{\otimes \leq \dim
  V + 2}$, for sufficiently large $n$ (in the sense of \cite{Delst} as
we will explain).

In the limit $\dim V \rightarrow \infty$, Mathieu observed that
$\SC_n(V) = \Quant_n(V) = \Inv_n(V)$ are all the regular
representation of $S_n$.  By the above, they limit to the associated
graded of the associative operad with respect to the PBW
filtration.\footnote{This filtration places linear combinations of
  abstract multiplications $(f_1 \otimes \cdots \otimes f_n) \mapsto
  \sum_{\sigma \in S_n} \lambda_\sigma f_{\sigma(1)} \cdots
  f_{\sigma(n)}$ in order $\leq 2m$ written as a linear combination of
  products of commutators, each of which has at most $m$ commutator
  brackets.  The name ``PBW'' comes from the fact that this induces
  the PBW filtration on the free associative algebra on $V$, viewed as
  the universal enveloping algebra of the free Lie algebra on $V$.}
However, for fixed $V$, Mathieu observed that $\Inv(V)$ is not
finitely generated as an operad, and the same is true for
$\Quant(V)$. In fact, the top order of $\Quant(V)$ (and hence also of
$\Inv(V)$) is at least $C \cdot n^{1+\frac{1}{\dim V}}$ for some
constant $C > 1$ (which can be explicitly computed): see Remark
\ref{quanttoporderrem} below.  To be finitely generated, the growth
would have to be linear in $n$.  That $\Quant(V)$ is
infinitely-generated is especially interesting since it is obtainable
from the associative operad by taking an associated graded (with
respect to a filtration depending on $V$). It would be interesting to
study its algebraic structure in more detail.

We would like to emphasize that the main ingredients of this work that
set it apart from Mathieu's excellent paper, are to study the
asymptotic $\Rep S_{n+1}$-structure in the sense of Deligne
\cite{Delst} (whose work didn't yet exist when Mathieu wrote his
paper), and also to study the space of invariant polydifferential operators of a
fixed order, which Mathieu didn't consider too heavily.

Note that one has a choice whether to take $V$ to be a real vector
space and let $\cO_V$ denote the algebra of smooth functions on $V$, or
to take $V$ to be a complex vector space and let $\cO_V$ denote the
algebra of holomorphic functions on $V$, or the algebra of formal power
series at $0$; it is easy to see that the result does not depend on
this choice (and neither do any of the arguments we use).  For
definiteness, we will let $V$ be a complex symplectic vector space and
let $\cO_V$ denote the space of formal power series.  The term
``symplectic automorphism'' will then mean formal complex symplectic
automorphism.

\subsection{Acknowledgements} This work grew out of an MIT UROP, and
the authors are grateful to MIT for its support. It also grew out of a
collaboration with Pavel Etingof on the study of Poisson traces, and
the authors are very grateful to Etingof for his invaluable
assistance, in particular for first discovering Theorem
\ref{isotpartthm}.(i)--(iii) and its proof at the beginning stages of
this work. We also thank him for pointing out the correspondence
between Mathieu's operad and certain Poisson and Hochschild homology
groups (Theorems \ref{invinterpthm} and \ref{quantinterpthm}), which
motivated this work and allowed us to apply computational techniques
developed for these homology groups.  We are grateful to R. Stanley
for suggesting hook length formulas and the combinatorial proof
found in Appendix \ref{combapp}, and to V. Dotsenko for providing the
character of the Poisson operad in terms of symmetric functions found
in Remark \ref{poissoprem}, and thank them for useful discussions.
The second author is a five-year fellow of the American Institute of
Mathematics, and was partially supported by the ARRA-funded NSF grant
DMS-0900233.

\subsection{A guide to the reader}
This paper is organized by stating and explaining all of the important
results in \S 1 (the most important ones appear in \S \ref{mrsec}),
and then proving the results from each subsection in a subsequent
sections.  As a result, the reader who wishes to read the proofs will
need to refer to the later section; however the statements of the
results already take enough space that we felt we wanted to make it
possible to understand the statements of the results for the reader
without time to study their proofs.  

In \S \ref{s:snp1str}, we introduce the idea (dating to \cite{Delst})
of studying a sequence of $S_{n+1}$-representations as $n$ varies by
finding corresponding isotypic parts which make sense for all $n$.
For example, the simplest such isotypic part is that of the trivial
representation, which makes sense for all $n$; the next simplest is
that of the ($n$-dimensional) reflection representation (viewing
$S_{n+1}$ as a type $A_n$ Weyl group).  In this section we compute the
simplest isotypic parts of the main object of study, $\Inv(V)$, the
polydifferential operators on a symplectic vector space invariant
under symplectic automorphisms: namely we compute this for
\emph{heights $\leq 2$}: the height, as we explain, is the number of
boxes in the Young diagram below the top row, which is also the degree
of the polynomial in $n$ which gives the dimensions of the
corresponding irreducible $S_{n+1}$-representations for each fixed
arrangement of boxes below the top row.  The proofs of results from
this section are given in \S \ref{s:isopf}.

In \S \ref{mrsec}, we state the main results on the structure of
$\Inv(V)$, which gives many cases in which these invariant
differential operators are all obtained from Poisson brackets or
polynomials. We also give generating functions, which are rational
fractions, for the isotypic parts of $\Inv(V)$ in the aforementioned
sense, and prove that their denominators are given by hook lengths.
Here we also compute the isotypic parts of $\Inv(V)$ for heights $\leq
4$, extending the results of the previous subsection (for heights
$\leq 2$). These results are proved in \S \ref{s:mrsecpf}.

In \S \ref{lowdimsec}, we give computational results for $n \leq 6$,
using programs in Magma \cite{magma}. These require no proof, since we
refer only to the computer programs used.

In \S \ref{scquantsec}, we identify 
$\Inv_n(V)$ with the space of Poisson traces for the type $A_{n+1}$
Weyl group singularity (i.e., the quotient $T^* \CC^n / S_{n+1}$ where
$\CC^n$ is the reflection representation of $S_{n+1}$ and $T^* \CC^n$
is the total space of its cotangent bundle, which is a symplectic
vector space).  We similarly identify the subspace $\Quant_n(V)
\subseteq \Inv_n(V)$ of polydifferential operators obtainable from
star products with the Hochschild trace space of the algebra of
differential operators on $\CC^n$ invariant under the action of
$S_{n+1}$.  
The proofs of these results are given in \S
\ref{s:scquantsecpf}.
 
In \S \ref{bicondsec}, we give explicit examples of elements of
$\Inv(V)$ which yield partial converses to some of our main results
from \S \ref{mrsec}. In particular, we show that for parameters
outside those where we prove in \S \ref{mrsec} that $\Inv(V)$ is
generated by Poisson brackets or by star-product quantizations, or
when the latter two subspaces coincide, these coincidences no longer
hold.  We also give an example which gives a lower bound on the top
order of $\Inv_n(\CC)$ in terms of $n$.  The proofs of results in this
section are given in \S \ref{s:bicondpf}.

Finally, in \S \ref{asympsec}, we prove general results which explain
precisely how $\Inv(V)$ asymptotically stabilizes as the degree of
polydifferential operators goes to infinity for fixed order of
differential operators, or as both the degree and the order go to
infinity.  We illustrate these results by giving the first few terms
of these asymptotic expansions, together with diagrams.  The proofs of
results in this section are given in \S \ref{s:asympsecpf}.

\subsection{The $S_{n+1}$-structure viewing $n$ is a parameter}\label{s:snp1str}
The starting observation for this paper is that, for fixed order $2m$,
following the idea of \cite{Delst}, one may consider the
$S_{n+1}$-structure of $\Inv_n(V)_{2m}$ where $n$ is viewed as a
parameter. To begin, let $\hh = \CC^n$ denote the reflection representation of $S_{n+1}$, viewed as a Weyl group. Our first result concerns the isotypic
part of $\hh^{\otimes \leq 2}$:
\begin{theorem} \label{isotpartthm}\footnote{We are grateful to Pavel Etingof for first observing and proving parts (i)--(iii).}
\begin{enumerate}
\item[(i)] \cite{hp0weyl} The invariant part $\Inv_n(V)^{S_{n+1}}$ has
  dimension one, occurring in order $0$ (i.e., constant multiples of
  the multiplication operator).
\item[(ii)] The isotypic part $\Hom_{S_{n+1}}(\hh, \Inv_n(V))$ is zero.
\item[(iii)] The isotypic part $\Hom_{S_{n+1}}(\wedge^2 \h,
  \Inv_n(V))$ has dimension $\lfloor
  \frac{n}{2} \rfloor$, and occurs in orders \\ $2, 6, 10, \ldots,
  2\bigl(2\lfloor \frac{n}{2} \rfloor - 1\bigr)$, each with
  multiplicity one.
\item[(iv)]For $n \geq 2$, $\Hom_{S_{n+1}}(\Sym^2 \h \ominus (\hh \oplus \CC), \Inv_n(V))$
 has dimension $\lfloor \frac{n-1}{2} \rfloor$, and occurs in orders $4, 8, \ldots, 4 \lfloor \frac{n-1}{2} \rfloor$, each with multiplicity one.
\end{enumerate}\end{theorem}
We will give the results for isotypic components of subrepresentations
of $\h^{\otimes 3}$ (as well as $\hh^{\otimes 4}$, albeit less
explicitly) in the next subsection, using some general results.  We
chose to state (and prove) the above theorem independently, since it
is simple and illustrative.

The polydifferential operators occurring in the above theorem are
special cases of the following important construction, which is also the essence of the proof.  For a Young
diagram $\lambda = (\lambda_1, \ldots, \lambda_k)$, let its
\emph{truncated} Young diagram be defined as $(\lambda_2, \ldots,
\lambda_k)$.
\begin{example} \label{fisxexam} (cf.~Proposition \ref{fisxprop})
  Suppose we are given an invariant polydifferential operator $\phi$
  which generates an irreducible representation of $S_{n+1}$
  corresponding to a certain Young diagram $\lambda = (\lambda_1,
  \ldots, \lambda_k)$. Up to permutations, we can assume that the
  first $\lambda_1$ inputs to $\phi$ correspond to the top row, so
  that the operator is determined by its restriction to elements of
  the form $f^{\otimes \lambda_1} \otimes g_{\lambda_1+1} \otimes
  \cdots \otimes g_n$.  Such an operator is determined by its action
  on elements $f$ with nonvanishing first derivative, and by Darboux's
  theorem, up to symplectic change of basis, we can assume $f=x$ is a
  linear function. This yields a new polydifferential operator $\psi:
  \cO_V^{n-\lambda_1} \rightarrow \cO_V$, which is invariant under all
  symplectic automorphisms fixing $x$, and which has the property that
  any automorphism sending $x$ to $c x$ must send $\psi$ to
  $c^{-\lambda_1} \psi$, for any nonzero $c \in \C$.  (This last
  condition, in the case $x,y$ is a basis of the symplectic plane
  $V=\CC^2$, says that the degree in $\frac{\partial}{\partial y}$
  minus the degree in $\frac{\partial}{\partial x}$ is $\lambda_1$.)
  Conversely, any such $\psi$ yields an invariant polydifferential
  operator $\phi$.

  In particular, in the case $V=\C^2$, we can take $\psi$ to be any
  linear combination of operators of the form $x^{j}
  \frac{\partial^{i_{\lambda_1+1}}}{\partial y^{i_{\lambda_1+1}}}
  \otimes \cdots \otimes \frac{\partial^{i_n}}{\partial y^{i_n}}$ with
  $i_{\lambda_1+1} + \cdots + i_n+ j = \lambda_1$, that is then
  Young-symmetrized according to the truncated diagram $(\lambda_2,
  \ldots, \lambda_k)$. (Note that this symmetrization can sometimes
  produce zero, since one of the cells of the truncated diagram
  corresponds to the output, i.e., we are symmetrizing $\phi$
  considered as a distribution on $\cO_{V}^{\otimes (n+1)}$).  This
  operator easily extends to a Poisson polynomial of $n$ inputs.  For
  arbitrary $V$, $\psi$ still makes sense if we extend $x$ to Darboux
  coordinates including $y$, with $\{x,y\}=1$.

  As an example, for $n \geq 2$ one has the operator $\phi(x^{\otimes
    (n-1)}\otimes f) = x^j \frac{\partial^{n-1-j}}{\partial y^{n-1-j}}
  (f)$, which in the case $n-1-j$ is odd generates the
  $S_{n+1}$-representation of $\wedge^2 \hh$ in the space of order
  $2(m-j)$-operators mentioned in Theorem \ref{isotpartthm}.(iii), and
  in the case $n-1-j$ is even, generates a copy of the representation
  $\Sym^2 \hh \ominus (\hh \oplus \CC$). These operators can be
  expressed as Poisson polynomials of two functions: $f \otimes g
  \mapsto (\ad f)^{n-1-j}(g)f^j$, of degree $n-1$ in the first input,
  $f$, and of degree $1$ in the second input.
\end{example}
Now, recall from \cite{Matso}, as also mentioned in the previous
section, that $\SC_n(V) \subseteq \Inv_n(V)$ denotes the subspace of
operators spanned by Poisson polynomials. We then deduce from the
theorem and the above example (which generates all of the
representations mentioned in the theorem, cf.~\S \ref{isotpartpfsec})
that $\Hom_{S_{n+1}}(\hh^{\otimes \leq 2}, \Inv_n(V)/\SC_n(V)) = 0$. We will
generalize this in the next section.
\subsection{Main results on the structure of $\Inv$} \label{mrsec} In
this section, we state our main results.  First we consider the
isotypic decomposition of $\Inv$. 
For us, the main measure of size (or complexity) of an $S_{n+1}$-representation
is its \emph{height}, which is defined so that the height $\leq k$ representations are those that occur in $\hh^{\otimes k}$. Precisely:
\begin{definition}
The \emph{height} of an irreducible representation of $S_n$ 
is the number of cells below the top row of the associated Young diagram,
i.e., the size of its truncation.
\end{definition}
We show that in height $\leq \dim V
+ 2$, $\Inv(V) = \SC(V)$, and compute $\Inv$ for general $V$ in
heights $\leq 4$. We construct an action on the height $\leq k$
part of $\Inv$ by the polynomial ring in $k$ variables which makes
the former into a finitely-generated module. This allows us to prove
that the generating function for the structure of $\Inv$ is a rational
function whose denominator has a hook length formula.  Finally, we end
by stating some results on the part of $\Inv_n$ in the top three
heights $2m-2, 2m-1$, and $2m$ (\S \ref{tophtsubsec}), which will be
explained in more detail in \S \ref{asympsec}. We use this to compute
the structure of $\Inv$ in orders $\leq 6$ (we will give the dimension
of the order $8$ part in \S \ref{asympsec}).

In particular, we find many cases where $\Inv_n(V) = \SC_n(V)$ (and
hence also they equal $\Quant_n(V)$, and the associated Brylinski
spectral sequence from Poisson to the associated graded of Hochschild
homology degenerates in degree zero; see the next section).  
\subsubsection{The equality $\Inv(V) = \SC(V)$ in heights $\leq \dim V + 1$}
Recall that $P_n$ denotes the space of abstract Poisson polynomials of
degree $n$.  This is graded by \emph{order}, defined as twice the
number of pairs of brackets that appear; let $(P_n)_{2m}$ be the part
of order $2m$.  One has a graded surjection $P_n \onto \SC_n(V)$,
which as Mathieu explained is an isomorphism if and only if $n \leq
\dim V + 1$. Mathieu also proved that $\SC_n(V) = \Inv_n(V)$ for $n
\leq \dim V$. Our first main result generalizes and strengthens this:
\begin{theorem}\label{scinvisotthm}
The height $\leq \dim V + 1$ isotypic parts of
$\Inv_n(V), \SC_n(V)$, and $P_n$ all coincide. 
\end{theorem}
We will prove a converse in \S \ref{scquantsec} below, which
roughly says that these are all different in greater heights.

The multiplicities of the irreducible representations of height $\leq
\dim V + 1$ can be determined because they coincide with their multiplicities
in $\Quant_n(V) = \Ind_{\ZZ/(n+1)}^{S_{n+1}}$ (see \S
\ref{scquantsec}).  However, it is much more difficult to determine
which orders they lie in, even though the question concerns abstract
Poisson polynomials.  Theorem \ref{isotpartthm} implies the answer for
height $\leq 2$.  We can extend this to height $3$:
\begin{definition} Given a Young diagram $\lambda = (\lambda_1,
  \ldots, \lambda_k)$ and any $n \geq |\lambda| + \lambda_1$, where
  $|\lambda| := \sum_i \lambda_i$, let $\lambda[n]$ denote the
  partition of $n$ with truncated Young diagram $\lambda$, and let
  $\rho_{\lambda}[n] := \rho_{\lambda[n]}$ denote the corresponding
  irreducible representation of $S_n$.
\end{definition}
  We will sometimes drop the bracket and use $\rho_\lambda$ to refer to either
  $\rho_\lambda[n]$ or $\rho_\lambda[n+1]$, depending on the context.
\begin{definition}
  Given an $S_{n+1}$ representation $R$ (e.g., $R = \Inv_n(V)_{2m}$ or
  $(P_n)_{2m}$), let $M_\lambda(R) := \Hom_{S_n}(\rho_\lambda[n], R)$
  and $M_\lambda^+(R) :=
  \Hom_{S_{n+1}}(\rho_{\lambda}[n+1],R)$. 
\end{definition}
\begin{corollary}\label{hot3cor} 
For $|\lambda|=3$, $M_\lambda^+(\Inv_n(V)_{2m})=0$ for 
$m \leq 1$ or $n \leq 3$. For 
$m \geq 2$ and $n \geq 4$: \\
\begin{tabular}{c|c|c|c}
\Bline & $M_{(3)}^+(\Inv_n(V)_{2m})$ & $M_{(1,1,1)}^+(\Inv_n(V)_{2m})$ & $M_{(2,1)}^+(\Inv_n(V)_{2m})$ \\[4pt] \hline
\Bline \Tline $m \leq n-2$ & $\lfloor \frac{m}{3} \rfloor$ & $\lfloor \frac{m}{3} \rfloor$ & $\lfloor \frac{2m-1}{3} \rfloor$ \\ \hline
\Bline \Tline $m = n-1$ & $\lfloor \frac{m-3}{6} \rfloor + \delta_{6 \mid m-1}$ &
$\lfloor \frac{m-3}{6} \rfloor + \delta_{6 \mid m-1} + \delta_{2 \mid m}$ & $\lfloor \frac{m-2}{3} \rfloor$ \\ \hline
\Bline \Tline $m \geq n$ & 0 & 0 & 0 
\end{tabular}
\end{corollary}
Here, $\delta_{a \mid b}$ is one if $a \mid b$ and zero otherwise. The above information, along with that of Theorem \ref{isotpartthm}, can be nicely encoded in the following generating functions.  If $\lambda$ is a Young
diagram, define
\begin{gather} \label{ildef}
I_{\lambda,V}(s,t) = \sum_{m \geq 0, n \geq |\lambda|+\lambda_1} \dim \Hom_{S_n}(\rho_{\lambda}[n], \Inv_n(V)_{2m}) s^m t^n, \\ 
\label{ilpdef} I_{\lambda,V}^+(s,t) = \sum_{m\geq 0, n+1 \geq |\lambda|+\lambda_1, n \geq 0} \dim \Hom_{S_{n+1}}(\rho_{\lambda}[n+1], \Inv_n(V)_{2m}) s^m t^n.
\end{gather}
Since the degree in $t$ will always be at least the degree in $s$, let us also make the substitution $u := st$ instead of $s$.
Then, we deduce, for all $V$, that
\begin{gather}\label{ifla}
I_{(),V} = \frac{1}{1-t}, \quad I_{(1),V} = \frac{ut}{(1-t)(1-u)}, \\
I_{(2),V} = \frac{2u^2t^2+u^4t-u^4t^2}{(1-t)(1-u)(1-u^2)}, \quad
I_{(1,1),V} = \frac{ut^2+u^3t}{(1-t)(1-u)(1-u^2)}, \\
\label{ipfla}
I_{(),V}^+ = \frac{1}{1-t}, \quad I_{(1),V}^+ = 0, \\ 
\label{i2pfla}
I_{(2),V}^+ = \frac{u^2t}{(1-t)(1-u^2)}, \quad I_{(1,1),V}^+ = \frac{ut}{(1-t)(1-u^2)}, \\
\label{i3pfla} 
I_{(3),V}^+ = \frac{u^3t^2+u^4t^2+u^7t-u^7t^2}{(1-t)(1-u^2)(1-u^3)}, \quad I_{(1,1,1),V}^+ = \frac{u^3t^2+u^4t}{(1-t)(1-u^2)(1-u^3)}, \\
\label{i21pfla} 
I_{(2,1),V}^+ = \frac{u^2t^2+u^4t^2+u^5t-u^5t^2}{(1-t)(1-u)(1-u^3)}.
\end{gather}
\subsubsection{Hook length formulas and module structures of $\Inv(V)$
over polynomial algebras}
More generally, we can prove the following, based on a suggestion of
R. Stanley. Given a Young diagram $\lambda$, let $h_i(\lambda), i = 1,
2, \ldots, |\lambda|$, be the hook lengths of $\lambda$, in some
ordering.  We recall their definition: let $\lambda'$ be the dual
partition to $\lambda$, obtained by swapping rows with columns; that
is, the length of the $i$-th row of $\lambda'$ equals the length of
the $i$-th column of $\lambda$, and vice-versa. To the cell in the
$i$-th row and $j$-th column of $\lambda$, we attach the hook length
$(\lambda_i-i) + (\lambda_j'-j) + 1$, which visually is the sum of the
distances to the boundary in the rightward and downward directions from
the center of the cell.

\begin{theorem} \label{hlthm}
\begin{enumerate}
\item[(i)]
Let $\lambda$ and $V$ be arbitrary. Then,
\begin{equation} 
I_{\lambda,V} = \frac{J_{\lambda,V}}{(1-t)\prod_{i=1}^{|\lambda|} (1 - u^{h_i(\lambda)})}, \quad I_{\lambda,V}^+ = \frac{J^+_{\lambda,V}\cdot (1-u)}{(1-t)\prod_{i=1}^{|\lambda|} (1 - u^{h_i(\lambda)})},
\end{equation}
where $J_{\lambda,V}$ and $J^+_{\lambda,V}$ are polynomials in $t$ and
$u := st$. Moreover, the same formulas hold (with different numerators) if
we replace $\Inv$ by $\Quant$ or $\SC$.

\item[(ii)] In the case $|\lambda| \leq \dim V + 1$ (i.e., when
  $\Inv_n(V) \cong P_n \cong \SC_n(V) \cong \Quant_n(V)$,
  independently of $V$), then the numerator can be taken to have
  degree in $t$ and $s$ less than the degrees in $t$ and $s$ of the
  denominator, up to adding to the LHS a polynomial in $t$ and $s$ of
  degree less than $|\lambda|+\lambda_1-1$.\footnote{Note that
    $I_{\lambda,V}$ does not contain any monomials whose degree in $t$
    (assigning $s$ degree zero, and hence $u$ degree one) is less than
    $|\lambda|+\lambda_1-1$, since $\lambda[n+1]$ does not make sense.
    Hence, the claim is that one can choose $J_{\lambda,V}$ of degree
    less than the denominator in $s$ and $t$ and obtain an equality
    after discarding such monomials that appear in the expansion of
    the RHS (cf.~the proof of the theorem in \S \ref{hlthmpfsec}
    below).} Moreover, in this case, $J_{\lambda,V}|_{s=t=1} =
  |\lambda|!$, and when $|\lambda| \geq 2$,
  $J^+_{\lambda,V}|_{s=t=1}=(|\lambda|-1)!$.
\end{enumerate}
\end{theorem}
In part (ii), to compute degrees in $s$ and $t$, one must replace $u$ by $st$ throughout.

Unfortunately, we were not able to prove a reasonable bound on the
degree of $J_{\lambda,V}$ and $J^+_{\lambda,V}$, or on its evaluation
at $s=t=1$ (in the case when $|\lambda| > \dim V + 1$).

The first observation about positivity is:
\begin{question}\label{sgnposques}
Are the coefficients of $J_{(1^k),V}(s,t)$ nonnegative, where $(1^k)$ denotes the diagram which is a single column of length $k$? By the preceding results and that of the next subsection, this is true for $k \leq 4$.
\end{question}

If we set $s = 1$, we obtain the following formula for the $S_{n+1}$ structure:
\begin{corollary} 
\begin{gather}
\sum_{n \geq |\lambda|+\lambda_1} \dim \Hom_{S_n}(\rho_{\lambda}[n], \Inv_n(V)) t^n = 
\frac{J_{\lambda,V}|_{s=1}}{(1-t)\prod_{i=1}^{|\lambda|} (1 - t^{h_i(\lambda)})}, \\ 
\sum_{m\geq 0, n+1 \geq |\lambda|+\lambda_1} \dim \Hom_{S_{n+1}}(\rho_{\lambda}[n+1], \Inv_n(V)_{2m})t^n = \frac{J^+_{\lambda,V}|_{s=1}}{\prod_{i=1}^{|\lambda|} (1 - t^{h_i(\lambda)})}.
\end{gather}
Moreover, the same formulas hold (with different numerators) if we replace $\Inv$ by $\Quant$ or $\SC$. 
\end{corollary}
In the case of $\Quant$, since $\Quant_n(V) \cong \Ind_{\ZZ/(n+1)}^{S_{n+1}}$ independently of $V$ (cf.~\S \ref{scquantsec} below), we deduce the following interesting combinatorial identity:
\begin{corollary}\label{combcor} For $|\lambda| \geq 2$,
\begin{equation}\label{combfla}
\sum_{n \geq |\lambda|+\lambda_1} \dim \Hom_{S_{n+1}}(\rho_{\lambda}[n+1], \Ind_{\ZZ/(n+1)}^{S_{n+1}} \CC) t^n = 
\frac{t^{|\lambda|+\lambda_1-1}K_{\lambda}(t)}{\prod_{i=1}^{|\lambda|} (1 - t^{h_i(\lambda)})},
\end{equation}
where $K_\lambda(t)$ is a polynomial of degree less than the degree of the denominator. Moreover, $K_\lambda(1) = (|\lambda|-1)!$.
\end{corollary}
With the help of R. Stanley, we found a purely combinatorial proof of
the above identity, which
also applies to induced representations of nontrivial characters of
$\ZZ/(n+1)$: see Appendix \ref{combapp}.

We can also deduce information about the spaces of Lie polynomials,
rather than Poisson polynomials. A Lie polynomial is a linear
combination of iterated brackets, e.g. $[a_1,[[a_2,a_3],[a_4,a_5]]]$,
and not a product of them as in Poisson polynomials. For each
symplectic vector space $V$, let $\SCLie_n(V)$ denote the space of
polydifferential operators of degree $n$ on $V$ spanned by Lie
polynomials.  Let $\Lie_n$ denote the space of abstract Lie
polynomials of degree $n$. One easily observes that $\SCLie_n(V)
\subseteq \SC_n(V)$ consists of the elements of top order,
$2(n-1)$. Thus, in the generating functions for $\SC_n(V)$, this
corresponds to the part where the degree in $s$ is one less than the
degree in $t$ (note that the degree in $t$ always exceeds the degree
in $s$, in the generating functions for $\SC$). Thus, Theorem
\ref{hlthm} implies
\begin{corollary}\label{liecor}
\begin{enumerate}
\item[(i)]
If we replace $\Inv_n(V)$ by
  $\SCLie_n(V)$, then the resulting generating functions are still a rational
fraction, of the form
\begin{equation}
\frac{t J_{\SCLie,\lambda,V}(u)}{\prod_{i} (1-u^{h_i(\lambda)})}, \quad \frac{t (1-u) J^+_{\SCLie,\lambda,V}(u)}{\prod_{i} (1-u^{h_i(\lambda)})},
\end{equation}
i.e., essentially rational functions of $u$.  
\item[(ii)]
When $\dim V + 1 \geq |\lambda|$, i.e., $\SCLie_n(V) \cong \Lie_n$,
then $J_{\SCLie,\lambda,V}$ and $J^+_{\SCLie,\lambda,V}$ are
independent of $V$ and has degree less than the degree (in $st$) of the denominator. Their evaluation at $s=t=1$ are $(|\lambda|-1)!$ and
  $(|\lambda|-2)!$, respectively.
\end{enumerate}
\end{corollary}
It is well known
that, as a representation of $S_n$, $\Lie_n \cong
\Ind_{\ZZ/n}^{S_n} \chi_1$, where $\chi_1$ is a primitive character of
$\ZZ/n$, i.e., $\chi(1) =\exp(\frac{2 \pi i}{n})$.\footnote{We remark
  that, while $\Lie_n$ is isomorphic to $\Ind_{\ZZ/n}^{S_n} \chi_1$ as
  a representation of $S_n$, it is not a priori obvious that the
  induced representation should admit an $S_{n+1}$ structure. Indeed,
  the $S_5$-representation $P_4 \cong \Ind_{\ZZ/5}^{S_5} \CC$ cannot
  admit an $S_6$ structure, as can be seen from explicitly looking at
  its decomposition ($\rho_{(5)} \oplus \rho_{(1,1,1,1,1)} \oplus
  2\rho_{(3,1,1)} \oplus \rho_{(3,2)} \oplus \rho_{(2,2,1)}$).} Thus,
we obtain the combinatorial identities 
\begin{gather}
\sum_{n \geq |\lambda|+\lambda_1} \dim \Hom_{S_{n}}(\rho_{\lambda}[n], \Ind_{\ZZ/n}^{S_n} \chi_1) t^n = 
\frac{t^{|\lambda|+\lambda_1}\kappa_{\lambda}(t)}{\prod_{i=1}^{|\lambda|} (1 - t^{h_i(\lambda)})}, \label{combliefla1} \\
\sum_{n \geq |\lambda|+\lambda_1} \dim \Hom_{S_{n+1}}(\rho_{\lambda}[n+1], \Lie_n) t^n = 
\frac{t^{|\lambda|+\lambda_1-1}(1-t)\kappa^+_{\lambda}(t)}{\prod_{i=1}^{|\lambda|} (1 - t^{h_i(\lambda)})},\label{combliefla2}
\end{gather}
where $\kappa^+_\lambda = J^+_{\SCLie,\lambda,V}|_{s=1}$ for any $V$ of dimension $\geq |\lambda|-1$, and similarly for $\kappa_\lambda(t)$. These are polynomials of degree less than the degrees of the denominators ($\sum_i h_i(\lambda)$). Moreover, $\kappa^+_\lambda(1) = (|\lambda|-2)!$ and $\kappa_\lambda(1)=(|\lambda|-1)!$.

We remark that \eqref{combliefla1} also follows from the
generalization of Corollary \ref{combcor} found in Appendix
\ref{combapp} (Claim \ref{combclaim}) to the case of induction of
nontrivial characters from $\ZZ/(n+1)$ to $S_{n+1}$.
\begin{example} From \eqref{ifla}--\eqref{i21pfla}, taking the part with degree
in $t$ one more than the degree in $s$,
  we deduce:
\begin{gather}
\kappa_{()}(t) = 0, \quad \kappa_{(1)}(t) = 1, \quad \kappa_{(2)}(t) = t = \kappa_{(1,1)}(t), \\
\kappa_{()}^+(t) = 0 = \kappa_{(1)}^+(t), \quad \kappa_{(2)}^+(t) = 1 = \kappa_{(1,1)}^+(t) = 1, \\
\kappa_{(3)}^+(t) = t^3, \quad \kappa_{(1,1,1)}^+ = t^2 = \kappa_{(2,1)}^+.
\end{gather}
\end{example}

We can also take the limit as $n \rightarrow \infty$ and obtain generating
functions for the multiplicities of irreducible representations in each order:
\begin{corollary}
\begin{gather}
\sum_{m \geq 0} \lim_{n \rightarrow \infty} \dim \Hom_{S_n}(\rho_{\lambda}[n], \Inv_n(V)_{2m}) s^{m} = 
\frac{J_{\lambda,V}|_{t=1}}{\prod_{i=1}^{|\lambda|} (1 - s^{h_i(\lambda)})}, \\ 
\sum_{m\geq 0} \lim_{n \rightarrow \infty} \dim \Hom_{S_{n+1}}(\rho_{\lambda}[n+1], \Inv_n(V)_{2m})s^m = \frac{J^+_{\lambda,V}|_{t=1}\cdot (1-s)}{\prod_{i=1}^{|\lambda|} (1 - s^{h_i(\lambda)})}.
\end{gather}
Moreover, the same formulas hold (with different numerators) if we replace $\Inv$ by $\Quant$ or $\SC$. 
\end{corollary}
The resulting polynomials in the above cases turn out to have
nonnegative coefficients:
\begin{gather}
J_{(),V}|_{t=1} = 1, \quad J_{(1),V}|_{t=1} = s, \quad
J_{(2),V}|_{t=1} = 2s^2, \quad J_{(1,1),V}|_{t=1} = s+s^3, \\
J^+_{(),V}|_{t=1}=1, \quad J^+_{(1),V}|_{t=1}=0, \quad
J^+_{(2),V}|_{t=1} = s^2, \quad J^+_{(1,1),V}|_{t=1} = s, \\
J^+_{(3),V}|_{t=1} = s^3+s^4, \quad J^+_{(1,1,1),V}|_{t=1}=s^3+s^4, \quad J^+_{(2,1)}|_{t=1} = s^2+s^4.
\end{gather}
\begin{question}\label{nonnegques}
Do $J$ and $J^+$ always have nonnegative coefficients when evaluated at $t=1$?  Is this true if we replace $\Inv_n(V)$ with $\Quant_n(V)$, $\SC_n(V)$, or $P_n$? (Or perhaps, it is better to use the related series obtained by replacing
$\rho_\lambda[n]$ and $\rho_\lambda[n+1]$ by $\Ind^{S_{n+1}}_{S_{|\lambda|} \times S_{n+1-|\lambda|}}(\rho_\lambda \boxtimes \CC)$ and $\Ind^{S_n}_{S_{|\lambda|} \times S_{n-|\lambda|}} (\rho_\lambda \boxtimes \CC)$, cf.~Question \ref{nonnegques2} below).
\end{question}
Finally, if we apply this last corollary in the case $|\lambda| \leq \dim V + 1$ (or equivalently in the limit $\dim V \rightarrow \infty$), we obtain the following interesting combinatorial identity:
\begin{corollary} There are polynomials $K'_\lambda, K'^+_\lambda$ such that
\begin{gather} 
\sum_{m \geq 0} \lim_{n \rightarrow \infty} \dim \Hom_{S_n}(\rho_{\lambda}[n], (P_n)_{2m}) s^{m} = 
\frac{K'_{\lambda}(s)}{\prod_{i=1}^{|\lambda|} (1 - s^{h_i(\lambda)})}, \\ 
\sum_{m\geq 0} \lim_{n \rightarrow \infty} \dim \Hom_{S_{n+1}}(\rho_{\lambda}[n+1], (P_n)_{2m})s^m = \frac{K'^{+}_{\lambda,V}(s)(1-s)}{\prod_{i=1}^{|\lambda|} (1 - s^{h_i(\lambda)})},
\end{gather}
where the degree of the numerators is less than that of the denominators.
\end{corollary}
\begin{question}\label{nonnegques2}
Do $K'_\lambda, K'^+_\lambda$ always have nonnegative coefficients?  It seems possible to prove at least that the numerators of the related series (with the
same denominators),
\begin{gather}\label{simnumeqn}
\sum_{m\geq 0} \lim_{n \rightarrow \infty} \dim \Hom_{S_{n+1}}(\Ind^{S_{n+1}}_{S_{|\lambda|} \times S_{n+1-|\lambda|}}(\rho_{\lambda} \boxtimes \CC), (P_n)_{2m})s^m, \\ \label{simnumeqn2}
\sum_{m\geq 0} \lim_{n \rightarrow \infty} \dim \Hom_{S_{n}}(\Ind^{S_n}_{S_{|\lambda|} \times S_{n-|\lambda|}}(\rho_{\lambda} \boxtimes \CC), (P_n)_{2m})s^m,
\end{gather}
have this property, using Theorem \ref{polythm}.(ii) and the
Chevalley-Shephard-Todd theorem. In more detail, the latter theorem
implies that, if $\hh$ is the reflection representation of
$S_{|\lambda|}$, then $\Sym \hh$ is a free module over $(\Sym
\hh)^{S_{|\lambda|}}$, and that this latter algebra is a polynomial
algebra. From this one can deduce that $\sum_{m \geq 0} \dim
\Hom_{S_{|\lambda|}}(\rho_\lambda, \Sym \hh) s^m = \frac{s^{\sum_{i}
    (i-1)\lambda_i}}{\prod_j (1-s^{h_j(\lambda)})}$. Hence, for any
free module $M$ over $\Sym \hh$, 
\begin{equation}\label{simmeqn}
\sum_{m \geq 0} \dim
\Hom_{S_{|\lambda|}}(\rho_\lambda, M) s^m =
\frac{P_M(s)}{\prod_j (1-s^{h_j(\lambda)})},
\end{equation} where $P_M(s)$ is a
polynomial with nonnegative coefficients.  
Finally, by Theorem \ref{polythm}.(ii) and Proposition \ref{fisxprop}, the direct sum $\bigoplus_{n \geq |\lambda|-1} (P_n)_{S_{n+1-|\lambda|}}$ is a free module over $\Sym \hh$. By Frobenius reciprocity, \eqref{simnumeqn} is the same as the RHS of \eqref{simmeqn} for $M = \bigoplus_{n \geq |\lambda|-1} (P_n)_{S_{n+1-|\lambda|}}$, graded by order. This implies that the numerator of \eqref{simnumeqn} indeed is a polynomial with nonnegative coefficients.  Similarly, for 
\eqref{simnumeqn2}, we use that $\Sym (\hh \oplus \CC)$
is a free module over $\Sym (\hh \oplus \CC)^{S_{|\lambda|}}$, and set $M = \bigoplus_{n \geq |\lambda|} (P_n)_{S_{n-|\lambda|}}$, now viewing $P_n$ as an $S_n$-representation. The same argument then implies that its numerator (after canceling the factor of $(1-s)$) also has nonnegative coefficients.
\end{question}

The proof of the theorem uses the following result, which is
interesting in itself, and related to Example \ref{fisxexam}:
\begin{definition} For $v \in V$, let $\Inv_{k}(V)^{v}$ be the
  space of polydifferential operators of degree $k$ 
  invariant under symplectic automorphisms which fix $v \in V$. Let
  $\Inv_{k,\ell}(V)^v$ be common eigenspace of symplectic
  automorphisms of the form $\phi(v)=c v$ with eigenvalue $c^\ell$,
  and call it \emph{weight $\ell$}.
\end{definition}
Note that $\Inv_k(V)^v$, viewed as $\C$-linear operators
$\cO_V^{\otimes k} \rightarrow \cO_V$ with constant coefficients, is a
module over the polynomial algebra $D_{v,k} := \C[\partial_v^{(1)},
\ldots, \partial_v^{(k)}]$. This is a graded module with respect to weight.

Further, note that $\Inv_k(V)^v$ and $\Inv_{k,\ell}(V)^v$ are additionally
graded by order of polydifferential operators.  Call the order $r$
parts $\Inv_k(V)^v_r$ and $\Inv_{k,\ell}(V)^v_r$. The difference $r-\ell$ is
preserved by the action of $D_{v,k}$, so $\Inv_k(V)^v$ naturally decomposes, as a $D_{v,k}$-module, into a direct sum $\Inv_k(V)^v = \bigoplus_{j \geq 0} \bigl(\bigoplus_{r-\ell = j} \Inv_{k,\ell}(V)^v_r \bigr)$.  The next theorem shows that this (outer) sum is finite, and gives information on the module structure.
\begin{theorem}\label{polythm}
\begin{enumerate}
\item[(i)]  The space $\Inv_k(V)^v$ is a finitely-generated
  $S_{k+1}$-equivariant graded module over $D_{v,k}$, and the order
  exceeds the weight by at most $k(k-1)$.  For $V =
  \CC^2$, this is sharp.
\item[(ii)] For $\dim V \geq k$, $\Inv_k(V)^v$ is a free module
  over the polynomial algebra $D_{v,k}$. 
It is naturally generated, as a graded
  representation of $S_{k+1}$, by the space of abstract Poisson
  polynomials of degree $k+1$.
\end{enumerate}
\end{theorem}
To relate $\Inv_k(V)^v$ back to $\Inv_n(V)$, we use the following result,
which is essentially an elaboration of Example \ref{fisxexam}.  Let
$v^* \in V^* \subseteq \cO_V$ be the element $v^*= \ad(v)$ corresponding to
$v$ via the symplectic structure.
\begin{proposition}\label{fisxprop}
  For all $n \geq k+\ell$, there is a natural $S_{k+1}$-linear
  injection $\iota_{k,\ell,n}: \Inv_{k,\ell}(V)^v \into \Inv_n(V)$ (using the
  inclusion $S_{k+1} \subseteq S_{n-k} \times S_{k+1} \subseteq
  S_{n+1}$), with left inverse given by restriction to $(v^*)^{\otimes
    (\ell)} \otimes 1^{\otimes n-(k+\ell)} \otimes \cO_V^{\otimes k}$. The $S_{n+1}$-linear span of
  the image of $\Inv_{k,\ell}(V)^v_r$ is the height $\leq k+1$ part of
  $\Inv_n(V)_{r+\ell}$.
 \end{proposition}
The proposition, together with Theorem \ref{polythm}, immediately
reduces the computation of the height $k \leq \dim V + 1$-part (as
$S_{n+1}$-representations) to the structure of the abstract Poisson
polynomials, $P_k$, of degree $k$.  In particular, this immediately
gives the result for height $\leq 3$ for all $V$, and implies Theorem
\ref{isotpartthm} and Corollary \ref{hot3cor} (which we also prove
independently: for Corollary \ref{hot3cor}, see Appendix \ref{s:hot3cor}).

We also obtain the following consequence on the
top order of $\Inv_n(V)$:
\begin{corollary} \label{toporddelcor} For any fixed $k \geq 1$ and
  sufficiently large $n \geq 0$, the order of the height $\leq k+1$ part of
  $\Inv_n(V)$, as representations of $S_{n+1}$, is at most $k(k-1) + 2(n-k)$.  When $V = \CC^2$, this is sharp. More precisely, for every
  irreducible representation $\rho$ of $S_{k+1}$, there exists $n \geq k$ such
  that $\rho[n+1]$ is an irreducible representation appearing in $\Inv_n(\CC^2)_{k(k-1) + 2(n-k)}$.  
\end{corollary}
This contrasts with the growth of the top order with respect to total
degree $n$, which we only know (by Remark \ref{invtoporderrem}) is at
least polynomial of degree $1+\frac{1}{\dim V}$ in $n$, i.e.,
$\frac{3}{2}$ in $n$ in the case $V = \CC^2$.

Similarly, one can deduce information about the growth of the
dimension of $\Inv_n(V)$:
\begin{corollary}The dimension of the height $\leq k$ part of $\Inv_n(V)$ grows
polynomially in $n$ of degree $k$.  
\end{corollary}
\begin{remark} \label{poissoprem} By Theorem \ref{scinvisotthm},
  Corollary \ref{hot3cor}, together with Theorem \ref{isotpartthm}, is
  equivalent to computing the height $\leq 3$ part of the space of
  abstract Poisson polynomials, as a graded representation of
  $S_{n+1}$.  This seems like it is a classical problem.  If we
  restrict our attention to the $S_n$-structure, V. Dotsenko pointed
  out to us how to compute the character of the Poisson operad as a
  graded $S_n$-representation, and using this, we can obtain the
  following generating function for the characters of
  $(P_n)_{2m}$. Let $(P_n)_{2m}^{(k)}$ be the part of $(P_n)_{2m}$
  which lies in a sum of irreducible $S_n$-representations of height
  $k$. We view $(P_n)_{2m}^{(k)}$ as an $S_k$-representation, and its
  isotypic decomposition gives the part of the irreducible
  $S_n$-decomposition of $(P_n)_{2m}$ corresponding to height $k$
  representations. Then, let $G_{P}(s,t)(p_1, p_2, \ldots)$ be the
  polynomial in $s, t, p_1, p_2, \ldots$ whose coefficient of
  $\frac{s^m t^n p_1^{r_1} \cdots p_j^{r_j}}{1^{r_1} r_1!2^{r_2} r_2!
    \cdots j^{r_j}r_j!}$, for $r_1 + \cdots + r_j = k$, is the trace
  of an element of $S_{k}$ whose cycle decomposition has $r_i$
  $i$-cycles, acting on $(P_n)_{2m}^{(k)}$ (cf.~Remark
  \ref{symmfunrem}).  Then, one may prove the formula
\begin{equation}\label{gpfla}
G_{P} \equiv  e^{-\sum_i \frac{p_i}{i}} \frac{1}{1-t}\prod_{i \geq 1} (1-q_i s^i)^{-\frac{g_i(s)}{is^i}}, \quad q_i := \frac{p_i t^i}{1-s^it^i}, \quad g_i(s) := \sum_{d \mid i} \mu(d) s^{d-1},
\end{equation}
where $\mu$ is the classical M\"obius function (giving $(-1)^\ell$ for
square-free integers with $\ell$ prime factors, and zero otherwise).  The
 $\equiv$ above means that the graded multiplicity of $\rho_\lambda[n]$ as
computed from both sides (using the part of degree $n$ in $t$) is the same for all $\lambda$ with $|\lambda|=n$: e.g., the equality is unaffected by terms
on the RHS whose total degree in $p_1, p_2,
\ldots$ exceeds the total degree in $t$ (which cannot occur in $G_P$
by definition). We derived \eqref{gpfla} from the formula
$\prod_{i \geq 1} (1-p_i s^i)^{-\frac{g_i(s)}{is^i}}$ for the graded
character of the $P_n$ themselves (taking traces using the $S_n$
structure on all of $P_n$), which Dotsenko provided us, using
the arguments in the proof of Theorem \ref{hlthm} (\S \ref{hlthmpfsec}).

To obtain
the formula for the graded $S_{n+1}$-structure, if $F_P^+(s,p_1,p_2,\ldots)$
is the formula for the graded $S_{n+1}$-character of the $P_n$ themselves,
the formula for the graded $S_k$-characters of the height $k$ parts
$(P_n)_{2m}^{(k),+}$ of $(P_n)_{2m}$ as $S_{n+1}$-representations would be
\begin{equation}
G_P^+(s,t)(p_1,p_2,\ldots) \equiv e^{-\sum_i \frac{p_i}{i}}\frac{1-st}{1-t} F_P^+(s,q_1,q_2,\ldots).
\end{equation}
We were not able to find the formula for $F_P^+(s,p_1,p_2,\ldots)$, although perhaps it is known.
\end{remark}
\subsubsection{The case of height $4$} \label{ht4sec} Using the
preceding subsection and its proof, we are able to deduce the
structure of $\Inv(V)$ also in height four.  Our main result here is,
in the language of \S \ref{scquantsec} below, that the height four
parts of $\Inv(V)$ and $\Quant(V)$ coincide, even for $V = \CC^2$ (for
$\dim V \geq 4$, $\Inv(V) = \SC(V)$, but this is not true for $V =
\CC^2$), cf.~Corollary \ref{ht4cor}.  In this section, we describe
the explicit structure without mentioning $\Quant$.

By the preceding section, it suffices to understand the
$D_{v,3}$-module $\Inv_3(V)^v$.  This module is bigraded by order and
weight, with the difference of the two gradings invariant under
$D_{v,3}$. For any representation $\rho$ of $S_3$, let $\rho(a,b)$
denote the bigraded representation considered to have order $a$ and
weight $b$. (Recall that the resulting height $\leq 4$ subspace of
$\Inv(V)$ will lie in $\Inv_{\geq 3+b}(V)_{a+b}$.)

In the case $\dim V \geq 4$, the structure of $\Inv_3(V)^v$
follows from Theorem \ref{polythm}: it
is the free module
generated by
\begin{equation} \label{ht4dimge4eqn}
\C(0,0) \oplus \hh(2,0) \oplus \rho_{(2,2)}(4,0),
\end{equation}
where, for any
$S_{n+1}$-representation $M$, we let $M(2m)$ denote the same
representation viewed as a graded representation in order $2m$.

For the case $V = \CC^2$, we prove the following theorem:
\begin{theorem} \label{ht4thm}
The $D_{v,3}$-module $\Inv_3(\CC^2)^{v}$
is isomorphic to the quotient of the free module generated
by
\begin{equation} \label{ht4genfla}
\C(0,0) \oplus \hh(2,0) \oplus \rho_{(2,2)}(4,0) \oplus \sgn(7,1)
\end{equation}
by the free submodule generated by $\sgn(3,1) \subseteq D_{v,3} \hh(2,0)$.
\end{theorem}
This immediately yields the generating functions $I_{\lambda,V},
I_{\lambda,V}^+$ for all partitions $\lambda$ of size four (with
denominators as in Theorem \ref{hlthm}).  We just give the evaluations of the numerators at $t=1$ (so as to verify Question \ref{nonnegques} in this case):
\begin{gather}
J_{(3),\CC^2}|_{t=1} = 3s^3+2s^4+s^7, \quad J_{(1,1,1),\CC^2}|_{t=1} = 2s^3+3s^4+s^6, 
\\ J_{(2,1),\CC^2}|_{t=1} = 2s^2+3s^4+s^5, \\ 
J_{(4),\CC^2}^+|_{t=1} = 2s^4+s^5+2s^6+s^{10}, \quad
J_{(1,1,1,1),\CC^2}^+|_{t=1} = 2s^4+2s^6+s^7+s^8, \\
J_{(2,2),\CC^2}^+|_{t=1}=s^2+s^4+s^5+3s^6, \\ J_{(2,1,1),\CC^2}^+|_{t=1} = s^3+s^4+3s^5+s^7, \quad J_{(3,1),\CC^2}^+|_{t=1}=2s^3+2s^5+s^6+s^7.
\end{gather}
For $\dim V > 2$, we obtain
\begin{gather}
  J_{(3),V}|_{t=1} = 3s^3+2s^4+s^7, \quad J_{(1,1,1),V}|_{t=1} = s^2+2s^3+2s^4+s^6,
  \\ J_{(2,1),V}|_{t=1} = 2s^2+s^3+3s^4, \\
  J_{(4),V}^+|_{t=1} = 2s^4+s^5+2s^6+s^{8}, \quad
  J_{(1,1,1,1),V}^+|_{t=1} = s^2+s^4+2s^6+s^7+s^8, \\
  J_{(2,2),V}^+|_{t=1}=s^2+2s^4+s^5+2s^6, \\ J_{(2,1,1),V}^+|_{t=1} =
  2s^3+3s^5+s^7, \quad J_{(3,1),V}^+|_{t=1}=2s^3+3s^5+s^6.
\end{gather}
The formulas including $t$ are not very enlightening. We only give them
for $\lambda=(1,1,1,1)$, thereby verifying Question \ref{sgnposques} (let $\dim V > 2$ below):
\begin{gather} \label{i1111pfla}
J_{(1,1,1,1),\CC^2}^+ = 
u^4+u^4t^2+u^6t+u^6t^3+u^7t^2+u^8t,  \\
J_{(1,1,1,1),V}^+ = 
u^2t^2 + u^4t^2 + u^6t + u^6t^3 + u^7t^2 + u^8t.
\end{gather}
If we evaluate at $s=1$, we obtain
\begin{gather}
  \sum_{n \geq 4} \dim \Hom_{S_{n+1}}(\wedge^4 \hh, \Ind_{\ZZ/(n+1)}^{S_{n+1}} \CC) t^n = \frac{t^4(1+t^2+t^3+3t^5)}{(1-t)(1-t^2)(1-t^3)(1-t^4)}.
\end{gather}

\subsubsection{$\Inv(V)_{2m} = \SC(V)_{2m}$ in the leading three heights}
\label{tophtsubsec}
Next, we switch gears and consider, instead of low heights, the
leading heights for a fixed order $m$, which turns out also to be
tractable. Let us first mention the following result, which we will
refine and explain in \S \ref{asympsec} below:
\begin{proposition}\label{polyprop} Fix a symplectic vector space $V$ and an order
  $2m$.  Then, the dimensions of $\Inv_n(V)_{2m}$ and
  $\SC_n(V)_{2m}$ are polynomials in $n$ of degree $2m$.\footnote{We
    prove this also for $\Quant_n(V)_{2m}$, but are avoiding its
    discussion until \S \ref{asympsec}.}
\end{proposition}
\begin{theorem}\label{scinvordthm}
For fixed order $2m$, the dimension of $\Inv_n(V)_{2m}/\SC_n(V)_{2m}$
is polynomial of degree $\leq 2m-3$. That is, no height $> 2m-3$ 
irreducible representations occur in
  $\Inv_n(V)_{2m}/\SC_n(V)_{2m}$. 
\end{theorem}
More precisely, we can divide the aforementioned theorem into the following
cases (as we will see, this will follow quickly from the previous theorems,
using Example \ref{modmatexam} for the equality of degree in part (iii)): 
\begin{theorem}\label{scinvordconvthm}
\begin{enumerate}
\item[(i)] For $2m \leq \dim V$, $\Inv_n(V)_{2m} = \SC_n(V)_{2m} \cong
(P_n)_{2m}$.
\item[(ii)] For $2m \in \{\dim V + 2, \dim V + 4\}$, $\Inv_n(V)_{2m} = \SC_n(V)_{2m}$, and there is a canonical (noninjective) surjection $(P_n)_{2m} \onto \Inv_n(V)_{2m}$.
\item[(iii)] For $2m \geq \dim V + 6$, $\Inv_n(V)_{2m} /
  \SC_n(V)_{2m}$ is nonzero and has polynomial dimension in $n$ of
  degree $2m-3$.
\end{enumerate}
\end{theorem}
We can therefore compute the order $2m \leq \dim V$ part by computing
the space of Poisson polynomials which are a sum of terms with exactly
$m$ brackets, and the order $2m=\dim V + 2$ or $\dim V+4$ part by
finding also which of these polynomials vanish. For instance,
\begin{equation}
\dim \Inv_n(\CC^{\geq 2m})_{2m} = \dim (P_n)_{2m} = n! \sum_{i_1 + 2i_2 + \cdots + ki_k = m} \bigl((n-m-\sum_j i_j)! (\prod_j (j+1)^{i_j} i_j!)\bigr)^{-1},
\end{equation}
which really comes from the identification, as $S_n$-representations,
\begin{equation} \label{pnsnstr}
(P_n)_{2m} \mathop{\cong}^{S_n} \bigoplus_{i_1 + 2i_2 + \cdots + ki_k = m} \Ind_{(S_{2}^{i_1} \rtimes S_{i_1}) \times \cdots \times (S_{k+1}^{i_k} \rtimes S_{i_k})
\times S_{n-m-\sum_j i_j}}^{S_n} \bigl(\bigotimes_j (\text{Lie}_{j+1})^{\otimes i_j} \otimes \CC\bigr),
\end{equation}
where $\text{Lie}_{j+1} \cong \Ind_{\ZZ/(j+1)}^{S_{j+1}} \chi_{j+1}$,
and $\chi_{j+1}$ is the cyclotomic character $1 \mapsto e^{\frac{2 \pi
    i}{j+1}}$.\footnote{Here, $\text{Lie}_n$ is the $n$-th part of the
  Lie operad, whose $S_n$-structure was computed by Shklyafarov.}
Computing the $S_{n+1}$-structure is harder, but for any fixed $m$,
this is a direct computation using the fact that only height $\leq 2m$
representations of $S_{n+1}$ can occur.  (One can also apply this
algorithm to compute the height $\leq 2d + 1$ part of
$(P_n)_{2m}$ and hence of $\Inv_n(\CC^{\geq 2d})_{2m}$, for each fixed
$m$, $n$, and $d$).

Finally, we can compute the irreducible representations which make up
the part of $\Inv_n(V)_{2m}$ of height $\geq 2m-2$, which coincides
with that of $(P_n)_{2m}$ for $\dim V \geq 2m$. In fact, this
decomposition does not essentially depend on $V$, except that,
roughly, the columns can increase in length by two when $\dim V$
increases by two (at least for the case $\dim V \geq 4$).  We only do
this for $\dim V = 2$ (cf.~Figures \ref{firstyoung}--\ref{lastyoung}
and Corollary \ref{youngcor} and the ensuing discussion):
\begin{corollary} \label{c2sslcor}
The $S_{n+1}$-subrepresentations of $\Inv_n(\CC^2)_{2m}$ of height
$\geq 2m-2$ are, for $n \geq 3m-1$,  those whose
truncated Young diagrams appear below with subscript at most $m$:
\begin{multline}
(m,m)_0, \quad (m,m-1)_{2}, (m,m-2,1)_3, \quad (m,m-2)_2, (m,m-2)_3, (m-1,m-2,1)_3, \\ (m,m-2)_4,  (m,m-1,1)_4, (m-2,m-2,2)_4, (m-1,m-3,1,1)_4, (m,m-3,1)_5, (m,m-4,2)_6.
\end{multline}
\end{corollary}

\subsubsection{The case of order $\leq 6$}
By Theorem \ref{scinvordconvthm}, $\Inv$ and $\SC$ coincide in orders
$\leq 6$. Here we compute this part explicitly.
\begin{corollary}\label{ordleq4dimcor}
$\Inv_n(V)_2 \cong \wedge^2 \hh$, for all $n$ and $V$.
Moreover, 
\begin{gather}
\dim \Inv_n(V)_{4} = \begin{cases} 2 {n+1 \choose 4}, & \text{if $\dim V = 2$}, \\
\frac{1}{4} {n \choose 3}(3n-1), & \text{if $\dim V \geq 4$}.
\end{cases} \\
\dim \Inv_n(V)_6 = \begin{cases}  \frac{1}{6} {n+1 \choose 5}(5n+16), & \text{if $\dim V = 2$}, \\  \frac{1}{3}{n+1 \choose 5}(7n-10), & \text{if $\dim V = 4$}, \\ {n \choose 4}{n \choose 2}, & \text{if $\dim V \geq 6$}.
\end{cases}
\end{gather}                                     
\end{corollary}
\begin{corollary} \label{meq23str}
For $\dim V > 2$, $\Inv_n(V)_4$ is the direct
sum of the four irreducible representations whose truncated Young
diagrams are $(2,2), (2,1), (2)$, and $(1,1,1,1)$ (where we omit those that would require more than $n+1$ cells, or alternatively restrict to $n \geq 5$). In the
case $\dim V = 2$, only the first three appear.
For $m=3$ and $n \geq 6$, $\Inv_n(V)_6$ is the direct sum of the representations whose truncated Young diagrams are 
\begin{multline}
(3,3), (2,2,1,1), (1,1,1,1,1,1), (3,2), (3,1,1), (2,2,1), (2,1,1,1), \\ (3,1), (3,1), (2,1,1), (2,1,1), (3), (2,1), (1,1,1), \text{ and } (1,1),
\end{multline}
except that $(1,1,1,1,1,1)$ occurs only for $\dim V \geq 6$, and
$(2,2,1,1), (2,2,1), (2,1,1,1)$, and one of the $(2,1,1)$
representations occur only for $\dim V \geq 4$. Also, we omit here the
full diagrams requiring more than $n+1$ cells, or alternatively
restrict to $n \geq 8$.
\end{corollary}
The proof of the corollary is a direct computation: it is not
difficult to compute $(P_n)_{\leq 6}$ in full generality (the $S_n$
structure is given in \eqref{pnsnstr}, and as explained there, this
yields the $S_{n+1}$-structure by a direct computation).  Then, it
remains only to see which of these representations vanish when $\dim V
= 4$ and $\dim V = 2$, which is not difficult to do directly.  We omit
the details.

We remark that a similar computation should yield $(P_n)_{8}$
without too much difficulty, and it is perhaps not hard to extend the
corollary to the case of order $8$ (although the answer will be
more complicated), using the computational results of \S
\ref{lowdimsec} for the case $\dim V = 2$
(cf.~\eqref{m4scdim}--\eqref{m4invdim}).

\subsection{Computational results for $\Inv_n(\CC^2)$, $n \leq
  6$} \label{lowdimsec} Thanks to Theorem \ref{invinterpthm} below, we
are able to present some computational data on $\Inv_n(\CC^2)$.
Namely, using programs written in Magma \cite{magma} together with
techniques from \cite{hp0bounds} to prove (sharp) bounds on the top
order of $\Inv_n(\CC^2), n \leq 5$, we were able to compute the
Hilbert series of $\Inv_n(\CC^2)$ for $n \leq 5$ and for $n \leq 6$ in
order $\leq 8$. This is given in Figure \ref{hilbsertable}.
\begin{figure}
\begin{center}
\begin{tabular}{c|c}
 \ $n$ \Bline \   & \ $h(\Inv_n(\CC^2);t)$
\ \\[4pt] \hline 
\ $1$ \Tline \Bline \  & $1$ \\ \hline
\ $2$ \Tline \Bline \  & $1 + t^2$ \\ \hline
\ $3$ \Tline \Bline \  & $1 + 3t^2 + 2t^4$ \\ \hline
\ $4$ \Tline \Bline \  & $1 + 6t^2 + 10t^4 + 6t^6 + t^8$ \\ \hline
\ $5$ \Tline \Bline \ &  $1 + 10t^2 + 30t^4 + 41t^6 + 30t^8 + 9 t^{10}$ \\ \hline
\ $6$ \Tline \Bline \  & $1 + 15t^2 + 70t^4 + 161t^6 + 224 t^8 + \ldots$ \\ \hline
\end{tabular}
\end{center}
\caption{Hilbert series of $\Inv_n(\CC^2)$ for $n \leq 5$ and for $n =
  6, m \leq 8$.\label{hilbsertable}}
\end{figure}

Note that, for $n \leq 4$, the entries of the table are also a
consequence of Theorem \ref{isotpartthm}, Corollary \ref{hot3cor}, and
Theorem \ref{ht4thm}, or alternatively a consequence of the generating
functions \eqref{ipfla}--\eqref{i21pfla} and the corresponding ones
for the height four case derivable from Theorem \ref{ht4thm} (e.g.,
\eqref{i1111pfla}).  More generally, these formulas account for the
height $\leq 4$ part, which is almost all of the table.  For $n=5$,
this is everything except for the part corresponding to the sign
representation of $S_6$. Here, we see that the sign representation
only occurs once, in degree $8$ (and the operator is the one from
\cite[(8.11)]{Matso}, recalled in Example \ref{matgenexam} below).
Let $\phi$ be an operator spanning this representation (e.g.,
\eqref{matgeneq} for $k=2$).  For $n=6$, the height $\leq 4$ part is
everything except for the six-dimensional copy of $\wedge^5(\hh \oplus
\CC)$ in $\Inv_6(\CC^2)_8$ generated by the operator $\mu \circ (\Id
\otimes \phi)$ resulting from $\phi \in \Inv_5(\CC^2)_8$ as above, and
one fourteen-dimensional height $5$ representation, $\rho_{(2,1,1,1)}$,
also contained in $\Inv_6(\CC^2)_8$.  In summary, all but a
one-dimensional part of $\Inv_5(\CC^2)$ and a twenty-dimensional part of
$\Inv_6(\CC^2)_{\leq 8}$ in the above table is derivable from the
aforementioned theorems (and we checked to make sure the theorems
yield the same answer as Figure \ref{hilbsertable}).

\subsection{The subspace $\Quant_n(V) \subseteq \Inv_n(V)$ and Poisson and Hochschild homology}\label{scquantsec}
Now, we explain an interpretation of $\Inv_n(V)$ using
Poisson homology.  Recall that, for any Poisson algebra $A_0$ (e.g., $A_0
= \cO_V$), one may consider the space of Poisson traces, defined as
linear functionals $A_0 \rightarrow \CC$ which annihilate all brackets
$\{f,g\}$, for $f, g \in A_0$.  This is dual to the zeroth Poisson, or
Lie, homology, $\HP_0(A_0) := A_0 /\{A_0,A_0\}$.  More generally, if $A_0
\subseteq B_0$ is an inclusion of Poisson algebras, we may define 
 $\HP_0(A_0, B_0) := B_0/\{A_0,B_0\}$ and its dual $\HP_0(A_0,B_0)^*$.
 \begin{theorem} \label{invinterpthm}\footnote{We are grateful to Pavel Etingof for observing this result, which motivated this work.} The space $\Inv_n(V)$ is
   canonically isomorphic to $\HP_0(\cO_{V^n}^{S_{n+1}},
   \cO_{V^n})^*$, as $S_{n+1}$-representations.
\end{theorem}

Next, we recall that Mathieu defined a certain subrepresentation
$\Quant_n(V)$ lying between $\SC_n(V)$ and $\Inv_n(V)$, consisting of
operations obtainable via deformation quantization of $V$.  In more
detail, one may consider a noncommutative $\CC[[\hbar]]$-linear
multiplication $\star$ on $\cO_V[[\hbar]]$ such that
\begin{equation}
f \star g \equiv fg \pmod{\hbar}, \\
f \star g - g \star f \equiv \hbar \{f,g\} \pmod{\hbar^2}.
\end{equation}
Explicitly, such a multiplication is given by the Moyal formula,
\begin{equation}
f \star g = \mu \circ e^{\frac{1}{2} \hbar \pi}(f \otimes g) 
= \mu \sum_{k \geq 0} \frac{1}{2^k k!} \hbar^k \pi^k(f \otimes g),
\end{equation}
where $\mu(f \otimes g) = fg$ is the undeformed multiplication, and $\pi$ is the Poisson bivector, given in standard coordinates $x_1,\ldots, x_d, y_1, \ldots, y_d$ by
\begin{equation}
  \pi = \sum_i \frac{\partial}{\partial x_i} \otimes \frac{\partial}{\partial y_i} - \frac{\partial}{\partial y_i} \otimes \frac{\partial}{\partial x_i}.
\end{equation}
For any element $f \in \cO_V[[\hbar]]$, let $\gr f \in \cO_V$ denote
the first nonzero coefficient of the power series in $\hbar$, i.e.,
the lowest-order nonvanishing derivative $\frac{1}{k!}
\frac{\partial^k}{\partial \hbar^k}(f)|_{\hbar = 0}$ in the variable
$\hbar$ at zero.

Then, one may consider polydifferential operators of the form 
\begin{equation}
  f_1 \otimes \cdots \otimes f_n \mapsto \gr \sum_{\sigma \in S_n} c_\sigma  f_{\sigma(1)} \star \cdots \star f_{\sigma(n)},
\end{equation}
for some constants $c_\sigma \in \CC$.  In \cite{Matso}, Mathieu
defined $\Quant_n(V) \subseteq \Inv_n(V)$ as the subspace of such
polydifferential operators.  As he noticed, this space does not depend
on the choice of star product $\star$, and defines a
sub-$S_{n+1}$-representation, and in fact a graded cyclic suboperad.
Moreover, the dimension of $\Quant_n(V)$ is $n!$, since the
expressions $f_{\sigma(1)} \star \cdots \star f_{\sigma(n)}$ for
$\sigma \in S_n$ are linearly independent over $\CC$ (as observed by
Mathieu, which is easy to verify using either the Moyal product or the
identification with differential operators). As a result, as stated in
the introduction, $\Quant_n(V)$ is the regular representation of $S_n$
and the representation $\Ind_{\ZZ/(n+1)}^{S_{n+1}}$ of $S_{n+1}$, and
$\Quant(V)$ is the associated graded of the associative operad.

One may alternatively define $\Quant_n(V)$ by writing $V = U \oplus
U^*$ for some Lagrangian $U \subseteq V$ (i.e., the
$(x_1,\ldots,x_d)$-hyperspace), lifting $f_1, \ldots, f_n$ to
differential operators $F_1, \ldots, F_n \in \cD_U$ whose principal
symbols satisfy $\gr(F_i) = f_i$, and then interpreting $\gr(f_1 \star
\cdots \star f_n)$ above as $\gr(F_1 \cdots F_n)$.

Then, Theorem \ref{ht4thm} translates (using Proposition
\ref{fisxprop}) to
\begin{corollary}\label{ht4cor} The height $\leq 4$ part of $\Inv(V)$ coincides with that of $\Quant(V)$.
\end{corollary}
\begin{proof} 
  Let $v^* \in \cO_V$ be the linear functional $\ad(v)$, corresponding
  to $v$ via the symplectic form.  It suffices to show that
  $\Inv_3(V)^v$ is generated by $\SC_3(V) = \Inv_3(V)$ and
  $\Quant_4(V)|_{v^* \times \cO_V^{\otimes 3}}$.  The first three
  summands of \eqref{ht4genfla} are spanned by Poisson polynomials,
  which already lie in $\SC_3(V)$, whereas the last summand
  corresponds to the sign representation of $\Quant_4(V)$.  This
  proves the result.
\end{proof}

Next, recall that for any associative algebras $A \subseteq B$, one may consider the zeroth Hochschild homology group $\HH_0(A, B) := A/[A,B]$, and the
Hochschild trace group $\HH_0(A,B)^*$, where here $[A,B]$ is the linear span
of elements $[a,b]$ for $a \in A, b \in B$.  In the case that $A = A_0[[\hbar]]$, $B = B_0[[\hbar]]$ are deformation quantizations of
  Poisson algebras $A_0 \subseteq B_0$, one has a natural surjection
\begin{equation}
\HP_0(A_0, B_0)((\hbar)) \onto \HH_0(A, B)[\hbar^{-1}],
\end{equation}
or in the case that $A$ and $B$ are filtered algebras (e.g., $A =
\cD_U$ and $A_0 = \cO_V$) whose associated graded are the Poisson
algebras $A_0, B_0$, this becomes $\HP_0(A_0, B_0) \onto \gr \HH_0(A,B)$.
Dually, in the latter case, one has the
inclusion $\HH_0(A, B)^* \subseteq \HP_0(A_0,B_0)$, and one may make a
similar statement for deformations.

We then have the following interpretation of $\Quant_n(V)$, analogous
to Theorem \ref{invinterpthm}:
\begin{theorem}\footnote{We thank Pavel Etingof for pointing out this result.} \label{quantinterpthm} The isomorphism of
  Theorem \ref{invinterpthm} restricts to $\Quant_n(V)
  \cong \HH_0(\cD_{U^n}^{S_{n+1}}, \cD_{U^n})^*$, or equivalently,
  $\Quant_n(V)((\hbar)) \cong \HH_0((\cO_{V^n}^{S_{n+1}}((\hbar)),
  \star), (\cO_{V^n}^{S_{n+1}}((\hbar)), \star))$.  Moreover, the direct
sum $\bigoplus_n \HH_0(\cD_{U^n}^{S_{n+1}}, \cD_{U^n})^*$ is canonically
identified with the associative operad, equipped with a certain filtration.
\end{theorem}
Using Theorems \ref{scinvisotthm} and \ref{ht4thm}, we deduce
\begin{corollary} The natural surjection $\HP_0(\cO_{V^n}^{S_{n+1}}, \cO_{V^n}) \onto \gr \HH_0(\cD_{U^n}^{S_{n+1}}, \cD_{U^n})$ is an isomorphism when restricted
to the isotypic part of height $\leq \dim V + 1$ representations.
Moreover, in heights $\leq 4$,
$\HP_0(\cO_{\CC^{2n}}^{S_{n+1}}, \cO_{\CC^{2n}}) \onto \gr \HH_0(\Weyl(\CC^{2n})^{S_{n+1}}, \Weyl(\CC^{2n}))$ is an isomorphism.
\end{corollary}
As we will see in the next section, this surjection is \emph{not} an
isomorphism when restricted to heights $\geq \dim V + 3$ (provided $n
\geq \dim V + 2$), and in particular, for $V = \CC^2$, the above result
is sharp.  This will be made more precise.

Similarly, Theorem \ref{scinvordthm} implies
\begin{corollary}
The dimension of the kernel of $\HP_0(\cO_{V^n}^{S_{n+1}}, \cO_{V^n})_{2m} \onto \gr_{2m} \HH_0(\cD_{U^n}^{S_{n+1}}, \cD_{U^n})$ is at most $\frac{C}{n^3}$ times the
dimension of the domain, where $C > 0$ is a constant independent of $n$.
\end{corollary}

\subsection{Examples and biconditional results on subquotients of $\SC
  \subseteq \Quant \subseteq \Inv$}\label{bicondsec}
In this section, we will explain a partial converse to Theorem
\ref{scinvisotthm}, that roughly says that 
any irreducible height $\geq \dim V
  + 2$ representation
occurs in the quotient $\Quant_n(V)/\SC_n(V)$ for sufficiently large $n$,
and similarly every irreducible height $\geq \dim V + 3$ representation
occurs in $\Inv_n(V)/\Quant_n(V)$ for sufficiently large $n$.
But first, we explain
some examples of invariant operators that we will use to construct these
representations.
\begin{example} \label{matgenexam} \cite[8.6]{Matso} Given $k \geq 1$,
  we may consider the following invariant polydifferential operator
  $\cO_{\C^2}^{\otimes n} \rightarrow \cO_{\C^2}$, of order $2{k+2
    \choose 3}$ for $n+1 = {k+2 \choose 2}$:
\begin{equation}\label{matgeneq}
  \frac{\partial}{\partial x} \wedge \frac{\partial}{\partial y} \wedge 
  \frac{\partial^2}{\partial x^2} \wedge \frac{\partial^2}{\partial xy} \wedge 
  \frac{\partial^2}{\partial y^2} \wedge \frac{\partial^3}{\partial x^3} \wedge 
  \cdots \wedge \frac{\partial^k}{\partial y^k}.
\end{equation}
This generates a representation which is skew-symmetric in the first
$n$ indices. In the case $k = 1$ ($n = 2$), this is the sign
representation, and otherwise it is the representation $\wedge^{n-1}
\hh \oplus \wedge^{n} \hh$, including the sign representation.  These
two representations cannot both occur in $\Quant_n(\C^2)\cong \Ind_{\ZZ/(n+1)}^{S_{n+1}} \CC$, so for $k \geq 2$, the above element does not lie in $\Quant_n(\CC^2)$, and hence $\Quant_n(\CC^2) \neq
\Inv_n(\CC^2)$. The first such case is $n=5$, where $\Quant_5(\CC^2)_8 \neq
\Inv_5(\CC^2)_8$ (as observed by Mathieu).

Moreover, it is clear that this is the largest value of $n$ for which
a sign representation can occur in the given order $2{k+2 \choose 3}$
or lower, and that the sign representation cannot occur in lower order
for this value of $n$.  For $k \geq 2$, one obtains a copy
of the sign representation as well for $n = {k+2 \choose 2}-2$, if
\eqref{matgeneq} is instead interpreted as a distributional function
of $n$ functions, i.e., an invariant polydifferential operator of degree
$n-1$ instead of $n$. This representation cannot lie in
$\Quant_n(\CC^2)$ in the case that $n$ is odd.

In particular, for $n=4$ ($k=2$), the above shows that the sign
representation occurs in order $8$. In this case, it is the top order
of invariant polydifferential operator, and lies in $\Quant_4(\CC^2)$.
However, it is not in $\SC_4(\CC^2)$, since any Poisson polynomial in
four variables can have order at most $6$ as a polydifferential operator:
hence, as Mathieu observed, $\SC_4(\CC^2)_8 \neq \Quant_4(\CC^2)_8$.
\end{example}
\begin{example}\label{matgenvexam}
  As noted in \emph{op. cit.}, the above operator generalizes to the
  setting of a symplectic vector space $V$ of dimension $\geq 2$, by
  wedging together all partial derivatives of order $\leq k$. One
  obtains an operator invariant not merely under symplectic
  automorphisms of the formal neighborhood of zero, but in fact under
  all automorphisms preserving the volume form.  This operator has
  order $2{k+\dim V \choose \dim V + 1}$ and degree ${k + \dim V
    \choose \dim V} - 1$. Moreover, it is shown there that this copy
  of the sign representation does not come from $\Quant_n(V)$ for
  large enough $n$.
\end{example}
\begin{remark}\label{invtoporderrem}  
  The above examples show that the top order of $\Inv_n(V)$ grows at
  least as fast as $C \cdot n^{1 + \frac{1}{\dim V}}$ asymptotically,
  where $C = 2 \frac{\bigl((\dim V)!\bigr)^{\frac{1}{\dim V}}}{\dim V
    + 1}$.
\end{remark}
\begin{remark}\label{quanttoporderrem}
  In fact, the top order of $\Quant_n(V)$ also grows at least as fast
  as $C \cdot n^{1 + \frac{1}{\dim V}}$, since the sign representation
  of $S_n$ must occur, and as explained above, it must occur in order
  at least as large as in the above example.  By Corollary
  \ref{toporddelcor}, its top order is at most $n(n-1)$. Hence the
  growth of the top order of $\Quant_n(V)$, as well as that of
  $\Inv_n(V)$, is polynomial of some degree in the interval
  $[1+\frac{1}{\dim V}, 2]$ in $n$ (which already follows from
  \cite{Matso}).  It would be interesting to compute the actual value.
\end{remark}
\begin{remark}\label{fingenrem}
  As observed by Mathieu in the case of $\Inv_n(V)$, the fact that the
  order does not grow linearly in $n$ implies that $\Inv_n(V)$ and
  $\Quant_n(V)$ are not finitely generated as operads.  This is
  especially interesting because, as a representation of $S_{n+1}$,
  $\Quant_n(V)$ is independent of $V$ and coincides with $\Ass_n$,
  where $\Ass$ is the associative operad.  In the limit $\dim V \rightarrow
  \infty$, the operad structure of $\Quant_n(V)$ stabilizes to that of
  $\Ass_n$ (as does $\SC_n(V)$ and $\Inv_n(V)$), or more precisely to
  its associated graded with respect to the PBW filtration. However,
  for fixed $V$, $\Quant_n(V)$ retains a complicated algebraic
  structure, even though the $S_{n+1}$-structure is simple.
\end{remark}
 We can give a modified example of an element of $\Inv_n(V)$ in smaller degree, $\dim V + 3$:
\begin{example}\label{modmatexam} 
Consider the element of $\Inv_{\dim V+3}(V)_{\dim V +6}$ sending $f_1 \otimes \cdots \otimes f_{\dim V +4}$ to the complete skew-symmetrization of 
\begin{equation}\label{modmateqn}
\pi^{1,2} \pi^{1,3} \pi^{2,3}(f_1 \otimes f_2 \otimes f_3) \{f_4, f_5\} \{f_6, f_7\} \cdots \{f_{\dim V +2}, f_{\dim V +3}\},
\end{equation}
where $\pi$ is the Poisson bivector.  This generates a copy of
$\wedge^{\dim V+3} \hh \oplus \wedge^{\dim V+4} \hh$, and the sign
representation cannot lie in $\Quant_{\dim V+3}(V)_{\dim V+6}$.  Hence,
$\Quant_{\dim V+3}(V) \neq \Inv_{\dim V+3}(V)$, and moreover,
$\Quant(V)_{\dim V+6} \neq \Inv(V)_{\dim V+6}$.  Below, we will
show that the second inequality is sharp.  
\end{example}
\begin{example}
If we consider \eqref{modmateqn} as a distribution, we obtain an
element of $\Inv_{\dim V + 2}(V)_{\dim V + 6}$. We claim that this is always
the sign representation of $\Quant_{\dim V + 2}(V)$.
Indeed, the latter is spanned by
 the first nonvanishing derivative in $\hbar$ of
\begin{equation}\label{matalteqn}
\Alt([f_1, f_2]_\star \star [f_3, f_4]_\star \cdots
\star [f_{\dim V + 1}, f_{\dim V + 2}]_\star),
\end{equation}
(cf.~\cite[Corollary 8.10]{Matso}), where $[f,g]_\star := f \star g -
g \star f$, and $\Alt$ denotes the total skew-symmetrization.  The
resulting element will act with positive order in all components.  If
we restrict to the case where $f_1, \ldots, f_{\dim V}$ are linear, it
is not difficult to see that the first nonvanishing derivative of
\eqref{matalteqn} and the operator obtained from viewing
\eqref{modmateqn} as a distribution coincide, up to a nonzero scalar.
Then, it is also not difficult to check that there can be no nonzero
elements of $\Inv_{\dim V + 2}(V)_{\dim V + 6}$ which act by zero when
$f_1, \ldots, f_{\dim V}$ are linear (using an explicit analysis along
the lines of \S \ref{scinvordthmpfsec}).

This operator is not a Poisson polynomial, since the total
skew-symmetrization of 
\begin{equation}
\{f_1, f_2\} \cdots \{f_{\dim V + 1}, f_{\dim
  V + 2}\}
\end{equation}
vanishes, so lies in the kernel of the surjection $P_{\dim
  V + 2} \onto \SC_{\dim V + 2}(V)$, but it is the only copy of the
sign representation of $S_{\dim V + 2}$ in $P_{\dim V + 2} \cong
\CC[S_{\dim V + 2}]$. 
\end{example}
Next, we explain how the above can be modified to provide instances of all
height $\geq \dim V + 3$ irreducible representations of the quotient $\Inv_n(V)/\Quant_n(V)$ for sufficiently large $n$, and also height $\geq \dim V + 2$ irreducible subrepresentations of $\Quant_n(V) / \SC_n(V)$.
\begin{example} \label{producehhkexam} Recall from Proposition
  \ref{fisxprop} that the height $k+1$ part of $\Inv_n(V)$, as an
  $S_{n+1}$ representation, is spanned by operators in $\Inv_k(V)^v$,
  for any fixed $v \in V$.  Let $\rho \subseteq \Inv_k(V)^v$ be an
  irreducible $S_{k+1}$-subrepresentation, so that $\rho[n+1]$ is an
  irreducible subrepresentation of $\Inv_n(V)$ for sufficiently large
  $n$.  Then, the span under the polynomial algebra $D_{v,k} =
  \C[\partial_v^{(1)}, \ldots, \partial_v^{(k)}]$, which is isomorphic
  to $\Sym \hh$ as an $S_{k+1}$-representation, is $\rho \otimes \Sym
  \hh \subseteq \Inv_k(V)^v$.  This contains all irreducible
  $S_{k+1}$-representations, and so for sufficiently large $n$, their
  image in $\Inv_k(V)$ yields all irreducible
  $S_{k+1}$-representations.

This construction works equally well for $\Quant$ and $\SC$, as 
well as for the quotients $\Inv/\Quant$ and $\Quant/\SC$.  Thus, given a single height $k$ subrepresentation of such a quotient, one can obtain all height $k$ subrepresentations, for sufficiently large degree $n$.
\end{example}
These examples allow us to refine the converse direction to Theorem
\ref{scinvisotthm}:
\begin{corollary}\label{classisotcor}
Let $\rho$ be an irreducible representation of $S_k$.
\begin{enumerate}
\item[(a)] For sufficiently large $n$, $\rho[n+1]$
occurs in the quotient $\Quant_n(V)/\SC_n(V)$ if and only if $\rho$ has height
$\geq \dim V +  2$.
\item[(b)] For sufficiently large $n$, $\rho[n+1]$ occurs
in the quotient $\Inv_n(V)/\Quant_n(V)$ if $\rho$ has height
$\geq \dim V + 3$, and does not occur (for any $n$)
if $\rho$ has height $\leq \dim V + 1$.
\end{enumerate}
\end{corollary}
This leaves open
\begin{question} \label{dimvp2ques} Do any
  $S_{n+1}$-subrepresentations of height $\dim V + 2$ occur in
  $\Inv_n(V)/\Quant_n(V)$?  
\end{question}
For $V = \CC^2$, the answer is negative,
  as a consequence of Theorem \ref{ht4thm}.

Next, we add more details to Theorem
\ref{scinvordconvthm}. Recall that, for $2m \leq \dim V+4$, $\SC(V)_{2m} =
  \Quant(V)_{2m} = \Inv(V)_{2m}$.
\begin{corollary}\label{scquantinvordcor}
For $2m \geq \dim V + 6$, $\SC(V)_{2m}
  \subsetneq \Quant(V)_{2m} \subsetneq \Inv(V)_{2m}$, and in
  particular, for such $m$,
\begin{gather}\label{scordineq}
\SC_{n}(V)_{2m}
  \subsetneq \Quant_{n}(V)_{2m}, \quad \forall n \geq \frac{1}{2} \dim V + m-1, \\ \label{quantordineq}
\Quant_{n}(V)_{2m} \subsetneq
  \Inv_{n}(V)_{2m}, \quad \forall n \geq \frac{1}{2} \dim V + m,
\end{gather}
and for large $m$ the inequalities can be strengthened to only require
$n$ to grow polynomially of degree $1 - \frac{1}{\dim V+1}$ in $m$,
rather than linearly.
\end{corollary}
This answers the question, ``For which orders does one have equalities
of $\SC, \Quant$, and $\Inv$?''  On the question of degree,
\cite[Theorem 7.5]{Matso} answers the question for $\SC$: one has
$\SC_n(V) = \Quant_n(V) = \Inv_n(V)$ if $n \leq
\dim V$, and $\SC_n(V) = \Quant_n(V)$ if and only if $n \leq
\dim V+1$.  This only leaves the question of which values of $n$ one has
$\Quant_n(V) = \Inv_n(V)$.  From Example \ref{modmatexam} and
Theorem \ref{scinvisotthm}, we immediately deduce
\begin{corollary}\label{scquantinvdegcor}
$\Quant_n(V) = \Inv_n(V)$ for $n \leq \dim V + 1$ and $\Quant_n(V) \neq \Inv_n(V)$ for $n \geq \dim V + 3$ (with inequality holding in
all orders $\geq \dim V+6$).  Finally,
 $\Inv_{\dim V + 2}(V) / \Quant_{\dim V+2}(V)$
is isomorphic to a direct sum of sign representations of $S_{\dim V+3}$ (possibly none).
\end{corollary}
One is therefore led to ask
\begin{question}\label{dimvp2ques2}
  Does one have $\Quant_{\dim V+2}(V) = \Inv_{\dim V+2}(V)$?
  Equivalently, does the sign representation of $S_{n+1}$ (for $n=\dim
  V + 2$) occur only once in $\Inv_{\dim V+2}(V)$?  If not, what is
  the Hilbert series of the quotient?  
\end{question}
For $\dim V = 2$, by the results of \S \ref{lowdimsec} (or
alternatively by Theorem \ref{ht4thm}), the answer is positive.  A
negative answer to Question \ref{dimvp2ques} in particular would imply
a positive answer to Question \ref{dimvp2ques2}.

\subsection{Stability results on $\Inv, \Quant$, and $\SC$ for fixed order and for leading heights}\label{asympsec}
Here we explain how the parts of $\Inv, \Quant$, and $\SC$ stabilize in two senses: for fixed order $m$, as the degree $n \rightarrow \infty$, or for fixed $\ell \geq 0$, in heights $2m-\ell, 2m-\ell+1, \ldots, 2m$, as the order $2m$ as well as the degree $n$ goes to infinity.
\begin{theorem}\label{snasympthm1}
  Fix a finite-dimensional symplectic vector space $V$ of dimension $2d$.
\begin{enumerate}
 \item[(i)] Fix an order $2m
  \geq 0$.  Then, for all $n \geq 2m$, $\Inv_n(V)_{2m}$ is spanned over $\C[S_n]$ by polydifferential operators of the form
\begin{equation} \label{snasympthm1eq1}
f_1 \otimes \cdots \otimes f_n \mapsto \phi(f_1 \otimes \cdots \otimes f_{2m}) f_{2m+1} \cdots f_n,
\end{equation} 
for $\phi \in \Inv_{2m}(V)_{2m}$.
\item[(i')] As an $S_n$-representation,
\begin{equation}\label{snasympthm1eq2}
\Inv_{n}(V)_{2m} \cong \bigoplus_i \Ind_{S_{k_i} \times S_{n-k_i}}^{S_n} (\rho_{k_i} \boxtimes \CC),
\end{equation}
for fixed $k_i \leq 2m$ and $\rho_{k_i} \in \Rep S_{k_i}$, independent of $n$.
In particular, $\dim \Inv_n(V)_{2m}$ is a polynomial in $n$ of degree $2m$.
\item[(ii)]  Fix $\ell \geq 0$. For $m > 2\ell$, then $\Inv_n(V)_{2m}$ is spanned over $\CC[S_{n}]$ by operators of the form
\begin{equation}\label{snasympthm1eq3}
f_1 \otimes \cdots \otimes f_n\mapsto \psi(f_1 \otimes \cdots \otimes f_{n-2})
 \{f_{n-1}, f_n\},
\end{equation}
up to an $S_n$-subrepresentation of dimension $< C \cdot n^{2m-\ell-1}$,
for some $C > 0$ independent of $n$.  Moreover, if $m > \ell(3d+1)$,
then we may take $\psi \in \Inv_{n-2}(V)_{2m-2}$ to be invariant.  
\item[(iii)] In the situation of (ii), when $d = 1$, to take $\psi \in
  \Inv_{n-2}(V)_{2m-2}$, it suffices to have $m > 3\ell$. For such $m$,
   the coefficients of the polynomial $\dim \Inv_n(V)_{2m}$ are
  of the form $\frac{P(m)}{(m+a)!(m+b)!}$ for $a, b$ integers and $P(m)$
  a polynomial with rational coefficients.
\end{enumerate}
Moreover, all of the above results are also valid replacing $\Inv$ with $\Quant$ or $\SC$.
\end{theorem}
Part (i) of the theorem shows that the $S_{n+1}$-structure of
$\Inv_n(V)_{2m}$ essentially depends only on $m$, and not on $n$ (for
large enough $n$), and part (ii) shows that, for each fixed $\ell \geq
0$, this doesn't even depend on $m$ when $m$ is sufficiently large, up
to a subrepresentation of dimension $< C n^{2m-\ell-1}$. That is,
roughly, the decomposition of $\Inv_n(V)_{2m}$ in terms of irreducible
$S_{n+1}$-subrepresentations stabilizes as $n \rightarrow \infty$, and
this decomposition itself stabilizes as $m \rightarrow \infty$.
Together with Figure \ref{hilbsertable} and \eqref{asympdimeqn} below, we
immediately deduce
\begin{corollary} 
\begin{gather} 
\label{m4scdim}
\dim \SC_n(\CC^2)_8 = \frac{1}{4}{n+1 \choose 6}(n^2 + 9n + 26),
\\
\label{m4quantdim}
\dim \Quant_n(\CC^2)_8 = \frac{1}{24}{n \choose 5}(n+5)(n^2+5n+10),
\\
\label{m4invdim}
\dim \Inv_n(\CC^2)_8 = \frac{1}{24} {n+1 \choose 5}(n^3 + 5n^2 -10n -80).
\end{gather}
\end{corollary}
We remark that the Hilbert series evaluated at $n=0$ is $1$.
 Moreover, it appears that
the Hilbert series at $n=-1$ is meaningful, at least
for $\SC$ and $\Inv$: at least through degree $8$, one obtains
$1 - t^2$, since the polynomials
for
$\Inv_n(\CC^2)_{2m}$ at $m=2,3,4$ have a root at $n=-1$; it might be
interesting to see if they have this root for all $m \geq 2$ (only for
$V = \CC^2$, though). This says that the number of copies of $\wedge^i \hh$
for even $i$ is the same as the number of copies for odd $i$, when $n$ is large.
\begin{question} Is it possible to improve the
  inequality $m \geq 3\ell$ for the leading $\ell+1$ terms of the
  polynomial $\dim \Inv_n(\CC^2)$ of $n \geq 3m-3$ to have fixed
  coefficients of the form $\frac{P(m)}{(m+a)!(m+b)!}$?  As we see
  above, when $\ell = 1$, we can take all $m$, and for $\ell=2$, we
  need only $m \geq 2$, not $m \geq 6=3\ell$.  Note though, as in
  the case where $m$ is fixed and $n < 3m$, the
  structure as a representation of $S_{n+1}$ does vary when $m <
  3\ell$; cf.~Corollary \ref{youngcor} and Remark \ref{sharprem}.
\end{question}

The above theorem allows us to state and prove
asymptotic results about this decomposition.  In the leading three heights, by
\S \ref{tophtsubsec}, 
$\Inv$, $\Quant$, and $\SC$ all coincide. We elaborate on those results:
\begin{corollary}\label{snasympcor}
 Let $V$ be a symplectic vector space of dimension $2d$, and let $m \geq 2$.
\begin{enumerate}
\item[(i)] As virtual $S_{n+1}$-representations,
\begin{equation}\label{invasympeqn}
\Inv_n(V)_{2m} = \Sym^{2m}(V^n)^{\sp(V)} \ominus \Hom_{\sp(V)}(\Sym^3 V, \Sym^{2m-1}(V^n)) \oplus \Sym^{2m-2}(V^n)^{\sp(V)} \oplus R_{m,d},
\end{equation}
where $R_{m,d}$ is an honest representation, with the property that
$\dim R_{m,d} < C \cdot n^{2m-2}$ for some $C > 0$ independent of $m$ and $V$. 
\item[(ii)] When $d = 1$, in fact we can take $\dim R_{m,d} < C' \cdot n^{2m-3}$ for some $C' > 0$, and we may rewrite the RHS by substituting 
$V = \hh \oplus \hh^*$:
\begin{equation}
\Hom_{\sp(V)}(\Sym^b V, \Sym^a(V^n)) \cong (\Sym^{\frac{a-b}{2}} \hh \otimes \Sym^{\frac{a+b}{2}} \hh) \ominus (\Sym^{\frac{a-b}{2}-1} \hh \otimes \Sym^{\frac{a+b}{2}+1} \hh).
\end{equation}
In particular, for all $m \geq 2$,
\begin{multline} \label{asympdimeqn}
  \dim 
\Inv_n(\CC^2)_{2m} 
= \frac{1}{m! (m+1)!} n^{2m} +\frac{m(m-4)}{(m-1)!(m+2)!} n^{2m-1} 
\\ + \frac{3m^4-32m^3+92m^2-145m+118}{6(m+2)!(m-1)!} n^{2m-2} + O(n^{2m-3}),
\end{multline}
where $O(n^{2m-3})$ denotes a fixed
polynomial in $n$ of degree $\leq 2m-3$.
\end{enumerate}
Moreover, all of the above results are also true if we replace $\Inv$ with $\Quant$ or $\SC$.
\end{corollary}

We can use these results to describe the decomposition into
irreducible $S_{n+1}$-subrepresentations, following the idea behind
Deligne's category $\Rep S_n$ of representations of $S_n$, as
discussed in the previous subsection:
\begin{corollary}\label{youngcor} \begin{enumerate}
\item[(i)]
For $n > 3m-1$, the Young diagrams corresponding to
  the decomposition of $\Inv_n(V)_{2m}$ into irreducible
  representations of $S_{n+1}$ are obtained from those for $n = 3m-1$ by
  adding $n-(3m-1)$ cells to the top row of each diagram. 

That is, the families
of irreducible representations of $S_{n+1}$ composing $\Inv_n(V)_{2m}$ 
all begin at or before $n = 3m-1$.
\item[(ii)] If
$m \geq \ell(3d+1)$ and $n \geq 3m-1$, the truncated Young diagrams
  corresponding to irreducible $S_{n+1}$-subrepresentations of
  $\Inv_{n}(V)_{2m}$ of height $\geq 2m-\ell$
are obtained, with some repetition, from those in order $2(3d+1)\ell$
by adding $2(m-(3d+1)\ell)$ cells total to the first
$2d$ rows, with at most $m-(3d+1)\ell$ added to each row.
\item[(iii)] If $n \geq 3m-1$, $V=\CC^2$, and $m \geq 3\ell$, then the
  truncated Young diagrams corresponding to irreducible
  $S_{n+1}$-subrepresentations of $\Inv_n(V)_{2m}$ of height 
  $\geq 2m-\ell$ are obtained, with multiplicity,
  from those in order $6\ell$
 by adding $m-3\ell$ cells to the
  top two rows.
\end{enumerate}
\end{corollary}
In these terms, Corollary \ref{snasympcor} yields the families of
irreducible representations of dimension polynomial of degree $\geq
2m-2$.  In the case $V = \CC^2$, this explains Corollary \ref{c2sslcor}:
one may verify (using the above theorems, or
following the proof of \S \ref{scinvordthmpfsec}; we also
double-checked with the help of the computational data underlying \S
\ref{lowdimsec}) that the truncated Young diagrams for $d=1$ with $2m,
2m-1$, or $2m-2$ cells are those obtained by adding an equal number of
cells to the top two rows of Figures
\ref{firstyoung}--\ref{lastyoung}. In particular, for $m \geq 6$,
there are exactly $12$ irreducible subrepresentations of
$\Inv_n(\CC^2)_{2m}$ for $n \geq 3m-1$ occurring in families of dimension
a polynomial in $n$ of degree $\geq 2m-2$.

\begin{figure}
\centerline{  
(empty diagram)
}
\caption{$2m$ cells\label{firstyoung}}
\end{figure}
\begin{figure}
\begin{minipage}[t]{3in}
\vspace{0in}
\centerline{\begin{Young}
&\cr\cr
\end{Young}}
\caption{$2m-1$ cells \label{young2}} 
\end{minipage}
\begin{minipage}[t]{3in}
\vspace{0in}
\centerline{\begin{Young}
&&\cr\cr\cr
\end{Young}}
\caption{$2m-1$ cells\label{young3}}
\end{minipage}
\end{figure}
\begin{figure}
\centerline{\begin{Young}
&\cr
\end{Young}}
\caption{$2m-2$ cells\label{young4}}
\end{figure}
\begin{figure}
\begin{minipage}[t]{2.5in}
\vspace{0in}
\centerline{\begin{Young}
&&\cr\cr
\end{Young}}
\caption{$2m-2$ cells\label{young5}}
\end{minipage}
\begin{minipage}[t]{2.5in}
\vspace{0in}
\centerline{$\ $ \begin{Young} 
&\cr\cr\cr
\end{Young}}
\caption{$2m-2$ cells\label{young6}}
\end{minipage}
\end{figure}
\begin{figure}
\begin{minipage}[t]{2.5in}
\vspace{0in}
\centerline{\begin{Young}
&&&\cr&\cr
\end{Young}}
\caption{$2m-2$ cells\label{young7}}
\end{minipage}
\begin{minipage}[t]{2.5in}
\vspace{0in}
\centerline{\begin{Young}
&&&\cr\cr\cr
\end{Young}}
\caption{$2m-2$ cells\label{young8}}
\end{minipage}
\begin{minipage}[t]{2.5in}
\vspace{0in}
\centerline{\begin{Young}
&\cr&\cr&\cr
\end{Young}}
\caption{$2m-2$ cells\label{young9}}
\end{minipage}
\begin{minipage}[t]{2.5in}
\vspace{0in}
\centerline{\begin{Young}
&&\cr\cr\cr\cr
\end{Young}}
\caption{$2m-2$ cells\label{young10}}
\end{minipage}
\end{figure}
\begin{figure}
\begin{minipage}[t]{2.5in}
\vspace{0in}
\centerline{$\ $ \begin{Young} 
&&&&\cr&\cr\cr
\end{Young}}
\caption{$2m-2$ cells\label{young11}}
\end{minipage}
\begin{minipage}[t]{2.5in}
\vspace{0in}
\centerline{$\ $ \begin{Young} 
&&&&&\cr&\cr&\cr
\end{Young}}
\caption{$2m-2$ cells\label{lastyoung}}
\end{minipage}
\end{figure}
\begin{remark}\label{sharprem} 
  The inequalities for $n$ and $m$ in the results above are sharp.  We
  give examples here to show this.  For the Young diagram results of
  Corollary \ref{youngcor}, in part (i), notice that for $n=3m-1$, one
  has the Young diagram occurring in $\Inv_n(V)_{2m}$ with three rows
  all having $m$ cells; whenever $n \geq 3m$, one then has the diagram
  obtained from this one by adding additional cells to the first
  row. For part (iii), notice that the truncated Young diagrams in
  Figures \ref{firstyoung}, \ref{young3}, and \ref{lastyoung}
  correspond to families of subrepresentations of
  $\Inv_n(\CC^2)_{6\ell}$, where $m = 3 \ell$ (i.e., the degree in $n$
  of the dimension is $2m-\ell=5\ell$) and the corresponding subspaces
  of $\Inv_n(\CC^2)_{6\ell}$ are not obtainable from
  $\Inv_n(\CC^2)_{6\ell-2}$ by multiplying by $f \otimes g \mapsto
  \{f, g\}$, i.e., by adding a cell to the first two rows of the
  truncated Young diagram.  For arbitrary $\ell$, the truncated Young
  diagram with $\ell$ columns of length $3$, and $2\ell$ columns of
  length $1$ occurs. Similarly, for part (ii), for $d \geq 2$, one has
  the truncated Young diagram with $\ell$ columns of length $2d+1$,
  and $2\ell$ columns of length $2d$.

  For Theorem \ref{snasympthm1}.(i) note that $f_1 \otimes \cdots
  \otimes f_{2m} \mapsto \{f_1, f_2\} \cdots \{f_{2m-1},f_{2m}\}$ is
  in $\Inv_{n}(V)_{2m}$ for $n=2m$, and is not obtainable from
  $\Inv_{2m-1}(V)_{2m}$ by multiplying by the $2m$-th function.  Parts
  (ii)--(iii) of the theorem have the same sharpness examples as in
  the corollary.
\end{remark}

\section{Proof of Theorem \ref{isotpartthm}}\label{s:isopf}
\label{isotpartpfsec}
We follow the reasoning of Example \ref{fisxexam}.

First we prove parts (i) and (ii).  Note that $\mathfrak{h} \oplus \CC
\cong \CC^{n+1}$ is the standard representation of $S_{n+1}$.  Next,
for any representation $R$ of $S_{n+1}$, the fact that $\mathfrak{h}
\oplus \CC \cong \Ind_{S_n}^{S_{n+1}} \CC$ implies that
$\Hom_{S_{n+1}}(\mathfrak{h} \oplus \CC, R) \cong
R^{S_n}$. (Explicitly, for example, $R^{S_n}$ identifies with the
image of the vector $(1,1,\ldots,1,0) \in \CC^{n+1}$ under
$S_{n+1}$-homomorphisms $\CC^{n+1} \rightarrow R$.)

Thus, it suffices to show that the only polydifferential operators
$\cO_{V}^{\otimes n} \rightarrow \cO_V$ which are
invariant under formal symplectic automorphisms and symmetric in the inputs
are the constant multiples of the multiplication operator.  In other
words, we need to show that polydifferential operators $\Sym^n
\cO_V \rightarrow \cO_V$ which are invariant under
formal symplectic automorphisms are multiples of $f^{\otimes n} \mapsto
f^n$. This statement is actually 
\cite[Lemma 2.1.8]{hp0weyl}, but we recall the argument for
purposes of self-containment, and since we will generalize it. 
As in Example \ref{fisxexam}, such an operator is determined by its
restriction to elements $f^{\otimes n}$ where $f$ is a function with
nonvanishing first derivative, since $\dim V \geq 2$. Moreover, up to
symplectic automorphism, we can assume that $f$ is a fixed linear
function, i.e., the operator is uniquely determined by its value on only
a single element $f^{\otimes n}$ with $f$ linear. 
In order for the operator to be invariant, the image of
$f^{\otimes n}$ must be a constant multiple of $f^n$.  Thus, the space
of such operators is one-dimensional, and this formula therefore holds
for all (not necessarily linear) functions $f$.

Next, we prove (iii).  We will use the isomorphism
$\Hom_{S_{n+1}}(\wedge^2 \mathfrak{h} \oplus \mathfrak{h}, R) \cong
\Hom_{S_{n-1} \boxtimes S_2}(\C \boxtimes \sgn, R)$, where $\sgn$ is
the sign representation of $S_2$.  This follows from the isomorphism
$\wedge^2 \mathfrak{h} \oplus \mathfrak{h} \cong \wedge^2
(\mathfrak{h} \oplus \C) \cong \Ind_{S_{n-1} \times S_2}^{S_{n+1}}(\C
\boxtimes \sgn)$. In view of part (ii), the $\wedge^2
\mathfrak{h}$-isotypic component of $\HP_0(\cO_{V^{n}}^{S_{n+1}},
\cO_{V^n})^*$ is identified with the space of polydifferential operators
$\Sym^{n-1} \cO_V \otimes \cO_V \rightarrow \cO_V$ which are
skew-symmetric in transposing the components of $\cO_V$ (as
distributions on the formal neighborhood of zero in $V^{n+1}$) and
invariant under formal symplectic automorphisms.  Thus, it suffices to
consider the image of elements of the form $f^{\otimes (n-1)} \otimes
g$. Furthermore, we may again assume that $f'(0) \neq 0$, and by
taking Darboux coordinates, assume that $f = x_1$. The resulting
polydifferential operator must be of the form
\begin{equation}
x_1^{\otimes (n-1)} \otimes g \mapsto D(g),
\end{equation}
where $D$ is a polydifferential operator which is invariant under formal
symplectic automorphisms fixing $x_1$.  Using linearity and invariance under formal
symplectic automorphisms fixing $x_1$ and under rescalings $x_1 \mapsto \alpha x_1, y_1 \mapsto \alpha^{-1} y_1$, we may assume that
\begin{equation}
  D = \sum_{i} \lambda_i x_1^i \frac{\partial^{n-1-i}}{\partial y_1^{n-1-i}},
\end{equation} 
for a choice of Darboux coordinates $x_1, \ldots, x_d, y_1, \ldots, y_d$ (i.e., $\frac{\partial}{\partial y_1}(-) = \{x_1, -\}$), and for some $\lambda_i \in \CC$. In other words, 
\begin{equation}
x_1^{\otimes (n-1)} \otimes g \mapsto \sum_i \lambda_i x_1^i (\frac{\partial}{\partial x_1}(x_1))^{n-1-i} \frac{\partial^{n-1-i}}{\partial y_1^{n-1-i}}(g).
\end{equation}
If we take the transpose of the last component of the input with the output as distributions, applying integration by parts in the variable $y_1$ $n-1-i$ times yields a result which is $(-1)^{n-1-i}$ times the above operation.  Hence, to be skew-symmetric under transposition, we require that $\lambda_{i} = 0$ whenever $n-1-i$ is even.

It is easy to check, conversely, that any operator as above has the desired form. Hence, since the summand as above corresponding to $\lambda_i$ has
degree $2(n-1-i)$, we deduce that
\begin{equation}
h(\Hom_{S_{n+1}}(\wedge^2 \mathfrak{h}, \HP_0(\cO_{V^n}^{S_{n+1}}, \cO_{V^n})^*); t) = t^2 + t^6 + \cdots + t^{2 (2 \lfloor \frac{n}{2} \rfloor - 1)}.
\end{equation}

(iv) This is similar to part (iii): we deduce that the desired operators have exactly the same form, except that we require $\lambda_i = 0$ whenever $n-1-i$ is odd, rather than even.

\section{Proof of theorems from \S \ref{mrsec}}\label{s:mrsecpf}

\subsection{Proof of Theorem \ref{scinvisotthm}}\label{scinvisotpfsec}
Note that polydifferential operators in $\Inv_n(V)$ are in particular
$\sp(V)$-invariant.  By classical invariant theory, it follows that
$\Inv_n(V)$ is spanned by compositions of elements of the form
$\pi^{i,j}$ followed by multiplying the resulting $n$ functions, where
$\pi \in V^{\otimes 2}$ is the Poisson bivector, which acts on
$\cO_V^{\otimes 2}$ by applying the componentwise partial
differentiation, and $\pi^{i,j}$ denotes applying $\pi$ to degrees $i$
and $j$.  Let $\mu: \cO_V^{\otimes n} \rightarrow \cO_V$ be the
multiplication.  By the second fundamental theorem of invariant
theory, the linear dependence between such operations are exactly
multiples of permutations of the relation
\begin{equation}\label{altreln}
\mu \circ \Alt(\pi^{1,2} \pi^{3,4}
\cdots \pi^{\dim V + 1, \dim V+2}) = 0,
\end{equation} where $\Alt$ denotes taking
the complete skew-symmetrization.  

These compositions are automatically invariant under Hamiltonian
vector fields of linear functions on $V$, since $\pi$ is a
constant-coefficient bivector; also, they are invariant under
Hamiltonian vector fields of quadratic functions on $V$, since the Lie
algebra of such vector fields is nothing but $\Sym^2 V \cong \sp(V)$.
Hence, for such a linear combination to yield an element of
$\Inv_n(V)$, it suffices for it to be invariant under Hamiltonian
vector fields of cubic functions on $V$, in view of the following
lemma (cf.~\cite[Lemma 7.1]{Matso}):
\begin{lemma}
  The Poisson algebra $\CC \oplus (\cO_V)_{\geq 2}$ is Poisson
  generated in degree two and three, that is, by $\sp(V) = \Sym^2 V^*$
  and $\Sym^3 V^*$.
\end{lemma}
\begin{proof} 
  Inductively, it suffices to show that $\{(\cO_V)_3, (\cO_V)_m\} =
  (\cO_V)_{m+1}$ for $m \geq 3$.  As a $\sp(V)$-representation,
  $(\cO_V)_m \cong \Sym^m V^*$ is an irreducible representation. Hence
  it suffices to show that $\{(\cO_V)_3, (\cO_V)_m\} \neq 0$, which is
  obvious.
\end{proof}
Moreover, it suffices to be invariant under a single such nontrivial
Hamiltonian vector field, since $\Sym^3 V^*$ is an irreducible
representation of $\sp(V)$.

Now, as in Example \ref{fisxexam}, take any Young diagram $\lambda =
(\lambda_1, \ldots, \lambda_k)$ with $n+1$ cells.  We are interested
in elements $\phi \in \Inv_n(V)$ generating an irreducible
$S_{n+1}$-representation corresponding to $\lambda$. As argued in
Example \ref{fisxexam}, it suffices to assume that the first
$\lambda_1$ inputs correspond to the top row, and to restrict $\phi$
to elements of the form $x^{\otimes \lambda_1} \otimes g_{\lambda_1+1}
\otimes \cdots \otimes g_{n}$, where $x$ is a fixed linear function on
$V$.  Restricting a composition of elements $\pi^{i,j}$ in this way is
nonzero only when the first $\lambda_1$ indices appear at most once
each.  Linear dependences among these compositions are all multiples
of permutations of \eqref{altreln}.  Since \eqref{altreln} is
skew-symmetric in the indices $1$ through $\dim V + 2$ and has order
one in each of these, any multiple which is still of first order in
two indices $i,j \leq \dim V + 2$ must be skew-symmetric in those
indices.  Since we are interested in applying operators to $x^{\otimes
  \lambda_1} \otimes g_{\lambda_1+1} \otimes \cdots \otimes g_{n}$,
any operator which is skew-symmetric in two of the first $\lambda_1$
indices acts trivially.  Hence, the only linear dependences we need to
consider are multiples of those permutations of \eqref{altreln} in
which at most one of the first $\lambda_1$ indices appears.  In the
case that $n < \lambda_1 + (\dim V + 1)$,
there are no such permutations, and thus the compositions we are considering
are linearly independent.  

We claim that, for $\dim V > n - \lambda_1-1$, the multiplicity of the
$\lambda$-isotypic part of $\Inv_n(V)$ is independent of $V$.  In view
of the above, we need to compute the space of linear combinations of
compositions of elements $\pi^{i,j}$ restricted to $x^{\otimes
  \lambda_1} \otimes -$, where, in the composition, the first
$\lambda_1$ indices appear at most once each, which are killed by the
infinitesimal action of the Hamiltonian vector field of the function
$x^3 \in \cO_V$. In Darboux coordinates $(x=x_1, x_2, \ldots, x_d,
y_1, y_2, \ldots, y_d)$ such that $\{x_i, y_i\} = 1$, we can write
this Hamiltonian vector field as $3x_1^2 \pd{}{y_1}$, and $\pi =
\sum_{i=1}^{d} \pd{}{x_i} \wedge \pd{}{y_i}$.  It is then easy to see
that the result of applying the Hamiltonian vector field is an
expression which does not essentially depend on $V$, and the condition
for it to vanish is independent of $V$.

Therefore, by letting the dimension of $V$ go to infinity, it follows
that the multiplicity of the $\lambda$-isotypic part of $\Inv_n(V)$ is
the same as its multiplicity in $\Quant_n(V)$, i.e., in
$\Ind_{\ZZ/(n+1)}^{S_{n+1}} \CC$.  This must then be true whenever
$\dim V > n - \lambda_1 -1$.  Moreover, as $\dim V \rightarrow
\infty$, we obtain that $\Inv_n(V) = \SC_n(V)$, and as a result the
desired linear combinations of compositions of $\pi^{i,j}$ which are
invariant must in fact be given by Poisson polynomials.  That is, the
$\lambda$-isotypic part of $\Inv_n(V)$ is the same as that of
$\SC_n(V)$. This holds for all $\dim V > n - \lambda_1-1$, and the
result is independent of such $V$. Rewriting this, we require that
$(n+1)-\lambda_1 \leq \dim V + 1$, i.e., that there are at most $\dim
V + 1$ cells below the first row. Equivalently, the irreducible
representation is a summand of $\bigoplus_{i \leq \dim V + 1}
\hh^{\otimes i}$, as desired. Moreover, since the result was
independent of $\dim V$, there can be no relations between any of the
nonidentical Poisson monomials, and hence the space identifies with
the same isotypic part of the space of abstract Poisson polynomials,
with grading given by the number of brackets.

\subsection{Proof of Theorem \ref{polythm}, Proposition \ref{fisxprop}, and Corollary \ref{toporddelcor}}
The core of these results is Theorem \ref{polythm}. To prove it, we
will use a modification of an argument from \cite[\S 4--5]{Matso}.  We
first need to recall some material from \emph{op. cit.} relating
$\Inv_n(\CC^2)$ to harmonic polynomials.
\subsubsection{Harmonic polynomials}
 \begin{definition} For values $a = (a_1, \ldots, a_n) \in \C^n$,
   define the space of \emph{harmonic polynomials}, $\Har_{n,a}
   \subseteq \C[x_1, \ldots, x_n]$, as the space of solutions of the
   differential equations
\begin{equation} \label{hareqn}
\sum_i a_i \frac{\partial^r}{\partial x_i^r} (\phi) = 0, \forall r \geq 1.
\end{equation}
\end{definition}
\begin{definition} Define the \emph{generic harmonic
    polynomials}, 
$\Har_n \subseteq
  \C[x_1,\ldots,x_n,a_1,\ldots,a_n]$, as the vector space of solutions $\phi$
  of the differential equations \eqref{hareqn}, viewing the $a_i$ as
  variables.
\end{definition}
\begin{definition}
  Let $\Har_{n,a}^k \subseteq \Har_{n,a}$ denote the homogeneous
  harmonic polynomials of degree $k$ in $x_1,\ldots, x_n$. Let
  $\Har_n^{k,\ell} \subseteq \Har_n$ denote the subspace of generic
  harmonic polynomials of degree $k$ in $x_1, \ldots, x_n$ and degree
  $\ell$ in $a_1, \ldots, a_n$.
\end{definition}
The connection with $\Inv_n$ is provided by the following, whose proof
we briefly recall:
\begin{theorem}\cite[Theorem 5.3]{Matso} There is a canonical isomorphism
$\Inv_n(\CC^2)_{2m} \cong \Har_n^{m,m}$. \label{invharthm}
\end{theorem}
\begin{proof}
  Set $V = \C^2 = \Span(x,a)$, with symplectic form $\omega(x,a) = 1$.
  The polynomial algebra $\Sym (V^n) = \C[x_1,\ldots,x_n,a_1,\ldots,
  a_n]$ identifies with constant-coefficient polydifferential
  operators of degree $n$ on $V$. Under this identification,
  $\Inv_n(V) \subseteq \Sym(V^n)$ is the subspace annihilated by the
  adjoint action of Hamiltonian vector fields on $V$.  As before, the
  invariance under the action of Hamiltonian vector fields of linear
  functions is automatic (it restricted us to constant-coefficient
  operators).  Next, if $x^*, a^*$ is a dual basis to $x, a$, then
  \eqref{hareqn} states that the adjoint action of the Hamiltonian
  vector field of $(x^*)^r$ annihilates $\phi$ for all $r \geq
  2$. Note that the Lie algebra $\Sym^2 V^*$ of Hamiltonian vector
  fields of quadratic functions is canonically identified with
  $\sp(V)$, and this preserves the canonical adjoint
  actions. Furthermore, $\Sym^k V^*$ is an irreducible representation
  of $\sp(V)$. Hence, $\sp(V)$-invariant solutions to \eqref{hareqn}
  correspond to elements of $\Inv_n(V)$.  Finally, in order to be
  $\sp(V)$-invariant, it suffices to be invariant under the
  Hamiltonian vector fields of $(x^*)^2$ and $x^* a^*$, i.e., to be a
  solution of \eqref{hareqn} for $r=1$ and also to have equal degree
  in the $x_1, \ldots, x_n$ as in the $a_1, \ldots, a_n$.
\end{proof}
Finally, we will need the following result:
\begin{proposition}\cite[Proposition 4.1]{Matso} \label{matharprop} Assume that $\sum_{i \in I} a_i \neq 0$ for all subsets $I \subseteq \{1,\ldots,n\}$. Then, the Hilbert series of $\Har_{n,a}$ is
\begin{equation}
h(\Har_{n,a};t) = \frac{\prod_{i=2}^{n} (1-t^i)}{(1-t)^{n-1}}.
\end{equation}
\end{proposition}
The proof uses the Koszul complex for the ideal in
$\C[y_1,\ldots,y_n]$ generated by $\sum_{i =1}^n a_i y_i^r$, for $r
\geq 1$ (which is a complete intersection ideal under the premises for
$a_i$); see \cite{Matso} for details.
\begin{corollary}\cite[\S$\!$\S 4--5]{Matso} The top degree of $\Har_{n,a}$ is
  ${n \choose 2}$, and its dimension is $n!$. In particular,
  $\Inv_n(\CC^2)_{2m} = 0$ for $m > {n \choose 2}$.
\end{corollary}
We will only need the fact that the top degree of $\Har_{n,a}$ is ${n \choose 2}$.
\subsubsection{Proof of Theorem \ref{polythm}, Proposition \ref{fisxprop}, 
and Corollary \ref{toporddelcor}}
We prove Proposition \ref{fisxprop} first, since we will need it in the proof
of the theorem.
\begin{proof}[Proof of Proposition \ref{fisxprop}]
This is essentially a formalization of the observation of Example
\ref{fisxexam}. Given $\psi \in \Inv_{k,\ell}(V)^v$, we
can produce the element $\phi \in \Inv_n(V)$ uniquely determined by
the condition
\begin{equation}\label{toinveqn}
w^{\otimes (n-k)} \otimes g_1 \otimes \cdots \otimes g_k \mapsto
w^{n-k-\ell} \psi(g_1 \otimes \cdots \otimes g_k), 
\end{equation}
for all $g_1, \ldots, g_k \in \cO_V$.  Since the result is symmetric
in the first $n-k$ components, it spans height $\leq k$
representations of $S_n$, and height $\leq k+1$ representations of
$S_{n+1}$.  Moreover, the above map is evidently a $S_{k+1}$-linear
map.

Conversely, take an arbitrary element $\phi \in \Inv_n(V)$ which
generates a height $k' \leq k$ irreducible representation of $S_n$,
which is Young symmetrized with top row corresponding to components
$1, 2, \ldots, n-k'$. Let $w \in V^* \cong V$ be the element
corresponding to $v \in V$, i.e., $v(f) = \{w, f\}$ for all $f \in
\cO_V$. Then, as explained in Example \ref{fisxexam}, $\phi$ restricts
uniquely to an element of $\Inv_k(V)^v$ by plugging $w$ into the first
$n-k$ components.  This would have weight $n-k$; to obtain an operator
of weight $\ell$ instead, we can plug in $w^{\otimes \ell} \otimes
1^{\otimes n-k-\ell}$.  The resulting operator $\psi \in
\Inv_{k,\ell}(V)^v$ yields $\phi$ under the construction
\eqref{toinveqn}.  Similarly, if $\phi$ generates a height $k'+1 \leq
k+1$ irreducible representation of $S_{n+1}$, the same procedure
yields an element $\psi \in \Inv_k(V)^v$ which recovers $\phi$ under
\eqref{toinveqn}.
\end{proof}

\begin{proof}[Proof of Theorem \ref{polythm}]
  (i) Because the space of constant-coefficient polydifferential
  operators of degree $k$ having order $\leq N$ is finite-dimensional
  for all $N$, it suffices to prove the second statement (the bound on
  the difference between order and weight).  Then, the desired module
  $\Inv_k(V)^v$ will be a submodule of the finitely-generated module
  over $D_{v,k}$ generated by all differential operators of order
  $\leq k(k-1)$, which form a finite-dimensional vector space. Since
  $D_{v,k}$ is noetherian, this will imply that $\Inv_k(V)^v$ is
  finitely-generated.

  Moreover, we claim that it is enough to prove the bound $k(k-1)$ of
  order minus weight in the case that $V = \CC^2$.  Indeed, if for
  arbitrary $V$, there is an element $\phi \in \Inv_{k,\ell}(V)^v_r$
  whose order minus weight, $N := r-\ell$, exceeds $k(k-1)$, then
  there must exist an element $w \in V$ such that, for some $i_1,
  \ldots, i_k \geq 0$ with $i_1 + \cdots + i_k = N$, the element
  $\bigl(\frac{\partial^{i_1}}{\partial (\partial_w)^{i_1}} \otimes
  \cdots \otimes \frac{\partial^{i_k}}{\partial
    (\partial_w)^{i_k}}\bigr) (\phi) \neq 0$.
  Moreover,
  we can assume that $\langle w, v \rangle \neq 0$. Then, we can
  consider $U := \Span(w,v) \cong \CC^2 \subseteq V$.  In this case,
  since $\phi$ is invariant under symplectic automorphisms of $U$
  fixing $v$, one has $\phi \in \Inv_k(U)^v \otimes \Inv_k(U^\perp)$,
  but by construction it has a nontrivial projection to some subspace
  $\Inv_{k,\ell'}(U)^{v}_{r'} \otimes \Inv_k(U)^\perp$ for which
  $r'-\ell' = N > k(k-1)$, a contradiction.
 
  Therefore, assume that $V = \CC^2$. It suffices to prove that
  $\Inv_{k,\ell}(V)^v_{r} = 0$ when $r - \ell > k(k-1)$, and that
  $\Inv_{k,\ell}(V)^v_{r} \neq 0$ for some $\ell, r$ satisfying $r -
  \ell = k(k-1)$.  Let $v = a \in \Span( a, x ) = V$.  Then,
  $\Inv_k(V)^v$ is the invariant subspace under operators of the form
  \eqref{hareqn}, i.e., $\Inv_k(V)^v$ is isomorphic to the space
  $\Har_k$ of generic harmonic polynomials.  By Proposition
  \ref{matharprop}, the degree in $x$ is at most ${k \choose 2}$.
  Now, if we consider the operator $a \mapsto \lambda a, x \mapsto
  \lambda^{-1} x$, then this acts by $f \mapsto \lambda^{\deg_a f -
    \deg_x f}$.  Hence, the difference between order and weight is
  twice the degree in $x$.  This proves that this difference is at
  most $k(k-1)$, as desired. Moreover, since $\Har_{n,a}$ is nonzero
  in degree ${k \choose 2}$, the difference of $k(k-1)$ is actually
  obtained. This proves the claim.

(ii) In the case that $k \leq \dim V$, it follows that the height
$\leq k+1$ part of $\Inv(V)$ is identical with the height $\leq k+1$
part of the space of abstract Poisson polynomials.  It follows from
Proposition \ref{fisxprop} that $\Inv_k(V)^v$ itself is generated as a
module by the abstract Poisson polynomials of degree $k$.  Since the
resulting module is again the height $\leq k+1$ subspace of the space
of all Poisson polynomials, there can be no relations, i.e., the
module is free. 

In more detail, the height $\leq k+1$ subspace (as an
$S_{n+1}$-representation) of Poisson polynomials of degree $n$ can be
viewed as the space of Poisson polynomials in $f, g_1, \ldots, g_k$
which are linear in each of the $g_i$ and of degree $n-k$ in $f$.
This space decomposes into a direct sum over $r_1, \ldots, r_k \geq 0$
with $r_1 + \cdots + r_k = n-k$, of the span of Poisson polynomials of
the form $P((\ad f)^{r_1}(g_1), (\ad f)^{r_2} g_2, \ldots, (\ad
f)^{r_k} g_k)$.  This naturally identifies the height $\leq k+1$
subspace of $\Inv(V)$ with the free module over $D_{v,k}$ generated by
abstract Poisson polynomials of degree $k$.
\end{proof}
\begin{proof}[Proof of Corollary \ref{toporddelcor}]
  As before, let $\phi \in \Inv_n(V)$ be an operator which spans a
  height $k+1$ representation of $S_{n+1}$. Label the Young diagram
  corresponding to $\phi$ in the standard way, so that the first $n-k$
  inputs correspond to the top row, and the output (or input as a
  distribution) corresponds to the last cell of the bottom row.  We
  assume that $\phi$ is Young symmetrized according to the resulting
  tableau.

  From $\phi$ we obtain $\psi \in \Inv_k(V)^v$ as in the above
  procedure.  By construction \eqref{toinveqn}, the weight can be no
  more than the number of cells in the top row of the associated Young
  diagram, which is $n-k$.  By part (i), the order of $\psi$ can
  therefore be at most $(n-k) + {k \choose 2} \dim V$, and the order
  of $\phi$ at most $2(n-k) + {k \choose 2} \dim V$.

  It remains to see that, for $V = \CC^2$, this order is attained for
  sufficiently large $n$, and moreover that one can obtain $\rho[n+1]$
  in this way for all irreducible representations $\rho \in \Rep S_{k+1}$.

  Fix $k \geq 1$ and take a nonzero element $\psi \in
  \Inv_k(\CC^2)^v$. Let $n = \ell + k$, and apply \eqref{toinveqn}.
  The $S_{n+1}$-span of the result, $\phi$, lies in heights $\leq
  k+1$. Next, we would like to ensure this span contains a height
  $k+1$ representation.  To ensure this, first replace $\psi$ with the
  operator
\begin{equation}
\widetilde \psi(g_1 \otimes \cdots \otimes g_k) = \psi(\partial_a^2 g_1 \otimes
\partial_a^{2} g_2 \otimes \cdots \otimes \partial_a^{2} g_k).
\end{equation}
Now, the resulting operator acts with order $\geq 2$ in all $k$
components.  Applying \eqref{toinveqn} yields an operator $\phi$
acting with order $1$ (by $\partial_x$) in the first $n-k$ components
and order $\geq 2$ in the last $k$ components. Hence, its $S_n$ span
contains a height $k$ representation.  Suppose that $\psi' \neq 0$ is
obtained by Young-symmetrizing $\widetilde \psi$ according to a Young
tableau of height $k$ and size $n$, labeled in the standard way (with
the first $n-k$ components corresponding to the top row).  If we
additionally assume that $n-k > k$, then if we skew-symmetrize the
$(k+1)$-st component with the $(n+1)$-st component, the result is nonzero,
which shows that the $S_{n+1}$-span of $\psi'$ contains a height $k+1$
representation. In fact, the resulting height $k+1$ representations
are the same as the representations of $S_{k+1}$ appearing in the
$S_{k+1}$-linear span of $\psi'$.

Thus, there exists $n$ large enough such that the operator $\phi$
obtained by \eqref{toinveqn} from $\widetilde \psi$ generates a height
$k+1$ representation of $S_{n+1}$.  Now, since all $S_{k+1}$
representations are realized in the algebra $D_{v,k}$, the action of
$D_{v,k}$ takes $\phi$ to all possible $S_{k+1}$-representations. That
is, for all $\rho \in \Rep S_{k+1}$, there is an element in $D_{v,k}
\phi$ which generates $\rho[n']$, for some possibly larger $n' \geq
n$.

Finally, the order of the resulting operator in $\Inv_{n'}(V)$ is
equal to $2(n'-k) + (r-\ell)$, where $r$ and $\ell$ are such that
$\psi \in \Inv_{k,\ell}(V)_r$.  If we took the maximum possible value
of $r-\ell$, namely $k(k-1)$ (guaranteed by Theorem
\ref{polythm}.(i)), then the resulting operator $\phi$ has order
$2(n'-k) + k(k-1)$. Thus, for all $\rho \in \Rep S_{k+1}$, there
exists $n'$ such that $\rho[n']$ is a $S_{n'+1}$-subrepresentation of
$\Inv_n'(V)_{2(n-k) + k(k-1)}$, as desired.
\end{proof}

\subsection{Proof of Theorem \ref{hlthm}, Corollary \ref{liecor}, and Corollary \ref{hot3cor}}\label{hlthmpfsec}
First we prove Theorem \ref{hlthm} followed by Corollary \ref{liecor}. We prove Corollary \ref{hot3cor} in the next subsection.

(i) Fix $k \geq 0$. The only fact about $\Inv_k(V)^v$ we will use is
that it is a finitely-generated module over $D_{v,k}$ (except for the
final statement about the case $|\lambda| \leq \dim V + 1$), so that
the result will hold for any other finitely-generated module. In
particular, we can consider the finitely-generated submodules
$
\sum_{\ell,n}
\iota_{k,\ell,n}^{-1}(\SC_n(V)_{S_{n+1-k}})$ 
and $
\sum_{\ell,n}
\iota_{k,\ell,n}^{-1}(\Quant_n(V)_{S_{n+1-k}})$,
using Proposition
\ref{fisxprop}.
 This yields the
claimed result in the cases of $\SC(V)$ and $\Quant(V)$.

For the rest of the proof of (i), we deal only with the main case of
$\Inv_n(V)$.

The theorem will follow from Theorem \ref{polythm} using the following
three ingredients.  First, we will need the well known character
formula for the polynomial algebra $D_{v,k} = \C[\partial_v^{(1)},
\ldots, \partial_v^{(k)}] \cong \Sym \CC^k$ as a graded representation
of $S_{k+1}$, where $\CC^k$ is the reflection representation:
\begin{equation}
\sum_{r \geq 0} \tr(\sigma)|_{(D_{v,k})_r} s^r = \det (1 - s \sigma)^{-1}|_{\CC^k}.
\end{equation}

Next, we will need the following vanishing result for traces of elements of $S_{k+1}$ on irreducible representations:
\begin{proposition}\cite[Corollary 7.5]{Stansbsc} \label{trvanprop} (see also \cite[Exercise 7.60.(b)]{Stanec2})
Let $\lambda$ be a Young diagram of size $|\lambda| = k+1$ and $\sigma \in S_{k+1}$ an element. Let $\ell_1, \ell_2, \ldots, \ell_j$ be the cycle lengths of $\sigma$, with $\ell_1 + \cdots + \ell_j = k+1$.  Let $h_1, h_2, \ldots, h_{k+1}$ be the hook lengths of $\lambda$. Then, $\tr \sigma|_{\rho_\lambda} = 0$ unless 
\begin{equation}
\prod_{i=1}^j (1-s^{\ell_i}) \mid \prod_{i=1}^{k+1} (1 - s^{h_i}).
\end{equation}
\end{proposition}
Finally, we will need to use Hilbert's syzygy theorem: any
finitely-generated graded module $M$ over the polynomial algebra
$D_{v,k}$ has a unique (up to unique isomorphism) minimal graded
resolution of length at most $k$ by free modules.  In the case that
the module was an $S_k$-equivariant module, uniqueness implies that
the resolution is also $S_k$-equivariant.  Hence, the graded character
of $M$ must be of the form
\begin{equation} \label{syzchareqn}
\sum_{r \geq 0} \tr(\sigma)|_{M_r} s^r = \det (1 - s \sigma)^{-1}|_{\CC^k} Q_M(s),
\end{equation}
where $Q_M(s)$ is a polynomial in $s$. 

Now, specialize to $M = \Inv_k(V)^v$. For every element $\phi \in
\Inv_{k,\ell}(V)^v$ and every $n \geq k+\ell$, let $\phi^{(n)} \in
\Inv_n(V)$ be the element obtained as in \eqref{toinveqn}. Let $N_n
\subseteq \Inv_n(V)$ be the $S_{n+1}$-span of the image of $M$ under
the map $\phi \mapsto \phi^{(n)}$. Further, let $N_{n,2m} \subseteq
\Inv_n(V)_{2m}$ be the part of order $2m$ (note that the image of
$\Inv_{k,\ell}(V)^v_r$ lies in order $2m = \ell + r$).  

Our goal is to describe the height $k+1$ part of the decomposition of
$\Inv_n(V)$ into irreducible $S_{n+1}$ representations. In view of
Proposition \ref{fisxprop}, $N_{n} \subseteq \Inv_n(V)$ is the height
$\leq k+1$ part of $\Inv_n(V)$.  In order to write formulas for the
resulting decomposition, we use the identity
\begin{equation}
\dim \Hom_{S_{n+1}}(\Ind_{S_{n-k} \times S_{k+1}}^{S_{n+1}} (\CC \boxtimes
\rho),
N_n) = \dim \Hom_{S_{k+1}}(\rho, (N_n)_{S_{n-k}}),
\end{equation}
where $U_{S_{n-k}}$ denotes the symmetrization of $U$ in the first
$n-k$ components.  Therefore, we will first describe the structure of
$(N_n)_{S_{n-k}}$, and show that the multiplicities of representations
form rational fractions of the desired form. Then, it will remain only to
replace $\Ind_{S_{n-k} \times S_{k+1}}^{S_{n+1}} (\CC \boxtimes \rho)$ above with its irreducible summand $\rho[n+1]$ of height $k+1$.

Given $m$ and $i$, write $M_{2i,2m} := \Inv_{k,m-i}(V)^v_{m+i}$, so
that $M_{2i} := \bigoplus_m M_{2i,2m}$ is the submodule of
$\Inv_k(V)^v$ of elements for which the difference, $2i$, between the
order and weight is fixed. By Theorem \ref{polythm}, $M=\Inv_k(V)^v$ decomposes
finitely as $M = \bigoplus_{i=0}^{{k \choose 2}} M_{2i}$.

We can resolve each $M_i$ as above, considered with the grading by
$|M_{2i,2m}|=2m$ (the sum of order and weight, or the order of the
image in $\Inv_n(V)$ under $\phi \mapsto \phi^{(n)}$).  Then, for all
$\sigma \in S_{k+1}$,
\begin{equation}\label{preqeqn}
\sum_{m,n} \tr(\sigma)|_{(N_{n,2m})_{S_{n-k}}} s^{m} t^n
= \sum_{m \geq i} \frac{\tr(\sigma)|_{\Inv_{k,m-i}(V)^v_{m+i}} s^{m}t^{k+m-i}}{1-t} =
\sum_{m \geq i} \frac{t^{k-i} \tr(\sigma)|_{M_{2i,2m}} (st)^{m}}{1-t}.
\end{equation}
Since $M_{2i}$ is a finitely-generated graded module with a finite resolution as above, and there are finitely many of these, 
\eqref{syzchareqn} implies that the RHS has the form
\begin{equation}
\frac{Q_{\sigma}(s,t)}{(1-t)\det(1-st \sigma)|_{\CC^k}},  
\end{equation} 
for some polynomials $Q_\sigma(s,t)$.

Next, fix a partition $\lambda$ of size $k+1$.  Then, 
\begin{equation}\label{hlnneqn}
  \sum_{m,n} \dim \Hom_{S_{n+1}}(\Ind_{S_{n-k} \times S_{k+1}}^{S_{n+1}} (\CC \boxtimes
\rho_\lambda), N_{n,2m}) s^m t^n \equiv_\lambda \sum_{\sigma \in S_{k+1}} \frac{\tr(\sigma)|_{\rho_\lambda} Q_{\sigma}(s,t)}{(k+1)!(1-t)\det(1-st \sigma)|_{\CC^k}},
\end{equation}
where $\equiv_\lambda$ means that the difference of the two sides is a polynomial in $s$ and $t$, whose degree in $t$ is less than the minimum $n$ such that $\rho_\lambda[n+1]$ exists. That is, the degree in $t$ is less than $|\lambda|+\lambda_1-1$.

If the cycle decomposition of $\sigma$ has $r_i$ cycles of length $i$
for $i = 1, 2, \ldots, \ell$ (so $r_1 + 2 r_2 + \cdots + \ell r_\ell =
k+1$), then 
\begin{equation}\label{hldetsymeqn}
\det(1-x\sigma)|_{\CC^k} = \frac{1}{1-x}\prod_{i=1}^{\ell} (1-x^{r_i}).
\end{equation}
Then, \eqref{hlnneqn}, \eqref{hldetsymeqn}, and Proposition
\ref{trvanprop} imply that 
\[
\sum_{m,n} \Hom_{S_{n+1}}(\Ind_{S_{k+1}
  \times S_{n-k}}^{S_{n+1}} (\rho_\lambda \boxtimes \CC), N_{n,2m}) s^m t^n
\]
is a rational fraction with denominator $(1-t)(1-st)^{-1}\prod_{i}
(1-(st)^{h_i})$, as desired.

This almost proves the theorem: it remains to replace $\Ind_{S_{k+1}
  \times S_{n-k}}^{S_{n+1}} (\rho_\lambda \boxtimes \CC)$ with its irreducible
summand $\rho_\lambda[n+1]$ of height $k+1$.  Equivalently, we can replace
$N_{n,2m}$ by the subspace spanned by height $k+1$ representations.

Let $\Inv_n(V)_{2m}^{(j+1)}$ be the part of $\Inv_n(V)_{2m}$ spanned by
irreducible representations
of height $j+1$. It follows that $N_{n,2m} = \bigoplus_{j \leq k}
\Inv_n(V)_{2m}^{(j+1)}$. 

We will need to vary $k$.  So, set ${}_k N_{n,2m} := N_{n,2m}$
as above, and more generally let ${}_jN_{n,2m}$ denote
the order $2m$ part of the image of $\Inv_j(V)^v$ in $\Inv_n(V)$. Then,
\begin{equation}
{}_k N_{n,2m}^{(j+1)} = \Ind_{S_{k-j} \times S_{j+1}}^{S_{k+1}} (\CC \boxtimes {}_j N_{n,2m}^{(j+1)}).
\end{equation}

To conclude, we use a well known
combinatorial identity, which is an $S_n$ refinement of the identity
$\sum_{i=0}^n (-1)^i {n \choose i} = 0$. Let $S_0$ denote the trivial
group, viewed as permutations of the empty set.
\begin{lemma}For $n \geq 1$, as virtual $S_n$ representations,
\begin{equation}
\bigoplus_{i+j = n} (-1)^{j} \Ind_{S_i \times S_j}  (\CC \boxtimes \sgn) = 0.
\end{equation}
\end{lemma}
See Remark \ref{symmfunrem} for an interpretation and proof using symmetric
functions.

We deduce that, as virtual representations of $S_{k+1}$,
\begin{equation}
  (\Inv_n(V)_{2m}^{(k+1)})_{S_{n-k}} \cong \bigoplus_{j \leq k} (-1)^{k-j} \Ind_{S_{k-j} \times S_{j+1}}^{S_{k+1}} (\sgn \boxtimes ({}_j N_{n,2m})_{S_{n-j}}).
\end{equation}
As before, each of the $({}_j N_{n,2m})_{S_{n-j}}$ are
finitely-generated graded $S_{j+1}$-equivariant modules over the
polynomial algebra $D_{v,j}$. 
Hence,
\begin{multline}\label{finalhlfla}
  \sum_{m,n} \dim \Hom_{S_{n+1}}(\rho_\lambda[n+1], \Inv_n(V)_{2m}) s^m t^n \\ \equiv_\lambda
  \sum_{-1 \leq j \leq k} \sum_{\sigma \in S_{j+1}} \sum_{\tau \in S_{k-j}} \frac{(-1)^\tau \tr(\sigma \times \tau)|_{\rho_\lambda} Q_{\sigma}(s,t)}{(j+1)!(k-j)!(1-t)\det(1 - st \sigma)|_{\CC^j}},
\end{multline}
where we formally set $\det(1-st \sigma)|_{\CC^{-1}} = 1$.
Using Proposition \ref{trvanprop} again, we deduce the main statement
of the theorem.

(ii) We first prove that, for $\dim V \geq k$, we
can take the numerator to have degree in each of $s$ and $t$ less than
the respective degrees of the denominator, up to replacing the
equality $=$ by $\equiv_\lambda$.  This is because, when $\dim V \geq
k$, by Theorem \ref{polythm}.(ii), the module $\Inv_k(V)^v$ is free
over $D_{v,k}$, and it is generated by elements of order $\leq
2(k-1)$. Hence, the polynomials $Q_\sigma(s,t)$ will all have degree
less than $k-1$ in $s$, for $\sigma \in S_{k+1}$. Moreover, all the
monomials in $Q_\sigma(s,t)$ will have degree exactly $k$ in $t$,
since $\Inv_k(V)^v$ is generated by elements of weight zero, i.e., the
generators all lie in $M_{2i,2m}$ for $i=m$ (in \eqref{preqeqn}
above).

The claim that $J_\lambda(1,1) = |\lambda|!$ is equivalent to
the statement that the multiplicity of $\rho_\lambda[n]$ in $P_n \cong
\CC[S_n]$ is a polynomial of degree $|\lambda|$ in $n$, with leading
coefficient $\frac{\dim \rho_\lambda}{|\lambda|!}$. This is immediate
from the fact that $\rho_\lambda[n]$ occurs with multiplicity $\dim
\rho_\lambda[n]$ in $\CC[S_n]$.

Next, the claim that $J_{\lambda}^+(1,1)=(|\lambda|-1)!$ is equivalent
to the statement that the multiplicity of $\rho_\lambda[n+1]$ in $P_n$
is a polynomial in $n$ of degree $|\lambda|-1$ with leading
coefficient $\frac{\dim \rho_\lambda}{|\lambda|!}$.  In terms of the
above formulas, we have to show that $Q_{\Id}(1,1) = |\lambda|!$,
where $\Id \in S_{|\lambda|}$ is the identity permutation and
$|\lambda| \geq 2$.  This is equivalent to the statement that $\dim
P_{|\lambda|} = |\lambda|!$.

\begin{remark}\label{symmfunrem}
  The above lemma can be interpreted (and proved) using characters, in
  a standard way: send a representation of $S_n$ to the symmetric
  function of variables $p_1, p_2, \ldots$ whose coefficient of
  $\frac{p_1^{r_1} \cdots p_k^{r_k}}{1^{r_1} r_1! 2^{r_2} r_2! \cdots
    k^{r_k} r_k!}$ is the trace of the element of $S_n$, with $r_1 + 2
  r_2 + \cdots + k r_k = n$, whose cycle decomposition contains $r_i$
  cycles of length $i$. Then, symmetric functions of infinitely many
  variables correspond to collections $(V_n)_{n \geq 0}$ where $V_n$
  is a virtual representation of $S_n$.  Multiplying the symmetric
  functions for collections $(U_n)$ and $(V_n)$ yields the collection
  $(W_n)$ with $W_n = \bigoplus_{i+j=n} \Ind_{S_i \times S_j}^{S_n}
  U_i \boxtimes V_j$. In these terms, the above identity becomes
  $e^{\sum_{i \geq 1} \frac{p_i}{i}} e^{\sum_{i \geq 1}
    \frac{-p_i}{i}} = 1$.
\end{remark}

\begin{proof}[Proof of Corollary \ref{liecor}]
  This follows immediately from Theorem \ref{hlthm} and the
  observations before the corollary, except for the statement that the
  evaluation of the numerator at $s=t=1$ is $(|\lambda|-1)!$ and
  $(|\lambda|-2)!$, respectively.  This follows from the proof of (ii)
  above and the fact that $\dim \Lie_{|\lambda|} = (|\lambda|-1)!$.
\end{proof}

\subsubsection{Proof of Corollary \ref{hot3cor}}
In Appendix \ref{s:hot3cor}, we give a direct proof of this result
that does not use Theorem \ref{hlthm}.  Here, we give a proof using
the theorem, partly for the purpose of giving an example which
clarifies the computation of the proof.  It suffices to verify
formulas \eqref{i3pfla}--\eqref{i21pfla}.

We apply \eqref{finalhlfla}.  The main part of the formula concerns
the polynomials $Q_\sigma$, defined for $\sigma \in S_{k+1}$ by
\begin{equation}
  \frac{Q_{\sigma}(s,t)}{\det(1 - st \cdot \sigma)} =\sum_{i=0}^{{k \choose 2}} \sum_{m \geq i} t^{k-i}(st)^m \tr(\sigma)|_{\Inv_{k,m-i}(V)^v_{m+i}}.
\end{equation}
Let us assume that $\dim V \geq k$ (we will only be concerned with the
case $k \leq 2$, so this will always hold). Then, by Theorem
\ref{polythm}.(ii), $\Inv_k(V)^v$ is a free module over the
polynomial algebra $D_{v,k}$ generated by $P_k$ (the vector space of
abstract Poisson polynomials of degree $k$). As a result, we deduce that
\begin{equation}
  Q_\sigma(s,t) = t^k \sum_i \tr(\sigma)|_{(P_k)_{2i}} s^i.
\end{equation}
For $k \leq 2$, the spaces $P_k$ are given by $P_{-1} \cong \CC \cong P_0 \cong P_1$, and $P_2 \cong \CC \oplus \sgn(2)$, where $\sgn(2)$ is the sign representation placed in order two. Hence, the polynomials $Q_\sigma$ are 
\begin{gather}
  Q_{()} = t^{-1}, \quad Q_{(1)} = 1, \quad Q_{(1)(2)} = t = Q_{(12)}, \\
Q_{(1)(2)(3)} = t^2 (1+s), \quad Q_{(12)(3)} = t^2(1-s), \quad Q_{(123)} = t^2(1+s).
\end{gather}
For each partition $\lambda$, denote the RHS of \eqref{finalhlfla} as $RHS_{\lambda}$. We compute it for partitions with $|\lambda| \leq 3$. Since the LHS of \eqref{finalhlfla} is obtained from the RHS by discarding terms whose degree
in $t$ is less than $|\lambda|+\lambda_1-1$, this proves the corollary (and
also verifies \eqref{ipfla}--\eqref{i2pfla}). We use here $u=st$.
\begin{gather}
RHS_{()} = \frac{t^{-1}}{1-t} = \frac{1}{1-t} + t^{-1}, \quad RHS_{(1)} = \frac{1-t^{-1}}{1-t} = -t^{-1}, \\
RHS_{(2)} = \frac{1}{1-t} \bigl(-1 + \frac{t}{2(1-u)} + \frac{t}{2(1+u)} \bigr) = \frac{u^2t}{(1-t)(1-u^2)} - 1, \\
RHS_{(1,1)} = \frac{1}{1-t} \bigl(t^{-1} - 1 + \frac{t}{2(1-u)} - \frac{t}{2(1+u)} \bigr) = \frac{ut}{(1-t)(1-u^2)} + t^{-1},
\end{gather}
\begin{multline}
RHS_{(3)} = \frac{1}{1-t} \bigl(  - \frac{t}{2(1-u)} - \frac{t}{2(1+u)}
 + \frac{t(t+u)}{6(1-u)^2} + \frac{3t(t-u)}{6(1-u)(1+u)} + \frac{2t(t+u)}{6(1+u+u^2)}\bigr) \\ = \frac{u^3t^2 + u^4 t^2 - u^7t - u^7t^2}{(1-t)(1-u^2)(1-u^3)} + u^3t+t^2-t.
\end{multline}
\begin{multline}
RHS_{(1,1,1)} = \frac{1}{1-t} \bigl( -t^{-1}+1  -\frac{t}{2(1-u)} + \frac{t}{2(1+u)}
 + \frac{t(t+u)}{6(1-u)^2} - \frac{3t(t-u)}{6(1-u)(1+u)} + \frac{2t(t+u)}{6(1+u+u^2)}\bigr) \\ = \frac{u^3t^2+u^4t}{(1-t)(1-u^2)(1-u^3)} -t^{-1}.
\end{multline}
\begin{multline}
RHS_{(2,1)} = \frac{1}{1-t} \bigl( 1 - \frac{t}{1-u} + \frac{2t(t+u)}{6(1-u)^2}  - \frac{2t(t+u)}{6(1+u+u^2)}\bigr) \\ = \frac{u^2t^2+u^4t^2+u^5t-u^5t^2}{(1-t)(1-u)(1-u^3)} + 1 - ut.
\end{multline}

\subsection{Proof of Theorem \ref{ht4thm}}
First, we identify the subspace of $\Inv_3(\CC^2)^v$ corresponding to
\eqref{ht4genfla}. By Theorem \ref{isotpartthm}, $\Inv_3(\CC^2) \cong \CC(0,0) \oplus
\hh(2,0) \oplus \rho_{(2,2)}(4,0)$, (or
by Theorem \ref{scinvisotthm}, or by \S \ref{lowdimsec}).  This space
consists of the Poisson polynomials of degree three.  Next, the sign
representation occurs in $\Quant_4(\CC^2)_8 \subseteq \Inv_4(\CC^2)
\subset \Inv_3(\CC^2)^v$, and occurs in order $7$ and weight $1$ (this
can be seen explicitly; it is also forced since the weight cannot be
zero as the element does not come from $\Inv_3(\CC^2)$, but cannot
exceed one since that is the length of the top row of the Young
diagram).  Thus, we consider this to be the summand $\sgn(7,1)$.

Next, we prove that the given subspace generates $\Inv_3(\CC^2)^v$ as a module over $D_{v,k}$.  Recall that $\Inv_3(\CC^2)^v$ identifies with the space of generic harmonic polynomials in $A := \CC[a_1,a_2,a_3,x_1,x_2,x_3]$, and that for generic $a_i$, the Hilbert series of this space is
\begin{equation}
(1+t)(1+t+t^2) = 1+2t+2t^2+t^3.
\end{equation}
The grading here, in $x_i$, coincides with half of order minus weight. The grading by the $a_i$ coincides with half of order plus weight.

If we identify the subspace \eqref{ht4genfla} with generic harmonic
polynomials as above, we obtain the same Hilbert series (in half of
order minus weight).  Hence, it follows that the given subspace
generically (in the $a_i$) generates $\Inv_3(\CC^2)^v$.  Call the
submodule generated by \eqref{ht4genfla} $M$. We have proved that
$\Inv_3(\CC^2)^v/M$ is torsion. If this is nonzero, then $A / M$ has nonzero torsion.  We will prove this is impossible.

First, note that none of the summands
in \eqref{ht4genfla} correspond to generic harmonic polynomials which
themselves vanish on a hyperplane: otherwise, the generic harmonic
polynomial would be a multiple of a linear polynomial in the $a_i$ and
another generic harmonic polynomial. This latter polynomial would have
order and weight one less, which is only possible in the case the
original element spanned $\sgn(7,1)$; one would then obtain a generic
harmonic polynomial of order three and weight zero, i.e., an element
of $\Inv_3(\CC^2)_6$, which as we know is zero. So this is impossible.

Since $M$ is graded, by degree in the $x_i$, in degrees $\leq 3$, and
the part of degrees $0$ and $3$ are generated by single polynomials,
this implies that in degrees $0$ and $3$, $A/M$ is torsion-free.
Hence, torsion can only occur in degrees one or two in the $x_i$.

In degree one in the $x_i$, $M$ is generated by the elements $a_i x_j
- a_j x_i$.  Call this module $M_1$, and call $A_1$ the submodule of
$A$ of all polynomials of degree one in the $a_i$.  Then, $M_1$ is the
kernel of the map $A_1 \to A_0$, $x_i \mapsto a_i$.  Hence, $A_1/M_1
\into A_0$ is torsion-free, and therefore so is $A/M_1$, as desired.

Call $M_2 \subset M$ the subspace in degree two. This is generated by
the cyclic permutations of
\[
(a_1(x_2+x_3) - x_1(a_2+a_3))(a_2 x_3 - a_3 x_2),
\]
which corresponds to the Poisson polynomial $\{f_1,\{f_2,f_3\}\}$. Let
$\xi_i := a_j x_k - a_k x_j$ for $(i,j,k)$ a cyclic permutation of
$(1,2,3)$.  Then $M_2 \subseteq \Sym^2 M_1 = \langle \xi_i \xi_j
\rangle_{i,j \in \{1,2,3\}}$ is the submodule generated by
$(\xi_i-\xi_j)\xi_k$ for $\{i,j,k\}=\{1,2,3\}$.
Now, $\Sym^2 M_1$ is 
the kernel of $A_2 \to A_1, f \mapsto a_1 \frac{\partial f}{\partial
  x_1} + a_2 \frac{\partial f}{\partial x_2} + a_3 \frac{\partial
  f}{\partial x_3}$.  So it suffices to show that $M_2$ is the kernel
of a map $\Sym^2 M_1 \to A_0$.  
Note that $M_1 \cong A_0^3 / A_0$,
presented by the single relation $a_1 \xi_1 + a_2 \xi_2 + a_3 \xi_3 =
0$.  We can thus consider the map $\varphi: \Sym^2 M_1 \to A_0,
\varphi(\xi_i \xi_j)=a_1a_2a_3$ for all $i \neq j$, and $\varphi(\xi_i^2) = 
-a_ja_k(a_j+a_k)$ for $\{i,j,k\}=\{1,2,3\}$.  This is well-defined and
its kernel is evidently $M_2$.  Hence $\Sym^2 M_1/M_2$ is torsion-free,
and hence also $A_2/M_2$ and therefore $A/M_2$, as desired.

We have thus proved that $\Inv_3(\CC^2)^v$ is generated by the given
subspace.  It remains to verify the claimed relation, that the image
of $\sgn(3,1)$ under $D_{v,k} \otimes \hh(2) \onto D_{v,k} \hh(2)$ is
zero, and to prove that there are no other relations.  The relation
says that $\wedge^2 \pi = 0$, i.e., that $\wedge^4 \CC^2 = 0$, so it
holds.  To see that there are no other relations, we will use the fact
that, as an $S_{n+1}$-representation, $\Quant_n(V)$ is independent of
$V$ (only the grading depends on $V$).  By Proposition \ref{fisxprop},
the height $\leq k+1$ part of $\Quant_n(V)$ is uniquely determined
combinatorially from the weight-graded $S_{k+1}$-subrepresentation of
$\Inv_3(V)^v$ which maps to $\Quant_n(V)$, but by the above, this is
all of $\Inv_3(V)^v$.
Thus $\Inv_3(V)^v$
 cannot depend on $V$ as
a weight-graded $S_{4}$-representation.

However, the weight-graded $S_4$-structure of the free module
generated by $\CC \oplus \hh \oplus \rho_{(2,2)}$ coincides with that
of the module generated by \eqref{ht4genfla} modulo $\sgn(3,1)$, since
we have merely added in an extra generator of $\sgn(7,1)$ and modded
by $\sgn(3,1)$.  Hence, there can be no other relations, and the proof
is complete.

\subsection{Proof of Theorems \ref{scinvordthm} and
  \ref{scinvordconvthm}} \label{scinvordthmpfsec} We first prove
Theorem \ref{scinvordthm}.  We will use the description of $\Inv_n(V)$
from \S \ref{scinvisotpfsec}, as the span of compositions of
$\pi^{i,j}$ which are invariant under a single Hamiltonian vector
field, associated to a nonzero cubic Hamiltonian in $\cO_V$. We fix an
order $2m$, so we are interested in the span of the composition of
exactly $m$ such $\pi^{i,j}$.  Viewing $n$ as a parameter, the part of
$\Inv_n(V)$ that concerns us is the span of compositions in which at
least $2m-2$ distinct indices occur.

  It is obvious that a composition of $m$ copies of $\pi$ applied to
  $2m$ distinct indices is the same as a Poisson polynomial, which is
  a permutation of $\{f_1, f_2\} \{f_3, f_4\} \cdots \{f_{2m-1},
  f_{2m}\}$. Hence we are really only interested in the span of
  compositions in which $2m-1$ or $2m-2$ distinct indices occur.

  If $2m-1$ distinct indices occur, we are looking at linear
  combinations of permutations of the element
\begin{equation}\label{2mm1comps}
\pi^{1,2} \pi^{1,3} \pi^{4,5} \pi^{6,7} \cdots \pi^{2m-2, 2m-1}.
\end{equation}
The $S_3$-representation spanned by $\pi^{1,2} \pi^{1,3}$ is
three-dimensional, since the element is symmetric in two of its indices,
and is neither completely symmetric, nor is its complete
symmetrization zero.  On the other hand, the $S_3$-representation
spanned by $\{f_1, \{f_2, f_3\}\}$ is two-dimensional, so its
complement in $\CC[S_3] \cdot \pi^{1,2} \pi^{1,3}$ is spanned by the
complete $S_3$-symmetrization of $\pi^{1,2} \pi^{1,3}$.  Therefore, it
suffices to consider linear combinations of permutations of the
element
\begin{equation}\label{2mm1comps2}
  (\pi^{1,2} \pi^{1,3} + \pi^{2,1} \pi^{2,3} + \pi^{3,1} \pi^{3,2})  \pi^{4,5} \pi^{6,7} \cdots \pi^{2m-2, 2m-1}.
\end{equation}
Let us apply the Hamiltonian vector field $3x_1^2 \pd{}{y_1}$ for some
$x_1 \in V^*$, discussed in \S \ref{scinvisotpfsec}. It is easy to
compute that one obtains
\begin{equation}\label{2mm1comps3}
18(\pd{}{y_1})^{\otimes 3}  \pi^{4,5} \pi^{6,7} \cdots \pi^{2m-2, 2m-1}.
\end{equation}
We need to find the $S_n$-subrepresentations of the span of
permutations of \eqref{2mm1comps2} which vanish when we apply this
Hamiltonian vector field.  The irreducible subrepresentations spanned
by permutations of \eqref{2mm1comps2} are those labeled by Young
diagrams obtained by combining the horizontal line of three cells
($\lambda = (3)$) with $m-2$ copies of the vertical line of two cells
($\lambda = (1,1)$, which we will call the ``vertical domino''), where
combining two diagrams means adding the cells of the second diagram to
the first in such a way that two cells are added to the same row only
if they originated in the same row, and to the same column only if
they originated in the same column.  The only such representation in
which \eqref{2mm1comps3} can vanish is one in which there is a column
with more than $\dim V$ cells, i.e., when the element $\pd{}{y_1}
\wedge (\pi)^{\wedge \frac{1}{2} \dim V} = 0$ appears.  In the case
that $2m=\dim V + 4$, this can happen exactly in the case where our
Young diagram, with $\dim V + 3$ cells, is $L$-shaped, of the form
$(3,1,1,1,\ldots,1)$, coming from the span of the skew-symmetrization
of \eqref{2mm1comps2} in components $1,4,5,6,\ldots,2m-1$ (leaving
components $2$ and $3$ fixed).  This element, however, coincides with
the skew-symmetrization in components $1,4,5,\ldots,2m-1$ of the
Poisson polynomial
\begin{equation}
\bigl(\{f_1, \{f_2, f_4\}\} \{f_3,f_5\} + \{f_1, \{f_3, f_4\}\} \{f_2,f_5\}\bigr) \{f_6,f_7\} \{f_8, f_9\} \cdots \{f_{2m-2}, f_{2m-1}\},
\end{equation}
and hence it is a Poisson polynomial.  To see this, note that the
Young diagram of this latter element is obtained from the $3$-celled
$L$-shaped diagram $(2,1)$ corresponding to $\{f_1, \{f_2, f_3\}\}$,
by adding one vertical domino $(2)$ to obtain the diagram $(3,1,1)$,
and then adding the remaining vertical dominoes to the first column,
obtaining $(3,1,1,\ldots,1)$ as before. The resulting elements must be
scalar multiples of each other, since the long column shows that these
elements are in the $\CC[S_n]$-span of multiples of the sign element
$\Alt(\pi^{4,5} \pi^{6,7} \cdots \pi^{2m-2, 2m-1})$, and by order and
degree considerations, such $\sp(V)$-invariant elements are the same
as the ones we considered above (the skew-symmetrization of
\eqref{2mm1comps2} in components $1,4,5,6,\ldots,2m-1$).

For $2m < \dim V + 4$, there is no subrepresentation of the span of
\eqref{2mm1comps2} which is killed by applying this Hamiltonian vector
field, and for $2m > \dim V + 4$, all such subrepresentations are
spanned by multiplying this one by pairwise Poisson bracketing in the
remaining $\dim V + 4 - 2m$-components.  We conclude that all linear
combinations of permutations of \eqref{2mm1comps} which are invariant
under symplectic automorphisms are Poisson brackets, as desired.

It remains to consider linear combinations of compositions of $m$
copies of $\pi$ applied to precisely $2m-2$ distinct indices. These must be
represented as a sum of permutations of the elements
\begin{gather}
  (\pi^{1,2} \pi^{1,3}) (\pi^{4,5} \pi^{4,6}) (\pi^{7,8} \pi^{9,10} \cdots \pi^{2m-3,2m-2}), \label{twodoublepb}\\
  (\pi^{1,2} \pi^{1,3} \pi^{1,4}) (\pi^{5,6} \pi^{7,8} \cdots \pi^{2m-3, 2m-2}), \label{threetreepb} 
\\
  (\pi^{1,2} \pi^{2,3} \pi^{3,4}) (\pi^{5,6} \pi^{7,8} \cdots
  \pi^{2m-3, 2m-2}), \label{threelinepb} \\
(\pi^{1,2} \pi^{1,2}) \pi^{3,4} \pi^{5,6} \cdots \pi^{2m-3,2m-2}. \label{doublebondpb}
\end{gather}
We will see later that all invariant elements decompose as an
invariant element spanned by permutations of \eqref{twodoublepb} only
(corresponding to Poisson polynomials involving a product of two
iterated brackets $\{-,\{-,-\}\}$), an invariant element spanned
by permutations of \eqref{threetreepb} and \eqref{threelinepb} only
(corresponding to Poisson polynomials involving an iterated bracket
$\{-,\{-,\{-,-\}\}\}$), and an invariant element spanned by \eqref{doublebondpb} only (which must be zero).

First, let us compute the invariant elements spanned by permutations
of \eqref{twodoublepb} only.  We can decompose components $1$--$3$ and
$4$--$6$ each along their irreducible representations of $S_3$.  The
case where both are the two-dimensional irreducible representation, as
we have seen, is already a Poisson polynomial.  So we can assume that
at least one of these components is a completely symmetrized
term. Thus, we can restrict our attention to sums of permutations of
the symmetrization in components $1$--$3$ of the element
\eqref{twodoublepb}.  As before, we will only obtain an element
invariant under the Hamiltonian vector field $3x_1^2
\frac{\partial}{\partial y_1}$ if one of the components $1, 2,$ or
$3$, say component $1$, is skew-symmetrized with a copy of
$\pi^{\wedge \frac{1}{2} \dim V}$, i.e., a nonzero volume element of
$\wedge^{\dim V}\Span(\pd{}{x_1}, \ldots, \pd{}{x_d}, \pd{}{y_1},
\ldots, \pd{}{y_d})$, writing $\dim V = 2d$ with (standard)
coordinates $x_1, \ldots, x_d, y_1, \ldots, y_d$.  Such an element,
though, is already a Poisson polynomial, of degree $\dim V + 3$, as we
saw. If we were dealing with components $4$--$6$ coming from the
two-dimensional irreducible representation of $S_3$, i.e., a Poisson
polynomial, then we can assume the whole element is a Poisson
polynomial.

So, we can restrict our attention to the span of permutations of
\eqref{twodoublepb} in which components $1$--$3$ and also $4$--$6$ are
symmetrized.  Then, the resulting element is also symmetric under the
permutation $(14)(25)(36)$, i.e., it has $(S_3 \times S_3) \rtimes
S_2$-symmetry, and so the Young diagrams for the irreducible representations
occurring in the $S_6$ span, with $S_6$ acting on the first six components,
must be either  the horizontal line with six
cells, $(6)$, or the diagram $(4,2)$ with two rows. In the former case,
the same argument as above shows that we must add two columns of
length $\dim V$ below this horizontal line to get an invariant
operator, and this will again be a Poisson polynomial (in the span of
the square of the aforementioned polynomial of degree $\dim V + 3$).

So, we are reduced to the case of the span of permutations of the
element obtained from \eqref{twodoublepb} by symmetrizing indices
$1$--$3$ and $4$--$6$ and furthermore restricting to the representation
$(4,2)$ in indices $1$--$6$.  By the same argument as before, we need to
add $\dim V$ cells to the first column.  The result is an invariant
operator, whose $S_n$-span is an irreducible representation
whose Young diagram has $\dim V + 4$ cells, of the form
$(4,2,1,1,\ldots,1)$.  In the case that $\dim V \geq 4$, we can also
obtain this representation as a Poisson polynomial in the span of
\eqref{twodoublepb}, with components $1$--$3$ and $4$--$6$ both in the
two-dimensional representation: we can combine the two resulting
$L$-shapes $(2,1)$ and $(2,1)$ into the Young diagram $(2,2,1,1)$, then
add two vertical dominoes to obtain $(4,2,1,1,1,1)$, and then add the remaining $\frac{1}{2} (\dim V - 4)$ vertical dominoes needed to the first column. In the case that $\dim V = 2$, this isn't possible, but let us consider more carefully the resulting invariant operator.  Labeling the Young diagram $(4,2,1,1)$ in the standard away (left-to-right, top-to-bottom, i.e., components $1,2,3,4$ in the first row, followed by $5,6$ in the second row, and then $7$ and $8$ in the last two rows),
we obtain the resulting Young symmetrization of the single composition
\begin{equation}
\pi^{1,2} \pi^{1,3} \pi^{5,4} \pi^{5,6} \pi^{7,8}.
\end{equation}
As we see, the second column, in components $2$ and $6$, has an
element of $\wedge^2 \Span(\pd{}{x_1}, \pd{}{y_1})$, which means we
actually have a multiple of $\pi^{2,6}$. We can conclude that the
above element is the same, up to scaling, as the Young symmetrization
of the element
\begin{equation}
(\pi^{3,1} \pi^{1,5} \pi^{5,4}) \pi^{2,6} \pi^{7,8}.
\end{equation}
This element is of type \eqref{threelinepb}, so we can reduce the
problem to the remaining cases.

So far, we have showed that any invariant element in the span of
permutations of \eqref{twodoublepb} is spanned by Poisson polynomials
and permutations of \eqref{threelinepb}.  It suffices to show
that all invariant elements in the span of all four types,
\eqref{twodoublepb}--\eqref{threelinepb}, have the property that the
terms of the form \eqref{threetreepb} and \eqref{threelinepb} sum to
an invariant element, that the terms of the form \eqref{doublebondpb} sum
to an invariant element, and finally that these invariant elements must be
Poisson polynomials.

To do this, we will consider carefully the action of the Hamiltonian
vector field $\xi := 3x_1^2 \pd{}{y_1}$ on the terms
\eqref{threetreepb}, \eqref{threelinepb}, and \eqref{doublebondpb}.

Let us first apply $\xi$ to \eqref{threetreepb}.  This amounts
to applying it to 
\begin{equation}\label{smallthreetreepb}
\pi^{1,2} \pi^{1,3} \pi^{1,4},
\end{equation}
 and we get, up to a nonzero
constant factor, the symmetrization in components $2,3,4$ of the element
\begin{equation} \label{threetreepbham}
\pi^{1,2} \bigl(\pd{}{y_1}\bigr)^{(1)} \bigl(\pd{}{y_1}\bigr)^{(3)} \bigl(\pd{}{y_1}\bigr)^{(4)}),
\end{equation}
where here the superscript of $(i)$ means applying the element in the $i$-th component.

Next, we apply $\xi$ to \eqref{threelinepb}. This amounts to applying it
to 
\begin{equation}\label{smallthreelinepb}
\pi^{1,2} \pi^{2,3} \pi^{3,4},
\end{equation}
 and we get, again up to a nonzero
constant factor,
\begin{equation}\label{threelinepbham}
\pi^{1,2} \bigl(\pd{}{y_1}\bigr)^{(2)} \bigl(\pd{}{y_1}\bigr)^{(3)} \bigl(\pd{}{y_1}\bigr)^{(4)} + \pi^{3,4}\bigl(\pd{}{y_1}\bigr)^{(1)} \bigl(\pd{}{y_1}\bigr)^{(2)} \bigl(\pd{}{y_1}\bigr)^{(3)}.
\end{equation}

Finally, we apply $\xi$ to \eqref{doublebondpb}. This amounts to applying
it to 
\begin{equation}\label{smalldoublebondpb}
\pi^{1,2} \pi^{1,2}
\end{equation}
and we get, up to a nonzero constant factor,
\begin{equation}\label{doublebondpbham}
(\frac{\partial^2}{\partial y_1^2})^{(1)} (\pd{}{y_1})^{(2)} + (\frac{\partial^2}{\partial y_1^2})^{(2)} (\pd{}{y_1})^{(1)}.
\end{equation}

Note that \eqref{doublebondpbham}, unlike everything else, has terms
which act in one component as $\frac{\partial^2}{\partial y_1^2}$ and
in another component by $\frac{\partial}{\partial y_1}$. So, in any
invariant element obtained from adding permutations of
\eqref{twodoublepb}, \eqref{threetreepb}, \eqref{threelinepb}, and
\eqref{doublebondpb}, the terms of the form \eqref{doublebondpb} must
themselves sum to an invariant element.  On the other hand, the map
from the $\CC[S_2]$-span of \eqref{smalldoublebondpb} to the
$\CC[S_2]$-span of \eqref{doublebondpbham} is an isomorphism, so terms
of the form \eqref{doublebondpb} cannot span to an invariant element.

Note that \eqref{threetreepbham} and \eqref{threelinepbham}, unlike
the nonzero terms that result from applying $\xi$ to any permutations
of terms of the form \eqref{twodoublepb}, have a single component $i$
such that the element is a multiple of $(\pd{}{y_1})^{(i)} \pi^{i,j}$
for some other index $j$.  In all other components, the polydifferential
operators act with order $\leq 1$.  These terms cannot cancel with any
of the terms which involve a product of an $\sp(V)$-invariant and
$(\pd{}{y_1})^{\otimes 3}$ acting in three separate components.
 We conclude that any
invariant element which is a linear combination of permutations of
\eqref{twodoublepb}, \eqref{threetreepb}, and \eqref{threelinepb}
has the property that the terms of the latter two forms
themselves span an invariant element, as we desired.

It remains to show that any invariant operator in the span of
permutations of \eqref{threetreepb} and \eqref{threelinepb}
multiplied by $\pi^{5,6} \pi^{7,8} \cdots \pi^{2m-3,2m-2}$ is a
Poisson polynomial. First, we consider the issue without multiplying
by $\pi^{5,6} \cdots \pi^{2m-3,2m-2}$.  Note that the $\CC[S_4]$-span
of the element \eqref{threetreepbham} is four-dimensional, consisting
of the trivial representation and the three-dimensional reflection
representation, $\hh$. This is the same as the dimension of the
$\CC[S_4]$-span of the original element \eqref{smallthreetreepb}.  The
$\CC[S_4]$-span of \eqref{smallthreelinepb} is a twelve-dimensional
space (picking which two components have order two, and each six
choices have a two-dimensional span of choices). It consists of two
copies each of $\hh$ and $\wedge^2 \hh$ (since the given element is
skew-symmetric under the permutation $(14)(23)$, and hence the
representation is isomorphic to $\Ind_{\langle (14)(23)\rangle}^{S_4}
\CC$).  On the other hand, the $\CC[S_4]$-span of \eqref{threelinepbham} is only
nine-dimensional: if we skew-symmetrize the element in components
$1,2$ and simultaneously in components $3,4$, the result vanishes, and
this is the entire kernel of the previous twelve-dimensional
space. This kernel is a copy of $\wedge^2 \hh$.  Moreover, the intersection
of the $\CC[S_4]$-spans of \eqref{threetreepbham} and \eqref{threelinepbham} is a copy of $\hh$.

Put together, the kernel of applying $\xi$ to the span of permutations
of \eqref{threetreepb} and \eqref{threelinepb} can be at most
six-dimensional, consisting of one copy each of $\hh$ and $\wedge^2
\hh$. The kernel must be exactly this and equals the $\CC[S_4]$-span
of the Poisson monomial $\{f_1, \{f_2, \{f_3, f_4\} \} \}$, since the
latter is isomorphic also to $\hh \oplus \wedge^2 \hh$ and
permutations of \eqref{smallthreetreepb} and \eqref{smallthreelinepb} are
all of the $\sp(V)$-invariants of degree four and order three.

Finally, if we multiply by $\{f_5, f_6\} \cdots
\{f_{2m-3},f_{2m-2}\}$, to consider applying $\xi$ to permutations of
\eqref{threetreepb} and \eqref{threelinepb}, the only new invariant
elements will come by skew-symmetrizing a component of order one in
\eqref{threetreepbham} or \eqref{threelinepbham} with the volume
element $\wedge^{\frac{1}{2}\dim V} \pi$.  In the case of
\eqref{threetreepbham}, to obtain this we already had to kill
\eqref{smallthreetreepb} itself.  In the case of
\eqref{threelinepbham}, the same is true except in the case of the
copy of $\hh$ in which components $2$ and $3$ are skew-symmetric,
which we can further skew-symmetrize with $\wedge^{\frac{1}{2} \dim V}
\pi$.  This final element, though, must be identified with the
complete skew-symmetrization of $\{f_1, \{f_2, \{f_3, f_4\}\}\} \{f_5,
f_6\} \cdots \{f_{2m-3}, f_{2m-4}\}$, and is hence a Poisson
polynomial.  This completes the proof.
\subsubsection{Proof of Theorem \ref{scinvordconvthm}}
We begin with part (i). Note first that the representations
that occur in $\Inv_n(V)_{2m}$ all lie in $\hh^{\otimes \leq 2m}$, since
they consist of elements that involve applying operators to at most $2m$ components.  Then, if $2m \leq \dim V$, by Theorem \ref{scinvisotthm}, we deduce that $\Inv_n(V)_{2m} \cong \SC_n(V)_{2m} \cong (P_n)_{2m}$.

(ii) When $2m \leq \dim V + 4$, then $\Inv_n(V)_{2m} /
\SC_n(V)_{2m}$ must consist only of height $2m-2, 2m-1$, or $2m$
irreducible representations.  If this is nonzero for some, or
equivalently sufficiently large $n$ (these are equivalent by composing
an operator $\phi$ with multiplication on arbitrarily many functions,
i.e., replacing one input with the product of several new inputs),
then $\dim \Inv_n(V)_{2m}/\SC_n(V)_{2m}$ would be polynomial of degree
$\geq 2m-2$, contradicting Theorem \ref{scinvordthm}.  Hence, this
quotient is zero.  Next, if $2m > \dim V$, then the map $(P_n)_{2m}
\onto \SC_n(V)_{2m}$ is not an isomorphism, since it kills for example
the skew-symmetrization of $\{f_1,f_2\} \cdots \{f_{2m-1},
f_{2m}\}$. Put together, for $2m \in \{\dim V+2, \dim V + 4\}$, one
obtains a canonical noninjective surjection $(P_n)_{2m} \onto
\Inv_n(V)_{2m}$.

(iii) Example \ref{modmatexam} supplies a copy of the sign
representation in 
\[
\Inv_{\dim V + 3}(V)_{\dim V + 6} / \Quant_{\dim V
  + 3}(V)_{\dim V + 6}.
\] 
(Alternatively, one can use the sign
representation in $\Quant_{\dim V + 2}(V)_{\dim V + 6} / \SC_{\dim V +
  2}(V)_{\dim V + 6}$.)  If we take an operator $\phi$ generating this
representation and consider the operator
\begin{equation}
\phi(f_1 \otimes \cdots \otimes f_{\dim V + 3}) \{f_{\dim V + 4}, f_{\dim V + 5}\} \cdots \{f_{2m-4}, f_{2m-3}\} f_{2m-2} \cdots f_n \in \Inv_n(V)_{2m},
\end{equation}
we get an element of order $2m$ which generates a subrepresentation of
dimension a polynomial of degree $2m-3$ in $n$, not coming from
$\SC(V)$ (in fact, not coming from $\Quant$; or in the alternative
example, we can get an element of $\Quant$ not coming from $\SC$).  We
conclude that $\dim \Inv_n(V)_{2m}/\SC_n(V)_{2m}$ is a polynomial of
degree precisely $2m-3$ in $n$ (moreover, the same is true for $\dim
\Inv_n(V)_{2m}/\Quant_n(V)_{2m}$ and $\dim
\Quant_n(V)_{2m}/\SC_n(V)_{2m}$).

\section{Proof of Theorems \ref{invinterpthm} and \ref{quantinterpthm}}\label{s:scquantsecpf}
\begin{proof}[Proof of Theorem \ref{invinterpthm}]
   Let $\phi: V^{n+1} \onto V^n$ be the quotient by the diagonal.  We
  are concerned with, on the one hand, polydifferential operators on $V^n$
  with constant coefficients, viewed as distributions on $V^{n+1}$,
  invariant under symplectic automorphisms of $V$, and on the other
  hand, with functionals on $\cO_{V^n}$ invariant under bracketing
  with $S_{n+1}$-invariant Hamiltonians.

First, note that we have a map from constant-coefficient polydifferential
operators to linear functionals, obtained by composing with the
evaluation at the origin.  This is evidently an injection, identifying
such operators as elements of $\cO_{V^n}^*$.  Next, consider the
pullback $\phi^*: \cO_{V^n} \into \cO_{V^{n+1}}$.  Viewed as
distributions on $V^{n+1}$, by the integration by parts rule,
constant-coefficient polydifferential operators on $\cO_{V^n}$ become the
quotient of constant-coefficient polydifferential operators on
$\cO_{V^{n+1}}$ by the relation, for all $v \in V$, that
$\sum_{i=1}^{n+1} \frac{\partial}{\partial v^{(i)}} = 0$, where
$v^{(i)} \in V^{n+1}$ denotes $v$ acting in the $i$-th
component. Identifying polydifferential operators with functionals by
composing with the evaluation at zero, this relation is merely the
kernel of $(\phi^*)^*: \cO_{V^{n+1}}^* \onto \cO_{V^n}^*$.  With this
identification, the $S_{n+1}$-action on such polydifferential operators
becomes transparent.

Next, as observed in \cite[Lemma and Definition 1.2]{Matso}, the
condition of being invariant under symplectic automorphisms of $V$ is
equivalent to being invariant under Hamiltonian vector fields on $V$.
Invariance in $\cO_{V^{n}}^*$, viewed as a quotient of $\cO_{V^{n+1}}^*$,
 under Hamiltonian vector fields on $V$
is the same as invariance under $S_{n+1}$-invariant Hamiltonian vector
fields on $V^{n}$, since a vector field $\xi \in
\operatorname{Vect}(V)$ acts on $\cO_{V^{n+1}}^*$ by the vector field
$\sum_{i=1}^{n+1} \xi^{(i)} \in
\operatorname{Vect}(V^{n+1})^{S_{n+1}}$.  This, in turn, is equivalent
to annihilating $\{\cO_{V^{n}}^{S_{n+1}}, \cO_{V^{n}}\}$.
This identifies $\Inv_n(V)$ with $\HP_0(\cO_{V^n}^{S_{n+1}}, \cO_{V^n})$,
as desired.
\end{proof}

\begin{proof}[Proof of Theorem \ref{quantinterpthm}]
 Over $\CC((\hbar))$, a similar argument identifies
 \[
\HH_0((\cO_{V^n}^{S_{n+1}}((\hbar)), \star), (\cO_{V^n}((\hbar)), \star))^*
\]
with the space of polydifferential operators invariant
under continuous $\CC((\hbar))$-algebra automorphisms of $(\cO_V((\hbar)), \star)$.  Such operators obviously
include the operations
\begin{equation}\label{spopeqn}
f_1 \otimes \cdots \otimes f_n \mapsto f_{\sigma(1)} \star f_{\sigma(2)} \star \cdots \star f_{\sigma(n)},
\end{equation}
for all $\sigma \in S_n$.  

To conclude, we have to show that these operations span the Hochschild
trace group. Since the star-product operations are all linearly independent,
it suffices to show that the dimension of the Hochschild trace group is $n!$.
This follows from a generalization of the main result of
\cite{AFLS}. In more detail, let $U$ be any complex vector space, $G <
\Sp(U \oplus U^*)$ be finite, and let $\cD_U$ denote polydifferential
operators on $U$. Then, the main result of \cite{AFLS} identifies the
$G$-invariants $\HH_0(\cD_U^{G}) = \HH_0(\cD_U^G, \cD_U)^G$ with the
linear span of conjugacy classes of elements $g \in G$ such that $g -
\Id$ is invertible.  The same proof shows that $\HH_0(\cD_U^G, \cD_U)$
is isomorphic, as a $G$-representation, to the subspace of $\C[G]$
spanned by elements $g \in G$ such that $g - \Id$ is invertible, with
the adjoint action.  Since $\cD_U$ is a filtered quantization of the
symplectic vector space $U \oplus U^*$, one obtains the same result
for any star-product quantization $(\cO_{U \oplus U^*}[[\hbar]],
\star)$. Specializing to $U \oplus U^* = V^n$ and $G = S_{n+1}$, one
obtains that $\HH_0((\cO_{V^n}^{S_{n+1}}((\hbar)),\star),
(\cO_{V^n}((\hbar)), \star))^* \cong \Ind_{\ZZ/(n+1)}^{S_{n+1}}
\CC((\hbar))$; viewed as a representation of $S_n \subset S_{n+1}$,
one recovers the group algebra.

For the final statement, note that there is an obvious map from the
associative operad to the span of star-product operations
\eqref{spopeqn} over all $n$, sending an abstract multiplication
operation to the corresponding star-product operation (in fancier
language, using the tautological morphism of operads $\Ass((\hbar))
\rightarrow \mathcal{E}\text{nd}(\cO_V((\hbar)), \star) = \bigoplus
\Hom_{\C((\hbar))}(\cO_V^{\otimes n}((\hbar)), \cO_V((\hbar)))$). This
is evidently surjective, and is injective since the elements of
\eqref{spopeqn} are linearly independent for distinct permutations
$\sigma$.
\end{proof}

\section{Proof of Corollaries \ref{classisotcor} and \ref{scquantinvordcor}}\label{s:bicondpf}
\begin{proof}[Proof of Corollary \ref{classisotcor}]
First of all, we have the simple
\begin{claim}
  If an irreducible representation of $S_{n+1}$ with Young diagram
  $\lambda = (\lambda_1, \ldots, \lambda_\ell)$ occurs in $\Inv_n(V)$,
  then $(\lambda_1+1, \lambda_2, \ldots, \lambda_\ell)$ occurs in
  $\Inv_{n+1}(V)$, and similarly replacing $\Inv$ with $\Quant$ or
  $\SC$, or with a quotient $\Inv/\Quant, \Inv/\SC$, or $\Quant/\SC$.
\end{claim}
\begin{proof}[Proof of the claim.]  We can restrict the representation
  to the irreducible $S_n$-subrepresentation with Young diagram
  $(\lambda_1-1, \lambda_2, \ldots, \lambda_\ell)$ if $\lambda_1 >
  \lambda_2$, and otherwise the diagram is obtained from this one by
  permuting the rows appropriately. Call the new diagram $\lambda' =
  (\lambda_1', \ldots, \lambda_\ell')$. Then, the resulting
  $S_n$-subrepresentation of $\Inv_n(V)$ is generated by an operator
  $\phi$ that we can view as one that acts on only $\ell$ functions,
  $f_1, \ldots, f_\ell$, but of degrees $\lambda_1', \ldots,
  \lambda_\ell'$ rather than being linear. Now, if we consider the
  operator $\phi \cdot f_1$, of degrees $\lambda_1'+1, \lambda_2'
  \ldots, \lambda_\ell'$, the resulting element of $\Inv_{n+1}(V)$
  generates an irreducible $S_{n+1}$-representation of the form $(\lambda_1'+1,
  \lambda_2', \ldots, \lambda_\ell')$, and if we consider the $S_{n+2}$ action,
  it generates an irreducible
  $S_{n+2}$-representation of the form $(\lambda+1, \lambda_2, \ldots,
  \lambda_\ell)$ (in both cases, $\phi \cdot f_1$ may also
  generate other irreducible
  representations).  The same argument applies to $\Quant$ and
  $\SC$. As for a quotient $\Inv/\Quant$, $\Inv/\SC$, or $\Quant/\SC$,
  we note that if an operator is not in $\SC$ or $\Quant$, then the
  above construction will produce an operator not in $\SC$ or $\Quant$
  in higher degree.
\end{proof}
Now, the equivalence of ``all $n$'' and ``some $n > k+j$'' follows from this, since if $\Inv_n(V)/\SC_n(V)$ had an irreducible representation $\rho[n]$
of the type $\lambda = (\lambda_1, \ldots, \lambda_\ell)$ with $\lambda_2+\ldots+\lambda_\ell = k$ and $\lambda_2=j$, then the claim shows that also $\rho[n+r]$,
of type $(\lambda_1+r, \lambda_2, \ldots, \lambda_\ell)$, occurs in $\Inv_{n+r}(V)/\SC_{n+r}(V)$, and similarly for $\Quant_n(V)/\SC_n(V)$.  

By Theorem \ref{scinvisotthm}, it remains only to show that, for any
irreducible representation $\rho[n] \in \Rep S_{n+1}$ of $\hh^{\otimes
  k}$ with $k > \dim V + 1$, i.e., obtained from some irreducible
representation $\rho \in \Rep S_k$, then $\rho[n+r]$ must occur in
$\Quant_{n+r}(V)/\SC_{n+r}(V)$ for some $r$.  To do this, we make use
of the sign representation in $\Quant_{\dim V + 2}(V)$, which does not
occur in $\SC_{\dim V + 2}(V)$ (see the comment after Example
\ref{modmatexam}), together with the construction of Example
\ref{producehhkexam}, which produces from an element of $\Quant_n(V)$
not in $\SC_n(V)$, a copy of $\hh^{\otimes n}$ in the quotient
$\Quant_N(V) / \SC_N(V)$ for sufficiently large $N$ (in fact, for $N
\geq {n+1 \choose 2}$).  In particular, $\hh^{\otimes \dim V + 2}$ occurs
in the quotient for $N = {\dim V + 3 \choose 2}$.  Now, to extend
this to $\hh^{\otimes k}$ for $k > \dim V + 2$, we use composition with
a bracket: given
 an operator $\phi \in \Quant_n(V)$, we consider
\begin{equation}\label{compbr}
  \Phi(f_1 \otimes \cdots \otimes f_n \otimes f_{n+1}) = \phi(f_1 \otimes \cdots \otimes \{f_n, f_{n+1}\}) \in \Quant_{n+1}(V).
\end{equation}
If $\phi$ does not lie in $\SC_n(V)$, neither does $\Phi$. Now, apply
this to the sign representation $r$ times to obtain an element of
$\Quant_{\dim V + 2 + r}(V)$ not in $\SC_{\dim V + 2 + r}(V)$; by the
above construction, we then obtain a copy of $\hh^{\otimes \dim V + 2
  + r}$ in $\Quant_{\dim V + 3 + r \choose 2}(V)$, not coming from
$\SC(V)$.  This completes the proof.
\end{proof}

\begin{proof}[Proof of Corollary \ref{scquantinvordcor}]
  For $2m=\dim V + 6$, \eqref{scordineq} and \eqref{quantordineq}
  follow from the copy of the sign representation in $\Quant_{\dim V +
    2}(V)_{\dim V + 6}$ not in $\SC_{\dim V + 2}(V)_{\dim V + 6}$ (see
  the comment after Example \ref{modmatexam}) and Example
  \ref{modmatexam}, respectively.  In general, one applies the
  construction \eqref{compbr} as in the preceding proof. For the final
  statement, note that the sign representation in $\Quant_n(V)$ and
  the copy of the sign representation of Example \ref{matgenexam} in
  $\Inv(V)/\Quant(V)$ have order which grows on the order of
  $n^{\frac{\dim V + 1}{\dim V}}$, so taking the reciprocal of this
  exponent, we recover the statement.
\end{proof}

Corollary \ref{scquantinvdegcor} follows immediately, as indicated.

\section{Proof of Theorem \ref{snasympthm1} and Corollaries
\ref{snasympcor} and
  Corollary \ref{youngcor}}\label{s:asympsecpf}
\begin{proof}[Proof of Theorem \ref{snasympthm1}]
  (i) We may assume that $m > 0$; otherwise the statement is
  trivial. If we only require $\phi$ to be a polydifferential operator,
  and not necessarily an invariant one, then the statement is
  obvious. To proceed, we refine this description: we can write an
  element of $\Inv_n(V)_{2m}$ uniquely as a sum of the form
\begin{equation}\label{invdecomp}
\psi_{2m} + \psi_{2m-1} + \cdots + \psi_1,
\end{equation}
where each $\psi_k$ is a linear combination of permutations of operators
of the form
\begin{equation}
  f_1 \otimes \cdots \otimes f_n \mapsto \phi_k(f_1 \otimes \cdots \otimes f_k) f_{k+1} \cdots f_n,
\end{equation}
where we additionally require that $\phi_k(f_1 \otimes \cdots \otimes
f_k) = 0$ whenever at least one of the $f_i = 1$ for $1 \leq i \leq
k$. That is, $\phi_k$ is a sum of tensor products of partial
derivatives of strictly positive order in each of its $k$ components.

Now, if we apply any symplectic automorphism of $V$, the decomposition
\eqref{invdecomp} does not change, so in particular each $\psi_k \in
\Inv_n(V)_{2m}$. Moreover, there is a unique expression of each
$\psi_k$ as a sum of the form
\begin{equation}
  \psi_k(f_1 \otimes \cdots \otimes f_n) = \sum_{\sigma \in S_n / (S_k \times S_{n-k})} \phi_{k,\sigma}(f_{\sigma(1)} \otimes \cdots \otimes f_{\sigma(k)}) f_{\sigma(k+1)} \cdots f_{\sigma(n)},
\end{equation}
and we can deduce that $\phi_{k,\sigma} \in \Inv_k(V)_{2m}$ as well.
Using the $\phi_{k,\sigma}$, we obtain in particular a linear
combination of permutations of elements of the desired form
\eqref{snasympthm1eq1}.  We remark that the above argument applies
equally well to $\Quant$ and $\SC$, since to be in these subspaces, it
is clear that the $\phi_{k,\sigma}$ must also be in these subspaces.

(i') This follows from the proof of (i), discarding the assumption
that $n \geq 2m$.  Namely, let $\Inv_k(V)_{2m}'$ denote the span of
elements of the form $\phi_k$ as above, i.e., such that $\phi_k(f_1
\otimes \cdots \otimes f_k) = 0$ whenever at least one of the $f_i =
1$ for $1 \leq i \leq k$. Then, we deduce (for $m > 0$) that
$\Inv_n(V)_{2m} \cong \bigoplus_{k=1}^{\min(2m,n)} \Ind_{S_k \times
  S_{n-k}}^{S_n} \Inv_k(V)_{2m}'$ as a representation of $S_n$.  This
evidently has dimension $\sum_{k=1}^{\min(2m,n)} {n \choose k} \dim
\Inv_k(V)_{2m}'$.  Since ${n \choose k} = 0$ for $0 < n < k$, the
result is a polynomial of degree $\leq 2m$ which is valid for all
$n$. Moreover, it has degree exactly $2m$ if and only if
$\Inv_{2m}(V)_{2m}' \neq 0$.  This is true, though, since this space
includes the operator
\begin{equation}
  f_1 \otimes \cdots \otimes f_{2m} \mapsto \{f_1, f_2\} \{f_3, f_4\} \cdots \{f_{2m-1}, f_{2m}\},
\end{equation}
which also lies in $\SC$ and $\Quant$, proving that the latter two are
also of dimension a polynomial in $n$ of degree $2m$.

(ii) As we pointed out in the proof of Theorem \ref{scinvisotthm}, all
elements of $\Inv_n(V)_{2m}$ are linear combinations of permutations
of compositions of elements $\pi^{i,j}$.  To such a composition, we
can consider the graph whose vertices are those indices that appear in
the composition, and where we put an edge from $i$ to $j$ for every
element of the form $\pi^{i,j}$ that appears in the composition.  The
$\CC[S_n]$-span of such a composition has dimension which is
polynomial in $n$ of degree equal to the number of vertices of this
graph. Next, note that any graph corresponds to a term of the form
\eqref{snasympthm1eq3} exactly in the case that there is an edge
between two vertices of valence one.  

Suppose that the graph has $k$ connected components consisting of two vertices of valence one.  Suppose that there are $x$ other vertices of valence one (aside from these), and that there are $y_i$ vertices of valence $i$ for all $i \geq 2$.  Since there are $m$ edges, this implies that 
\[
2m = 2k+x+\sum_{i \geq 2} iy_i.
\]
Then,
\[
\ell = \sum_{i \geq 2} (i-1)y_i.
\]
Substituting, we obtain
\[
2m = 2k + x + \ell + \sum_{i \geq 2} y_i.
\]
Next, by assumption, $x \leq 2\sum_{i \geq 2} y_i$.  Thus, 
\begin{equation}\label{e:mineq}
2m \leq 2k + \ell + 3 \sum_{i \geq 2} y_i \leq 2k+4\ell.
\end{equation}
Thus, $2m \leq 2k+4\ell$.  Therefore if $2m > 4\ell$, i.e., $m > 2\ell$, it follows that $k > 0$, i.e., there is at least one connected component consisting of a single edge. This proves the first statement.

To prove the last statement, our goal is to show that, when $m >
(3d+1)\ell$ and $\dim V = 2d$, and for sufficiently large $n$,
$\Inv_n(V)_{2m}$ is not merely spanned by elements of the form
\eqref{snasympthm1eq3}, but that, in fact, each irreducible
$S_{n}$-subrepresentation is spanned by the Young symmetrization of
such an element of a particular form: the rightmost column of the
truncated Young tableau (removing the top row) has an even number,
$2j$, of cells, labeled by components $n-2j+1, n-2j+2, \ldots, n$, and
the element is a linear combination of compositions of $\pi^{k,\ell}$
corresponding to graphs in which vertices $n-2j+1, \ldots, n$ all
have valence one and $n-2i+1$ is adjacent to $n-2i+2$ for all $1 \leq
i \leq j$. Here we are considering graphs with $n$ vertices, and
allowing some vertices to have valence zero (meaning that that index
does not appear in the corresponding composition of elements
$\pi^{k,\ell}$ over edges of the graph). It then follows that, if we
divide by $\pi^{n-2j+1, n-2j+2} \cdots \pi^{n-1,n}$, the resulting element
must also be invariant, as desired.

To show this, let us consider more carefully the preceding argument
analyzing the structure of the graph. Let $k, x$, and $y_i$ be as
defined above.  The columns containing a cell labeled 
by one of the $x$ vertices of valence one
that are not in connected components of size two can accommodate at most $\dim V - 1$ cells labeled by vertices among the $2k$ in connected components with at most one edge.  The columns containing a cell labeled by one of the $y_i$ vertices of valence $i$ can contain at most $\dim V$ cells labeled by vertices among the aforementioned $2k$.  Hence, the number of vertices among these $2k$ which do not appear in a column on their own is at most
\[
(\dim V - 1)x + \dim V \sum_{i \geq 2} y_i.
\]
Our goal is thus to prove that, when $m \geq (3d+1)\ell$, then 
\begin{equation}\label{lyoungpartfla}
2k > (\dim V - 1)x + \dim V \sum_{i \geq 2} y_i.
\end{equation}
Recall now that $x = 2m-2k-\sum_i iy_i = 2m-2k-\ell-\sum_i y_i$.  Also let us substitute
$\dim V = 2d$. So the above becomes
\[
2k > (2d-1)(2m - 2k-\ell) + \sum_i y_i,
\]
which we can rewrite as
\[
4dk > (2d-1)(2m-\ell) + \sum_i y_i.
\]
By \eqref{e:mineq}, $2m \leq 2k + \ell + 3 \sum_{i \geq 2} y_i$, so the
above would be guaranteed if
\[
2d(2m-\ell-3\sum_{i \geq 2} y_i) > (2d-1)(2m-\ell) + \sum_i y_i.
\]
Simplifying this, we obtain
\[
2m > \ell + (6d+1)\sum_i y_i.
\]
Finally, recall that $\ell \geq \sum_i y_i$.  So the above would be
guaranteed if $2m > (6d+2)\ell$, i.e., if $m > (3d+1)\ell$, as desired.

Visually, the inequality $m > (3d+1)\ell$ can be viewed by
constructing a Young tableau which maximizes the size of the part
whose columns contain cells corresponding to vertices of the graph not
in connected components containing only a single edge (i.e., the
$x+\sum_{i \geq 2} y_i$ vertices).  Namely, this is maximized when one
has exactly $\ell$ connected components of size $3$ (one cannot have
more if the total number of vertices exceeds $2m-\ell$), each having a
single vertex of valence two and two vertices of valence one.  Then,
the part of the diagram in question has $3\ell$ cells corresponding to
the vertices in connected components of size three, and at most
$2d\ell + (2d-1)(2\ell)$ additional vertices of valence one, in
connected components containing only a single edge; this adds to a
total of $2(3d + 1)\ell$ vertices, so when $m > (3d+1) \ell$, there
are additionally columns labeled by pairs of vertices of valence one
which form connected components containing only a single edge.

Now, if we are interested in $\SC$ or $\Quant$ rather than $\Inv$, the
above arguments still apply, and show that a linear combination of
star-products or Poisson polynomials of the desired order must
decompose into Young tableaux with the rightmost columns of the
truncated Young tableaux corresponding to $\{f_{n-(2j-1)},f_{n-2j}\}
\cdots \{f_{n-1}, f_n\}$, so the remaining parts of these Young
tableaux yield the irreducible representations spanned by
$\psi$ in $\SC$ or $\Quant$.

(iii) In \eqref{lyoungpartfla}, in the case $\dim V = 2$, we can
eliminate the first term on the RHS ($(\dim V - 1)x$), for the
following reason.  If there is a column with a cell labeled by one of
the $x$ vertices and also one of the $2k$ vertices mentioned above,
then the vertex adjacent to the latter one must appear in a column
with at least two vertices among the $x + \sum_i y_i$, i.e., the
vertices not in connected components containing at most a single edge.
This allows us to subtract $\dim V$ from the RHS of
\eqref{lyoungpartfla}.  Also, we can replace the coefficient $x$ of
$\dim V - 1$ by the number of those $x$ vertices which label a column
which also has a cell labeled by one of the aforementioned $2k$
vertices.  Put together, the RHS of \eqref{lyoungpartfla} is dominated
by $\dim V \sum_{i \geq 2} y_i$ (always assuming $\dim V = 2$).

Therefore, in this case, in order to guarantee that the last column will
be labeled only by dominoes of the $2k$ vertices, it suffices to have
\[
2k > 2\sum_{i \geq 2} y_i.
\]
Substituting again the expression from \eqref{e:mineq}, $2k \geq 2m - \ell - 3 \sum_{i \geq 2} y_i$,  it suffices to have
\[
2m > \ell + 5 \sum_{i \geq 2} y_i.
\]
Since $\ell \geq \sum_{i \geq 2} y_i$, it suffices to have $m > 3\ell$, as desired.

For the final statement, we note that, in the case
$\dim V = 2$, no two vertical dominoes can appear in the same column,
and so the truncated Young diagrams we obtain for $m > 3\ell$ are
obtained from a finite list of truncated Young diagrams by combining
with a rectangle with exactly two rows.  For $m > 3\ell$,
the result is a finite collection of truncated Young diagrams, which
as a function of $m$ only vary by adding an equal number of cells to
the top two rows. If we compute the dimension of the irreducible
representations of $S_n$ corresponding to such a family, the result
has the given form, since the top two rows of the truncated Young
diagram have lengths $m-a', m-b'$, for some constant integers $a',b' \geq 0$, 
and the remaining rows have constant
length.  Applying the hook length formula for dimensions of irreducible
representations of $S_n$ yields the desired result.
As before, these arguments also apply to $\SC$ and $\Quant$.
\end{proof}

\begin{proof}[Proof of Corollary \ref{snasympcor}]
  (i) As a first approximation, note that $\Inv_n(V)_{2m} \subseteq
  \Sym^{2m}(V^n)^{\sp(V)}$, simply because all invariant operators are
  $\sp(V)$-invariant.  Next, as pointed out in the proof of Theorem
  \ref{scinvisotthm}, the subset of operators which are invariant are
  those that are invariant under a single Hamiltonian vector field on
  $V$, corresponding to a nonzero cubic Hamiltonian.  Applying this
  Hamiltonian vector field to $\Sym^{2m}(V^n)^{\sp(V)}$ lands in the
  $\Sym^3 V$-isotypic part of $\Sym^{2m-1}(V^n)$ as an
  $\sp(V)$-representation. Moreover, it is straightforward to show that this map
  is surjective up to subrepresentations of dimension a polynomial of
  degree $< 2m-1$: this amounts to showing that $\ad(x_1^2
  \frac{\partial}{\partial y_1}) (\Sym^{2m}(V^n)^{\sp(V)})$ includes
  the entire subspace of the $\Sym^3 V$-isotypic component of
  $\Sym^{2m-1}(V^n)$ spanned by polydifferential operators of order $\leq
  1$ in each of the components. In more detail, this is a consequence of
 the computation \eqref{2mm1comps3}.  

 This explains the first two terms on the RHS of
 \eqref{invasympeqn}. In more detail, let $X_{\Sym^3 V}$ denote the
 $\Sym^3 V$-isotypic part of $X$ as a $\sp(V)$-representation.  We
 deduce that $\Inv_n(V)_{2m} \cong \Sym^{2m}(V^n)^{\sp(V)} \ominus
 \Sym^{2m-1}(V^n)_{\Sym^3 V} \oplus K$, where $K$ is the cokernel of
 the map $\pr_{\Sym^3 V} \circ \ad(x_1^2 \frac{\partial}{\partial
   y_1}): \Sym^{2m}(V^n)^{\sp(V)} \rightarrow \Sym^{2m-1}(V^n)_{\Sym^3
   V}$, where $\pr_{\Sym^3 V}$ denotes the projection to the $\Sym^3
 V$-isotypic component.  We know that $\dim K < c \cdot n^{2m-2}$ for some $c > 0$.  It remains only to identify the third term in
 \eqref{invasympeqn} with a part of this cokernel $K$ (since we required that $R_{m,d}$ be an honest, rather than a virtual, representation of $S_{n+1}$).

We claim that
\begin{equation}
\bigl( \ad(y_1^2 \frac{\partial}{\partial x_1}) \ad(x_1^2
  \frac{\partial}{\partial y_1}) \Sym^{2m}(V^n)^{\sp(V)}\bigr)^{\sp(V)} = 0.
\end{equation}
This follows because $(\Sym^3 V \otimes \Sym^3 V)^{\sp(V)} \subseteq \wedge^2 (\Sym^3 V)$, and so it suffices to consider
\begin{equation}
\bigl( \ad([y_1^2 \frac{\partial}{\partial x_1}, x_1^2
  \frac{\partial}{\partial y_1}]) \Sym^{2m}(V^n)^{\sp(V)}\bigr)^{\sp(V)},
\end{equation}
but $[y_1^2 \frac{\partial}{\partial x_1}, x_1^2
\frac{\partial}{\partial y_1}]$ is the Hamiltonian vector field of a
quartic Hamiltonian, and hence its application to $\sp(V)$-invariants
is isotypic of type $\Sym^4 V$ as a $\sp(V)$-representation. Hence,
the image $\ad(x_1^2 \frac{\partial}{\partial y_1})
\Sym^{2m}(V^n)^{\sp(V)}$ lies in the kernel of $\pr \circ \ad(y_1^2
\frac{\partial}{\partial x_1})$ on the $\Sym^3 V$-isotypic part of
$\Sym^{2m-1}(V^n)$, where $\pr$ is the projection to
$\sp(V)$-invariants.

We claim that this latter map is surjective onto
$\Sym^{2m-2}(V^n)^{\sp(V)}$.  To see this, one can extend the map to
all of $\Sym^{2m-1}(V^n)$ and consider the image, since the map sends
an irreducible representation $Y$ of $\sp(V)$ to one isomorphic to a
subrepresentation of $Y \otimes \Sym^3 V$, which can only contain
invariants when $Y \cong \Sym^3 V$. Now, it is enough to show that the
dual map $(\Sym^{2m-2}(V^n)^{\sp(V)})^* \rightarrow
\Sym^{2m-1}(V^n)^*$ is injective. It suffices to show that the map
$\{y_1^3, -\}$ is injective acting on $(\cO_{V^n})_{\geq 1}^{\sp(V)}$. 
Since $y_1^3$ and $\sp(V) \cong \Sym^2 (V^*)$ generate the
Poisson algebra $\CC \oplus (\cO_V)_{\geq 2}$, it is enough to show
that no element of $(\cO_{V^n})_{\geq 1}^{\sp(V)}$ is killed by the adjoint
action of all of $(\cO_V)_{\geq 2}$.  Indeed, no element of $(\cO_{V^n})_{\geq 1}$ is killed by the adjoint action of all of $(\cO_V)_{\geq 2}$:
it is enough to consider the adjoint action of elements $w^N$ for
$w \in V^*$, and $N \gg 0$.

Hence, the cokernel $K$ of $\pr_{\Sym^3 V} \circ \ad(x_1^2
\frac{\partial}{\partial y_1})$ acting on $\Sym^{2m}(V^n)^{\sp(V)}$
surjects to $\Sym^{2m-2}(V^n)^{\sp(V)}$ via the map $\pr \circ
\ad(y_1^2 \frac{\partial}{\partial x_1})$.  This proves the desired
statement.

(ii) The statement that $R_{m,d}$ has dimension bounded by a
polynomial of degree only $2m-3$ in $n$ is equivalent to the statement
that the kernel of the surjection $K \onto \Sym^{2m-2}(V)^{\sp(V)}$ in
the proof of the previous part has dimension bounded by a polynomial
of degree $2m-3$ in $n$.  This isn't true for general $V$, but it is
true for $V = \CC^2$. This can be verified by an explicit computation
along the lines of the proof of Theorem \ref{scinvordthm}. In more detail, letting $(x,y)$ be a symplectic basis of $V = \CC^2$, 
one needs to show dually that $\pr \circ \ad(x^3)$ is injective on the part of
\begin{equation}
((\cO_{\CC^{2n}})_{2m-1})_{\Sym^3 \CC^2}/ \{y^3, \cO_{\CC^{2n}}^{\sp(V)}\}
\end{equation}
spanned by monomials in which exactly $2m-2$ distinct components of
$\CC^{2n} = (\CC^2)^{\oplus n}$ (out of $n$ total) appear.  This
decomposes into a few different $S_n$-representations, as follows.
Since we are considering monomials in which, in one component of
$(\CC^2)^{\oplus n}$, one has a quadratic function of $\cO_{\CC^2}$,
and in $2m-3$ other components one has a linear function, we obtain
the $\sp(\CC^2)=\mathfrak{sl}_2$-representation $\Sym^2 \CC^2 \otimes
(\CC^2)^{\otimes (2m-3)}$ (here and for the rest of the proof, we
identify $\CC^2 \cong (\CC^2)^*$ using the symplectic form, in order
to have cleaner notation). We are interested in the $\Sym^3
\CC^2$-isotypic part, and this means that we can restrict to the
$\CC^2, \Sym^3 \CC^2$, and $\Sym^5 \CC^2$-isotypic parts of
$(\CC^2)^{\otimes (2m-3)}$.  Hence, we can restrict our attention to
elements of $(\CC^2)^{\otimes 2m-3}$ lying in an
$S_{2m-3}$-representation whose Young diagram is of the form $(2m-1,
2m-2)$, $(2m,2m-3)$, or $(2m+1,2m-4)$.  We must combine this Young
diagram with a single cell for the quadratic part $\Sym^2(\CC^2)$, and
with a horizontal line of length $n-(2m-2)$ for all the components
where no differentiation happens.  One then verifies that the kernel
of $\pr \circ \ad(x^3)$ lies only in the part where the Young diagram
for $\Sym^2 (\CC^2) \otimes (\CC^2)^{\otimes (2m-3)}$ with $2m-2$
cells, is a rectangle with two rows.  This, on the other hand, is
exactly the image of $\ad(y^3)$ on $((\CC^2)^{\otimes
  (2m-2)})^{\sp(V)}$.  This proves the needed assertion. The rest of
the statement of (ii) is an explicit calculation.
\end{proof}

\begin{proof}[Proof of Corollary \ref{youngcor}]
  (i)   It suffices to consider those irreducible $S_{n+1}$-subrepresentations
  $\rho$ of $\Inv_n(V)_{2m}$ such that, if the Young diagram of
  $\rho$ is $\lambda = (\lambda_1, \ldots, \lambda_r)$, then $\rho$ is
  not obtainable from a subrepresentation $\rho' \subseteq
  \Inv_{n-1}(V)_{2m}$ with Young diagram
  $(\lambda_1-1,\lambda_2,\ldots,\lambda_r)$ by taking the unique
  summand of $\Ind_{S_n \times S_1}^{S_{n+1}} \rho'$ isomorphic to
  $\rho$.  Namely, we need to show that, for such $\rho$, $n \leq
  3m-1$.  

  For the argument, let us fix an $S_n$-subrepresentation $\rho_0
  \subseteq \rho$, with Young diagram $\mu$, which is obtained from
  $\lambda$ by reducing a single row in length by one cell, i.e.,
  $\mu_j = \lambda_j$ for all $j$ except one index $i$, where $\mu_i =
  \lambda_i - 1$.  We will use the general fact that $\rho_0$ is
  spanned over $S_n$ by a polydifferential operator on $r$ functions,
  of degrees $\mu_1, \ldots, \mu_r$, and hence $\rho$ is so spanned
  over $S_{n+1}$. Moreover, such differential operators must be the
  Young symmetrization of a linear combination of tensor products $D_1
  \otimes \cdots \otimes D_n$, where the $D_i$ are monomials in the
  partial derivatives, and the Young symmetrization is according to a
  labeling of $\mu$.

  Note that $\rho$ is spanned by the Young symmetrizations according
  to $\mu$ of tensor products of monomials in the partial derivatives,
  and each such tensor product can have at most $\mu_1 \leq \lambda_1$
  components where one places order-zero operators (i.e., a
  scalar, say $1$). Moreover, the cell that one adds to $\mu$ to
  obtain $\lambda$ cannot lie in the same column as one of these at
  most $\mu_1$ cells.  Hence, we deduce that $\rho$ is spanned by
  operators which act with positive order in at least $n+1-\lambda_1$
  components.  By the argument of Theorem \ref{snasympthm1}.(i), such
  operators can be taken to be invariant, and to lie in a
  subrepresentation $\rho_i \in \Rep S_{k_i}$ appearing in Theorem
  \ref{snasympthm1}.(i').  We deduce that $n+1 \leq \lambda_1 +
  k_i$. Hence, it suffices to show
\begin{equation}\label{leq3mineq1}
k_i + \lambda_1 \leq 3m.
\end{equation}
Now, we return to the specific form of $\rho$ that we seek. For
$\rho'$ not to exist, there are two cases. The first case is
$\lambda_1 = \lambda_2$, i.e., $\rho$ is the first possible element in
a sequence of representations of $S_{n+1}$ in the sense of Deligne, in
which case $\rho'$ cannot exist by definition.  In this case,
\eqref{leq3mineq1} becomes
\begin{equation}\label{leq3mineq2}
k_i + \lambda_1^i \leq 3m,
\end{equation}
where $\lambda^i = (\lambda_1^i, \ldots, \lambda_{r_i}^i)$ is the Young
diagram associated to $\rho_i$.

To prove \eqref{leq3mineq2}, we use the argument of the proof of
Theorem \ref{snasympthm1}, using that $\Inv_n(V)_{2m}$ is spanned by
compositions of elements $\pi^{j,k}$, and each such composition can be
represented as a graph whose vertices are labeled $1, 2, \ldots, n$,
with an edge between $j$ and $k$ for every term $\pi^{j,k}$ in the
composition.  This implies that the Young diagram $\lambda^i$ must be
obtainable from combining Young diagrams corresponding to each
connected component of the graph, each of which have the property that
the given Young symmetrization of the corresponding composition of
$\pi^{j,k}$ terms is nonzero.

For a connected component with exactly two vertices, the only Young
symmetrization over $S_2$ of $\pi^{j,k}$ which is nonzero is the
skew-symmetrization, so the size-two Young diagrams in the
aforementioned combination must be vertical dominoes. These have the
property that
\begin{equation}\label{smallleq3mineq}
\text{\# vertices} + \text{length of top row} \leq 3 (\text{\# edges}),
\end{equation}
which is a translation of \eqref{leq3mineq2} to this component Young
diagram.  (Note here that the number of vertices is the same as the
total number of cells of this component Young diagram.)

To finish the proof of \eqref{leq3mineq2}, it suffices to show that
\eqref{smallleq3mineq} is also satisfied for all possible Young
diagrams corresponding to connected components of size $\geq 3$.  The
general inequality $\text{\# vertices} \leq \text{\# edges} + 1$ for a
connected graph implies that $2 (\text{\# vertices}) \leq 2
(\text{\#edges}) + 2$, and in the case that $\text{\# vertices} \geq 3$, we
obtain also that $2 \leq \text{\#edges}$, and substituting in, we see that
$2 (\text{\# vertices}) \leq 3 (\text{\# edges})$, which implies
\eqref{smallleq3mineq}.

Finally, we return to the second case of a $\rho$ not obtainable from
$\rho'$, namely, $\lambda_1 > \lambda_2$. First, by Theorem
\ref{isotpartthm}.(i), it suffices to assume that $\lambda$ has at
least two rows ($r \geq 2$), and in this case we can also assume that
$\mu_1 = \lambda_1$.  Next, we can assume that $\rho$ is spanned by
Young symmetrizations of differential operators which act with
positive degree in the cells of the first row: according to the
decomposition of the proof of Theorem \ref{snasympthm1}.(i), we can
always decompose into such a part and a part that is obtained by
inducing from a representation of the form $\rho'$.  Recall that
$\rho$ was spanned by a polydifferential operator of $r$ functions
with degrees $\mu_1 \geq \cdots \geq \mu_r$.  By the Darboux theorem,
such an operator is determined by its restriction to setting the
first of the $r$ functions to be $x_1 \in V^* \subseteq \cO_V$. Let
$\phi$ be the polydifferential operator of the remaining $r-1$
functions. The order of $\phi$ must be at least $\mu_1$, for the
result to be invariant under the operation $x_1 \mapsto c x_1, y_1
\mapsto c^{-1} y_1$.  Furthermore, $\phi$ can act with order zero in
at most $\mu_2 \leq m-1$ components, in order for its Young
symmetrization not to vanish.  Thus, we conclude that $\lambda$ can
have at most $\mu_1 + (m-1) + (2m-\mu_1) = 3m-1$ cells. That is, $n+1
\leq 3m-1$, and hence $n \leq 3m-2$, which yields the desired result.
\footnote{Note that, if we were a bit more careful, we would notice
  that in the case $\dim V = 2$, $\phi$ can also act in at most $m-1$
  components by an order one operation (since it can be assumed to be
  $\frac{\partial}{\partial y_1}$), and hence must act with order
  $\geq 2$ in at least one component. Hence, when $\dim V = 2$, we
  would in fact conclude that $n \leq 3m-3$.  The upshot of this is
  that, for $n \geq 3m-2$, or $n \geq 3m-3$ and $V =\CC^2$, the
  irreducible $S_{n+1}$-representations in $\Inv_n(V)_{2m}$ are
  exactly those obtainable, with multiplicity, from the same finite
  list of truncated Young diagrams; and all of these truncated Young
  diagrams produce a full Young diagram at or before $n=3m-1$.}

(ii) In the proof of Theorem \ref{snasympthm1}.(ii), we saw that, under
the given assumptions, the Young diagrams corresponding to irreducible
$S_{n}$-subrepresentations (or, $S_{n+1}$-subrepresentations) of
$\Inv_n(V)_{2m}$ with at least $2m-\ell$ cells below the top row, and
for sufficiently large $n$, are obtainable from a finite list of
diagrams by combining with vertical dominoes of size two in the
described manner.  To conclude, it remains only to say that it is
enough to have $n \geq 3m$.  But, this is an immediate consequence of
(i).

(iii) This similarly follows from part (i) together with Theorem
\ref{snasympthm1}.(iii).
\end{proof}

\appendix
\section{Combinatorial proof and generalization of Corollary
  \ref{combcor}}\label{combapp}
For every Young diagram $\lambda = (\lambda_1,
\ldots, \lambda_r)$, and all $n \geq |\lambda| + \lambda_1$, let
$\lambda[n] := (n-|\lambda|,\lambda_1, \ldots, \lambda_r)$ be the given
partition of $n$.  Also, let $\rho_\lambda$ denote the irreducible
representation of $S_{|\lambda|}$ associated to $\lambda$.

Fix an integer $c$.  Let $\chi_c$ be the character of $\ZZ/(n+1)$ given by $1 \mapsto \exp(\frac{2c\pi i}{n+1})$. We are interested in the series
\begin{equation}
L_{\lambda,c}^+(t) := \sum_{n+1 \geq |\lambda|+\lambda_1}  \dim \Hom_{S_{n+1}}(\rho_{\lambda}[n+1], \Ind_{\ZZ/(n+1)}^{S_{n+1}} \chi_c) t^n.
\end{equation}
For any Young diagram $\lambda$, let $h(\lambda) = (h_1(\lambda), \ldots, h_{|\lambda|}(\lambda))$ be the $|\lambda|$-tuple of hook lengths of $\lambda$ (for any fixed Young tableau with shape $\lambda$).
\begin{claim}\label{combclaim} Let $|\lambda| \geq 1$.
The series $L_{\lambda,c}^+(t)$ is a rational function of $t$ of the form
\begin{equation} \label{clfla}
\frac{t^{|\lambda|+\lambda_1-1} K_{\lambda,c}^+(t)}{\prod_{i=1}^{|\lambda|} (1-t^{h_i(\lambda)})},
\end{equation}
where $K_{\lambda,c}^+(t)$ is a polynomial in $t$ of degree less than $\sum_{i=1}^{|\lambda|} h_i(\lambda)$. Moreover, $K_{\lambda,c}^+(1) = (|\lambda|-1)!$.
\end{claim}
The case $c=0$ of the claim is Corollary \ref{combcor}, and the case $c=1$ 
is \eqref{combliefla1}.
\begin{proof}\footnote{Thanks to R. Stanley for pointing out \eqref{staneq}--\eqref{zfla}, as well as
the Kraskiewicz-Weyman theorem.}
  It is a theorem of Kraskiewicz-Weyman that the multiplicity 
\begin{equation*}
\dim
  \Hom_{S_{n+1}}(\rho_{\lambda}[n+1], \Ind_{\ZZ/(n+1)}^{S_{n+1}}
  \chi_c)
\end{equation*} is given by the number of standard Young tableaux of shape
  $\lambda[n+1]$ whose major index is congruent to $c$ mod $n+1$.
   Such tableaux can be counted using another well known combinatorial
result (where $\operatorname{SYT}(\lambda)$ is the set of standard Young tableaux of shape $\lambda$ and $\operatorname{maj}(T)$ is the major index of such a tableau $T$),
\begin{equation}\label{staneq}
  \sum_{T \in \operatorname{SYT}(\lambda[n+1])} q^{\operatorname{maj}(T)} = q^{\sum_{i=1}^r i \lambda_i} \frac{(1-q^{n+1})(1-q^n) \cdots (1-q)}{\prod_{i=1}^{n+1} (1-q^{h_i(\lambda[n+1])})}.
\end{equation}
To simplify this, let $\lambda'$ be the dual partition to
$\lambda$ (whose $k$-th row is the $k$-th column of $\lambda$ and
vice-versa), and set $X_\lambda := [0..(\lambda_1+|\lambda|-1)] \setminus
\{|\lambda| + i - 1 - \lambda'_i: i \in [1..\lambda_1]\}$. Then, the fraction appearing above simplifies as
\begin{equation}\label{fnqfla}
F_n(q) := q^{\sum_{i=1}^r i \lambda_i} \sum_{n+1 \geq |\lambda|+\lambda_1} 
\frac{\prod_{p \in X_\lambda} (1-q^{n+1-p})}{\prod_{i=1}^{|\lambda|} (1-q^{h_i(\lambda)})}.
\end{equation}
We deduce that
\begin{equation}\label{zfla}
L_{\lambda,c}^+(t) = \sum_{n+1 \geq |\lambda|+\lambda_1} 
\frac{1}{n+1}\sum_{j=0}^n \zeta_n^{-cj} F_n(\zeta_n^j) t^n,
\end{equation}
where $\zeta_n$ is a primitive $(n+1)$-st root of unity.

We claim that
\begin{equation}\label{qtclfla}
\sum_{n+1 \geq |\lambda|+\lambda_1, k \mid (n+1)} t^{n} \sum_{j \in [1..k]: \gcd(j,k)=1} \zeta_k^{-cj}F_{n}(\zeta_k^j)
=  \frac{t^{|\lambda|+\lambda_1-1}\phi_k(t)}{(1-t^k)^{1+|\{i \in [1..|\lambda|]: k \mid h_i(\lambda)\}|}},
\end{equation}
where $\phi_k(t)$ is a polynomial in $t$ of degree less than
$k$,  independent of $n \geq |\lambda|+\lambda_1-1$. 
Moreover, we claim that $\phi_k(t) = 0$ when $|\{i \in [1..|\lambda|]:
k \mid h_i(\lambda)\}| = 0$ (and in particular when $k > |\lambda|$).
To obtain from this the desired result, \eqref{clfla}, we can add
together \eqref{qtclfla} over all $k \leq |\lambda|$, and then apply
$\frac{1}{t} \int dt$. In view of \eqref{zfla}, this implies the
formula \eqref{clfla}, and that the degree of $K_{\lambda,c}^+(t)$ is
less than the degree of the denominator. 

To prove these claims, we note two facts about the LHS of
\eqref{qtclfla}. First, the coefficients of $t^{ak+b}$, for fixed $b
\geq 0$ and as a function of $a \geq 0$ (such that $ak+b \geq
|\lambda|+\lambda_1$), are polynomial in $a$ of degree at most $|\{i
\in [1..|\lambda|]: k \mid h_i(\lambda)\}|$.  To see this, plug in $q
= \zeta_k$ and $n=ak+b$ in the RHS of \eqref{fnqfla}, and observe that
the result is a polynomial in $a$, with coefficients in
$\ZZ[\zeta_k]$, of degree equal to the number of exponents of $q$
appearing in the denominator which are multiples of $k$. This proves
\eqref{qtclfla}. To see that $\phi_k(t) = 0$ when $|\{i \in
[1..|\lambda|]: k \mid h_i(\lambda)\}| = 0$, note that the fraction on
the RHS of \eqref{fnqfla} evaluates to zero at $z = \zeta_k$ if the
numerator of has more exponents which are multiples of $k$ than the
denominator. This will in particular be the case if $k$ does not
divide any of the exponents of the denominator, but divides $n+1$.
This is the case for all contributions to the RHS of \eqref{qtclfla}
when $|\{i \in [1..|\lambda|]: k \mid h_i(\lambda)\}| = 0$, which
proves that $\phi_k(t)=0$ in this case.

Finally, to see that $K_{\lambda,c}^+(1) = (|\lambda|-1)!$, note
that only for $k=1$ does \eqref{qtclfla} yield a series in $t$ whose
coefficients grow polynomially of degree $|\lambda|-1$ (for $k > 1$,
the coefficients grow polynomially of degree less than this).  Hence,
it suffices to observe that, if we plug $q=1$ into the RHS of
\eqref{fnqfla} (or indeed into the LHS of \eqref{staneq}), we obtain the dimension of $\rho_{\lambda}[n+1]$.
\end{proof}

\section{Direct proof of Corollary \ref{hot3cor}}\label{s:hot3cor}
Here, we
explain how to give a direct proof of Corollary \ref{hot3cor} in the
spirit of the proof of Theorem \ref{isotpartthm} in \S
\ref{isotpartpfsec}, using only Theorem \ref{scinvisotthm}.

First, by Theorem \ref{scinvisotthm}, the $\hh^{\otimes 3}$-isotypic
component of $\Inv_n(V)$ coincides with that of $\Quant_n(V) \cong
\Ind_{\ZZ/(n+1)}^{S_{n+1}} \CC$.  In more detail, one has
$\Hom_{S_{n+1}}(\rho, \Ind_{\ZZ/(n+1)}^{S_{n+1}} \CC) \cong
\Hom_{\ZZ/(n+1)}(\Res^{S_{n+1}}_{\ZZ/(n+1)} \rho, \CC)$ by Frobenius
reciprocity, and the latter is equal to the multiplicity of the
eigenspace of one of the cyclic permutation $(12\cdots n)$ acting on
$\rho$.  Next, we can decompose the representation $\hh$ into a direct
sum of its eigenspaces under this cyclic permutation: the eigenvalues
are all the $n$-th roots of unity except for $1$ itself, occurring
with multiplicity one.  

Using this, it is not difficult to verify
the formulas
\begin{gather}
\dim \Hom_{S_{n+1}}(\hh^{\otimes 3}, \Ind_{\ZZ/(n+1)}^{S_{n+1}} \CC) = n(n-1), \\
\dim \Hom_{S_{n+1}}(\wedge^3 \hh,  \Ind_{\ZZ/(n+1)}^{S_{n+1}} \CC) = \frac{1}{3} ({n \choose 2} - \lfloor \frac{n}{2} \rfloor  - n + 2\delta_{3 \mid n+1} + \delta_{2 \mid n+1}), \\
\dim \Hom_{S_{n+1}}(\Sym^3 \hh, \Ind_{\ZZ/(n+1)}^{S_{n+1}} \CC) = \dim \Hom_{S_{n+1}}(\wedge^3 \hh,  \Ind_{\ZZ/(n+1)}^{S_{n+1}} \CC) + n - \delta_{2 \mid n+1}.
\end{gather}
Next, we use 
the decompositions (where $\rho_\lambda$ denotes the irreducible
representation with truncated Young diagram $\lambda$):
\begin{equation}\label{h3dec}
\h^{\otimes 3} = \rho_{(3)} \oplus \rho_{(1,1,1)} \oplus 2\rho_{(2,1)} \oplus 3\rho_{(1,1)} \oplus 3\rho_{(2)} \oplus 4 \rho_{(1)} \oplus \CC, \quad
\Sym^3 \h = \rho_{(3)} \oplus \h^{\otimes 2},
\end{equation}
and also $\rho_{(1,1,1)} = \wedge^3 \h$, $\rho_{(1,1)} = \wedge^2 \h$, $\rho_{(1)} = \h$, and $\Sym^2 \h = \rho_{(2)} \oplus 2 \rho_{(1)} \oplus \CC$.

As in the proof of Theorem \ref{isotpartthm}.(iii), operators of
$\Inv_n(V)_{2m}$ spanning a subrepresentation of $\h^{\otimes 3}$ can
be viewed as polydifferential operators of three functions, $f, g$,
and $h$, which are linear in $g$ and $h$, and we can restrict to the
case where $f = x_1$.  The resulting differential operators $g \otimes
h \mapsto D(g,h)$ must be of the form
\begin{equation}
D(g,h) = \sum_{i=0}^m \lambda_{i} 
\frac{\partial^i g}{\partial x_1^i} \frac{\partial^{m-i} h}{\partial x_1^i}
+ \sum_{i=0}^{m-1} \mu_i \{ \frac{\partial^i g}{\partial x_1^i}, \frac{\partial^{m-i-1} h}{\partial x_1^{m-1-i}} \}.
\end{equation}
Conversely, any isotypic component of height $\leq 3$ must be spanned
by such an operator.  Let $X_m$ be the subspace spanned by such $D$
(which are obtainable from $\Inv_n(V)$), which is an $S_2$
representation by permuting $g$ and $h$, and moreover extends to an
$S_3$-representation if we view $D$ as a distribution on three
functions. Moreover, as vector spaces (or $S_2$-representations),
$X_m = Y_m \oplus Z_m$, where $Y_m$ is the span
of terms with $\mu_i=0$ for all $i$, and $Z_m$ is the span of terms
with $\lambda_i =0$ for all $i$.

Each element of $Y_m$ extends to Poisson
polynomials of degree $n \geq m+2$, and each element of $Z_m$
extends
to Poisson polynomials of degree $n \geq m+1$.  As a result, $Y_m \neq 0$
whenever $n \geq m+2$, and in this case has dimension $m+1$, whereas
$Z_m \neq 0$ whenever $n \geq m+1$, and in this case has dimension $m$.
The
total dimension of $X_m$ is
$2m+1$ when $n \geq m+2$, and $m$ when $n \geq m+1$.  

Since homomorphisms $\h^{\otimes 3} \rightarrow \Inv_n(V)_{2m}$ are
uniquely determined by the operator $D$ determined by restricting the
first input to be $x_1$,  we deduce that, for $n \geq 2$,
\begin{equation} 
\dim \Hom_{S_{n+1}}(\hh^{\otimes 3}, \Inv_n(V)_{2m}) = \begin{cases}
2m+1, & \text{if $m \leq n-2$}, \\  m, & \text{if $m = n-1$}, \\ 0, & \text{if $m \geq n$}.\end{cases}
\end{equation}
Using Theorem \ref{isotpartthm}.(i)--(iv) and the decomposition \eqref{h3dec},
we deduce that, for $n \geq 4$,
\begin{equation}
M_{(3)} + 2 M_{(2,1)} + M_{(1,1,1)} = \begin{cases} 2m-2, & \text{if $m \leq n-2$}, \\ m-3, & \text{if $m = n-1$}, \\ 0, & \text{if $m \geq n$}.\end{cases}
\end{equation}
To do better, we can analyze the skew-symmetric and symmetric
parts under $S_2$ (i.e., in $g$ and $h$) of $Y_m$ and $Z_m$.
 By doing so, we can deduce
\begin{equation}
M_{(1,1,1)} - M_{(3)} = \begin{cases} 0, & \text{if $n \geq m+2$ or $m$ is odd}, \\ 1, & \text{if $n=m+1$ and $m$ is even.} \end{cases}
\end{equation}
This follows from the facts
 (for $m \leq n-2$):
\begin{itemize}
\item The skew part of $Y_m$
has dimension $\lfloor \frac{m+1}{2} \rfloor$ while the symmetric part has dimension $\lfloor \frac{m+2}{2} \rfloor$;
\item The skew part of $Z_m$
has dimension $\lfloor \frac{m+1}{2} \rfloor$,
while the symmetric part has dimension $\lfloor \frac{m}{2} \rfloor$.
\end{itemize}
Thus, it only remains to see how the copies of $\rho_{(2,1)}$
distribute among orders $4, \ldots, 2(n-1)$.  Moreover, in orders $4$
and $6$, the answer follows from Corollary \ref{meq23str}: here, in
both cases, $M_{(2,1)} = 1$ (in order $4$, this holds for $n \geq 4$,
and in in order $6$, this holds for $n \geq 5$).

We need one more piece of information. Consider the natural map $\Phi:
X_m \to Y_{m-1}$, sending a differential operator $D(g,h)$ given by
coefficients $\lambda_i, \mu_j$ to the operator $D'(g,h)$ of order
$m-1$ with coefficients $\lambda'_i := \mu_i$ and $\mu'_j := 0$.  This
map does not preserve the $S_3$-structure, but it does if we twist by
a sign, yielding a surjective map $X_m \otimes \mathrm{sgn} \onto
Y_{m-1}$ of $S_3$-representations.  Moreover, $\ker(\Phi_m) = Y_m$.
We deduce that, as $S_3$-representations, $X_m = Y_m \oplus (Y_{m-1}
\otimes \mathrm{sgn})$.  Now, $Y_m = 0$ for $m \geq n$, and otherwise
$Y_m$ does not depend on $n$.  Hence, $\Hom_{S_{n+1}}(\hh^{\otimes 3},
\Inv_{m+1}(V) \ominus \Inv_{m}(V)) \cong Y_{m-1} \oplus (Y_{m-1}
\otimes \mathrm{sgn})$. However, the LHS follows from part
(ii). Hence, using this, the preceding observations, and Theorem
\ref{isotpartthm}.(i)--(iv), it is straightforward to verify the
desired formulas.

\bibliographystyle{amsalpha}
\newcommand{\etalchar}[1]{$^{#1}$}
\def\cprime{$'$}
\providecommand{\bysame}{\leavevmode\hbox to3em{\hrulefill}\thinspace}
\providecommand{\MR}{\relax\ifhmode\unskip\space\fi MR }
\providecommand{\MRhref}[2]{%
  \href{http://www.ams.org/mathscinet-getitem?mr=#1}{#2}
}
\providecommand{\href}[2]{#2}

\end{document}